\documentclass[a4paper,12pt,reqno]{amsart}

\usepackage{pgfplots}
\pgfplotsset{compat=1.18} 
\usepackage{amsmath,amssymb,amsthm}
\usepackage{latexsym}
\usepackage{color}
\usepackage{amsfonts}
\usepackage{graphicx}
\usepackage{mathrsfs}
\usepackage{enumerate}
\usepackage[abbrev]{amsrefs}
\usepackage[T1]{fontenc}
\usepackage[colorlinks,
           linkcolor=blue,      
           anchorcolor=blue,  
           citecolor=red,       
            ]{hyperref}
\usepackage{mathtools}
\mathtoolsset{showonlyrefs=true}
\usepackage{tikz}
\usetikzlibrary{arrows.meta, positioning, shapes.geometric}
\usepackage{tikz-cd}
\usetikzlibrary{patterns,arrows.meta}
\usepackage{pgfplots}
\usepackage{xcolor}
\usepackage{tikz-cd}
\usetikzlibrary{calc}

\allowdisplaybreaks[4]

\theoremstyle{plain}
\newtheorem{thm}{Theorem}[section]
\newtheorem{lemm}[thm]{Lemma}
\newtheorem{prop}[thm]{Proposition}

\theoremstyle{definition}

\newtheorem{rem}[thm]{Remark}

\makeatletter

\@addtoreset{equation}{section}
\makeatother

\renewcommand{\div}{\operatorname{div}}
\newcommand{\dB}{\dot{B}}
\newcommand{\dH}{\dot{H}}

\newcommand{\supp}{\operatorname{supp}}

\renewcommand{\leq}{\leqslant}
\renewcommand{\geq}{\geqslant}

\newcommand{\pnabla}{{\nabla}^{\perp}}

\newcommand{\n}[1]{{\left\|#1\right\|}}

\newcommand{\lp}[1]{\left[#1\right]}
\newcommand{\Mp}[1]{\left\{#1\right\}}
\renewcommand{\sp}[1]{\left(#1\right)}
\newcommand{\abso}[1]{{\left|#1\right|}}

\begin{document}
\title[Asymptotic instability for the forced Navier--Stokes equations]
{Asymptotic instability for the forced Navier--Stokes equations in critical Besov spaces}
\author[M.~Fujii]{Mikihiro Fujii}
\address[M.~Fujii]{Graduate School of Science, Nagoya City University, Nagoya, 467-8501, Japan}
\email[M.~Fujii]{fujii.mikihiro@nsc.nagoya-cu.ac.jp}
\author[H.~Tsurumi]{Hiroyuki Tsurumi}
\address[H.~Tsurumi]{Graduate School of Technology, Industrial and Social Sciences, Tokushima University, Tokushima, 770--8506, Japan}
\email[H.~Tsurumi]{tsurumi.hiroyuki@tokushima-u.ac.jp}
\begin{abstract}
The asymptotic stability is one of the classical problems in the field of mathematical analysis of fluid mechanics.
In $\mathbb{R}^n$ with $n \geq 3$, it is easily proved by the standard argument that if the given small external force decays at temporal infinity, then the small forced Navier--Stokes flow also strongly converges to zero as time tends to infinity in the framework of the critical Besov spaces $\dot{B}_{p,q}^{n/p-1}(\mathbb{R}^n)$ with $1 \leq p < n$ and $1 \leq q < \infty$. 
In the present paper, we show that this asymptotic stability {\it fails} for $p \geq n$ with $n \geq 3$
in the sense that there exist arbitrary small external forces whose critical Besov norm decays in large time, whereas the corresponding Navier--Stokes flows oscillate and do not strongly converge as $t \to \infty$ in the framework of the critical Besov spaces $\dot{B}_{p,q}^{n/p-1}(\mathbb{R}^n)$.
Moreover, we find that the situation is different in the two-dimensional case $n=2$ and show the forced Navier--Stokes flow is asymptotically unstable in $\dot{B}_{p,1}^{2/p-1}(\mathbb{R}^2)$ for all $1 \leq p \leq \infty$.
Our instability does not appear in the linear level but is caused by the nonlinear interaction from external forces.
\end{abstract}
\maketitle


\section{Introduction}\label{sec:intro}

Let us consider the incompressible Navier--Stokes equations with external forces on $\mathbb{R}^n$ with $n \geq 2$:
\begin{align}\label{eq:f-NS}
    \begin{cases}
    \partial_t u - \Delta u + \mathbb{P}\div (u \otimes u ) = \mathbb{P}f, & t>0, x \in \mathbb{R}^n, \\
    \div u = 0, & t \geq 0, x \in \mathbb{R}^n,\\
    u(0,x)=a(x), & x \in \mathbb{R}^n,
    \end{cases}
\end{align}
where $u=u(t,x):[0,\infty) \times \mathbb{R}^n \to \mathbb{R}^n$ is the unknown solenoidal velocity field of the fluid, while  $a=a(x):\mathbb{R}^n \to \mathbb{R}^n$ is the solenoidal initial datum and $f=f(t,x):(0,\infty) \times \mathbb{R}^n \to \mathbb{R}^n$ is the given external force.
We denote by $\mathbb{P}:=I+\nabla \div (-\Delta)^{-1}$ the Helmholtz projection to the divergence-free vector fields.
It is well known that \eqref{eq:f-NS} has an invariant scaling structure; if $u$ solves \eqref{eq:f-NS}, then the rescaled functions $u_{\lambda}(t,x):=\lambda u (\lambda^2 t, \lambda x)$ also satisfy \eqref{eq:f-NS} with $f$ replaced by $f_{\lambda}(t,x) := \lambda^3 f (\lambda^2 t, \lambda x)$ for all $\lambda>0$.
We say that the solution space $X=X(\mathbb{R}^n)$ and the external force space $Y=Y((0,\infty)\times \mathbb{R}^n)$ are critical if they satisfy $\n{u_{\lambda}(0,\cdot)}_X=\n{u(0,\cdot)}_X$ and $\n{f_{\lambda}}_{Y}=\n{f}_Y$ for all $\lambda>0$.
It is well-known as the Fujita--Kato principle that considering the solvability in the scaling critical framework is important.  
From a general perspective, it is expected that if $X$ and $Y=BC((0,\infty);\widetilde{Y})$ are scaling critical spaces for \eqref{eq:f-NS} and given data $a \in X$ and $f \in Y$ are sufficiently small, then \eqref{eq:f-NS} admits a unique small global solution $u \in BC([0,\infty);X)$ with the asymptotic stability, which means that if $\n{f(t)}_{\widetilde{Y}} \to 0$ as $t \to \infty$, then $\n{u(t)}_X \to 0$ as $t \to \infty$.
In this paper, we consider the critical Besov space setting $X=\dB_{p,q}^{n/p-1}(\mathbb{R}^n)$ and $\widetilde{Y}=\dB_{p,q}^{n/p-3}(\mathbb{R}^n)$ and investigate which indices $p$ and $q$ ensure or fail to ensure the asymptotic stability for \eqref{eq:f-NS}.

\subsection{The place of our study in the related literature}
In order to present the statement of our main results and the context of this study more precisely, we begin by recalling known results related to our work.
In the case of $f \equiv 0$, local well-posedness for large data and global well-posedness for small data were proven in the following critical scaling spaces 
\begin{align}
    &
    \dot{H}^{n/2-1}(\mathbb{R}^n) \hookrightarrow L^n(\mathbb{R}^n) \hookrightarrow L^{n,\infty}(\mathbb{R}^n)
    \hookrightarrow {\rm BMO}^{-1}(\mathbb{R}^n),
    \\
    &
    \dB_{p,q}^{n/p-1}(\mathbb{R}^n) \hookrightarrow {\rm BMO}^{-1}(\mathbb{R}^n) \qquad (1 \leq p < \infty, 1 \leq q \leq \infty). 
\end{align}
See Fujita--Kato \cite{Fuj-Kat-64} for $\dH^{n/2-1}(\mathbb{R}^n)$, and Kato \cite{Kat-84} and Giga--Miyakawa \cite{Gig-Miy-85} for $L^n(\mathbb{R}^n)$.
For the case of critical Besov spaces, Cannone--Planchon \cite{Can-Pla-96} and and Planchon \cite{Pla-98} proved the well-posedness of \eqref{eq:f-NS} with $f \equiv 0$ in $\dB_{p,q}^{n/p-1}(\mathbb{R}^n)$ with $1 \leq p < \infty$ and $1 \leq q \leq \infty$ and the solution satisfies $\n{u(t)}_{\dB_{p,q}^{n/p-1}} \to 0$ as $t \to \infty$ provided that $q<\infty$.
On the limiting case $p=\infty$, Koch--Tataru \cite{Koc-Tat-01} proved the global well-posedness in ${\rm BMO}^{-1}(\mathbb{R}^n)$, whereas the ill-posedness in $\dB_{\infty,q}^{-1}(\mathbb{R}^n)$ with all $1 \leq q \leq \infty$ was proven by Bourgain--Pavlovi\'{c} \cite{Bou-Pav-08}, Yoneda \cite{Yon-10}, and Wang \cite{Wan-15} in the sense that the solution-map is discontinuous. 
For the case of $f \not \equiv 0$, the interaction between the energy decay due to the viscous term and the energy supply from external forces makes the problem on the decay of solutions a more delicate one, and the functional framework for the well-posedness may be consequently different.
Yamazaki \cite{Yam-00} considered the case where the force may not decay in the temporal infinity, and showed the unique existence of small solutions $u \in C((0,\infty);L^{n,\infty}(\mathbb{R}^n))$ to \eqref{eq:f-NS} provided that $n \geq 3$ and $f$ has the form $f=\div F$ with small $F \in L^{\infty}(0,\infty;L^{n/2,\infty}(\mathbb{R}^n))$; more precisely, if $\n{a}_{L^{n,\infty}}+
\n{F}_{L^{\infty}(0,\infty;L^{n/2,\infty})} \ll 1$, then \eqref{eq:f-NS} admits a unique global solution $u$ with $\n{u}_{L^{\infty}(0,\infty;L^{n,\infty})} \leq C(\n{a}_{L^{n,\infty}}+\n{F}_{L^{\infty}(0,\infty;L^{n/2,\infty})})$.
See Ogawa--Rajopadhye--Schonbek \cite{Oga-Raj-Sch-97}, Okabe--Tsutsui \cites{Oka-Tsu-17, Oka-Tsu-22}, and Takeuchi \cite{Tak-25} for more related results on the forced Navier--Stokes equations. 
For the asymptotic stability of global solutions to  \eqref{eq:f-NS}, 
Kozono--Shimizu \cite{Koz-Shi-18} and Takeuchi \cite{Tak-pre} considered  in scaling critical Besov spaces with temporal weight.
In \cite{Tak-pre}, it was shown that for $n \geq 2$, $0 \leq \theta < 1$ (additionally $\theta \neq 0$ if $n=2$), and $1\leq p < n/(1-\theta)$,
if $a \in \dB_{p,\infty}^{n/p-1}(\mathbb{R}^n)$ and $f \in L^{\infty}_{\rm loc}((0,\infty);\dB_{p,\infty}^{n/p-3+2\theta}(\mathbb{R}^n))$ satisfy
\begin{align}
    \n{a}_{\dB_{p,\infty}^{n/p-1}}
    +
    \sup_{t>0}
    t^{\theta}
    \n{f(t)}_{\dB_{p,\infty}^{n/p-3+2\theta}}
    \ll 1,
\end{align}
then, there exists a unique solution $u$ with 
\begin{align}
    \sup_{t>0}\n{u(t)}_{\dB_{p,\infty}^{n/p-1}}
    +
    \sup_{t>0}t^{\theta}\n{u(t)}_{\dB_{p,\infty}^{n/p-1+2\theta}}
    \ll 1.
\end{align}
Therefore, when $\theta>0$, this result implies the sub-critical norm $\n{u(t)}_{\dB_{p,\infty}^{n/p-1+2\theta}}$ of the solution decays as $O(t^{-\theta})$, provided that $\n{f(t)}_{\dB_{p,\infty}^{n/p-3+2\theta}}=O(t^{-\theta})$.

In the present paper, we focus on the critical case $\theta=0$ in \cite{Tak-pre} and first observe that $\n{u(t)}_{\dB_{p,q}^{n/p-1}}=o(1)$ as $t \to \infty$ if $\n{f(t)}_{\dB_{p,q}^{n/p-3}}=o(1)$ as $t \to \infty$ holds for $n \geq 3$, $1 \leq p < n$, and $1 \leq q < \infty$.
Motivated by this, our central interest is to prove the asymptotic stability does not hold for the case of $n \geq 3$, $p \geq n$, and the case of $n=2$, $1 \leq p \leq \infty$.
\subsection{Our results} 
Let us mention our results precisely.
We begin with the stability result.
\begin{prop}\label{prop:stab}
    Let $n$, $p$, and $q$ satisfy
    \begin{align}\label{sta:expo}
        n \geq 3, \qquad 1 \leq p < n, \qquad 1 \leq q \leq \infty.
    \end{align}
    Then, there exist positive constants $\varepsilon=\varepsilon(n,p,q)$ and $C=C(n,p,q)$ such that if $a \in \dB_{p,q}^{n/p-1}(\mathbb{R}^n)$ with $\div a = 0$ and $f \in \widetilde{C}((0,\infty);\dB_{p,q}^{n/p-3}(\mathbb{R}^n))$ satisfy
    \begin{align}
        \n{a}_{\dB_{p,q}^{n/p-1}}
        +
        \n{f}_{\widetilde{L^{\infty}}(0,\infty;\dB_{p,q}^{n/p-3})}
        \leq
        \varepsilon,
    \end{align}
    then \eqref{eq:f-NS} possesses a unique solution $u \in \widetilde{C}([0,\infty);\dB_{p,q}^{n/p-1}(\mathbb{R}^n))$ with the estimate
    \begin{align}\label{u-apriori}
        \n{u}_{\widetilde{L^{\infty}}(0,\infty;\dB_{p,q}^{n/p-1})}
        \leq
        C
        \sp{
        \n{a}_{\dB_{p,q}^{n/p-1}}
        +
        \n{f}_{\widetilde{L^{\infty}}(0,\infty;\dB_{p,q}^{n/p-3})}
        }.
    \end{align}
    Moreover, if we additionally assume 
    \begin{align}
        1 \leq q<\infty, 
        \qquad
        \lim_{t \to \infty}
        \n{f(t)}_{\dB_{p,q}^{n/p-3}}
        =0,
    \end{align}
    then, it holds
    \begin{align}\label{u-lim}
        \lim_{t \to \infty}
        \n{u(t)}_{\dB_{p,q}^{n/p-1}} = 0.
    \end{align}
\end{prop}
Proposition \ref{prop:stab} is an improvement of the result of \cite{Tak-pre}*{Theorem 1.1} with $\theta=0$ as we are able to treat general third index $q$ of the Besov spaces by using the Chemin--Lerner spaces and also claim the decay property in the case of $\theta=0$ in \cite{Tak-pre}.
Therefore, this proposition is one of the novelties of this paper, while it is {\it not} our main outcome, since the proof may be easily obtained by the method of \cites{Pla-98}.
See in Appendix \ref{sec:pf_prop} for the proof.

Let us focus on the range \eqref{sta:expo} of $(n,p,q)$.
The assumption \eqref{sta:expo} ensures the bilinear estimate
\begin{align}\label{key-bilin}
    \begin{aligned}
    &
    \n{\int_0^t e^{(t-\tau)\Delta}\mathbb{P}\div (u\otimes v)(\tau)d\tau}_{\widetilde{L^{\infty}}(0,\infty;\dB_{p,q}^{n/p-1})}
    \\
    &
    \quad
    \leq
    C
    \n{u}_{\widetilde{L^{\infty}}(0,\infty;\dB_{p,q}^{n/p-1})}
    \n{v}_{\widetilde{L^{\infty}}(0,\infty;\dB_{p,q}^{n/p-1})},
    \end{aligned}
\end{align}
which plays a key role in the proof of Proposition \ref{prop:stab}, whereas it seems a natural question what happens in the case where $n$, $p$, and $q$ do not satisfy \eqref{sta:expo}.
The central interest of this paper is to show that uniform and asymptotic stability \eqref{u-apriori} and \eqref{u-lim} {\it fail} unless $n$, $p$, and $q$ admit \eqref{sta:expo}.
Now, let us summarize\footnote{See Theorems \ref{thm:nD} and \ref{thm:2D} below for slightly stronger results.} our main result of the present paper as follows.
\begin{thm}\label{thm:rough}
    Let $n$, $p$, and $q$ satisfy either \textup{(1)}, \textup{(2)}, or \textup{(3)}:
    \begin{enumerate}
        \item [\textup{(1)}] $n \geq 3$,\quad $n<p \leq \infty$,\quad $1 \leq q \leq  \infty$,
        \item [\textup{(2)}] $n \geq 3$,\quad $p=n$,\quad $2<q \leq \infty$,
        \item [\textup{(3)}] $n = 2$,\quad $1 \leq p \leq \infty$,\quad $q=1$.
    \end{enumerate}
    Then, there exist positive constants $\varepsilon_0=\varepsilon_0(n,p,q)$ and $c_0=c_0(n,p,q)$ such that for any $0<\varepsilon\leq \varepsilon_0$, there exists an external force $f_{\varepsilon} \in \widetilde{C}((0,\infty);\dB_{p,q}^{n/p-3}(\mathbb{R}^n))$ such that 
    \begin{align}
        \n{f_{\varepsilon}}_{\widetilde{L^{\infty}}(0,\infty;\dB_{p,q}^{n/p-3})}\leq \varepsilon,
        \qquad
        \lim_{t \to \infty}
        \n{f_{\varepsilon}(t)}_{\dB_{p,q}^{n/p-3}}
        =
        0,
    \end{align}
    and
    \eqref{eq:f-NS} with $a=0$ and $f=f_{\varepsilon}$ admits a global solution $u_{\varepsilon} \in \widetilde{C}([0,\infty);\dB_{p,q}^{n/p-1}(\mathbb{R}^n))$ satisfying 
    \begin{align}\label{u_e:osc}
        \limsup_{t \to \infty}\n{u_{\varepsilon}(t)}_{\dB_{p,q}^{n/p-1}}
        \geq c_0,
        \qquad
        \liminf_{t \to \infty}
        \n{u_{\varepsilon}(t)}_{\dB_{p,q}^{n/p-1}}
        =0.
    \end{align}
    In particular, $u_{\varepsilon}(t)$ does not strongly converge to any flow in $\dB_{p,q}^{n/p-1}(\mathbb{R}^n)$ as $t \to \infty$.
\end{thm}
\begin{rem}
    Some remarks on Theorem \ref{thm:rough} are in order.
    \begin{enumerate}
        \item Since $c_0$ is independent of $\varepsilon$, our result implies not only the asymptotic instability but also uniform instability, that is, the failure of the global uniform a priori estimate \eqref{u-apriori} with $a=0$.
        \item In Theorem \ref{thm:rough}, it seems difficult to specify the decay rare of the external force; see next subsection for the reason.
        On the other hand, if we aim to find a solution $u_{\varepsilon}$ with \eqref{u_e:osc} replaced by
        \begin{align}
            \liminf_{t \to \infty}
            \n{u_{\varepsilon}(t)}_{\dB_{p,q}^{n/p-1}}
            \geq c_0,
        \end{align}
        we may provide an example of external force with the decay rate $O(t^{-(1-n/p)})$ for the case of $n \geq 3$ and $p>n$. See Theorem \ref{prop:nD-non-osci} below for detail.
        \item Comparing Proposition \ref{prop:stab} and Theorem \ref{thm:rough}, the case of $p=n$, $1 \leq q \leq 2$ with $n \geq 3$ and the case of $1 \leq p \leq \infty$, $q>1$ with $n=2$ remains open.
        For the former case, it seems possible to obtain the instability for some $1 \leq q\leq 2$ by making use of the method of Li--Yu--Zhu \cites{Li-Yu-Zhu-25}; we do not focus on this case since the structure of the paper may be too complicated.
        However, it seems to be hopeless to achieve the latter case by using any of the methods used in previous research.
    \end{enumerate}
\end{rem}

\subsection{Idea of the proof} 
Let us mention our method for the proof of Theorem \ref{thm:rough}.
Here, we only focus on the simplest case of $n \geq 3$ with $p>n$ since the other cases may be considered by following the same split with some modification on the nonlinear estimates and choice of the external forces. 
See the explanation after Remark \ref{rem:2D} for more precise circumstance in the two-dimensional case.
We construct the external force with the intermittent temporal support $[0,T_*] \cup [T_1,T_1+T_*]\cup \cdots \cup [T_k,T_{k}+T_*] \cup [T_{k+1},T_{k+1}+T_*] \cup \cdots$; see Figure \ref{fig:1}. 
Here, the sequence $\{T_k\}_{k=1}^{\infty}$ and a constant $T_*>0$ are to be suitably defined so that the critical Besov norm of corresponding solution oscillates as $t \to \infty$.
Let mention how to construct the decaying external forces with the oscillating solutions more precisely.
We use the inductive argument for $k=0,1,2,\cdots$ and assume that there exists a solution $u_{\varepsilon}$ on $[0,T_k]$ with $\n{u_{\varepsilon}(T_k)}_{\dB_{p,q}^{n/p-1}} \leq 2^{-k}\eta^3$ with some constant $0<\eta\ll 1$.
We construct a solution $u_{\varepsilon}$ so that, on the interval $[T_k,T_k+T_*]$, 
the value of $\|u_{\varepsilon}(t)\|_{\dB_{p,q}^{n/p-1}}$ grows to the extent that it is bounded below by a constant 
independent of $k$; see Figure \ref{fig:2}. 
To achieve this, we choose the spatial profile of the external force $f$ on this interval 
so that its Fourier transform is a smooth function with the size $\eta$ and supported on a very high-frequency region like $\{\xi \in \mathbb{R}^3\ ;\ |\xi \pm ((2^{-k}\varepsilon)^{-p/(p-n)},0,0 )| \leq 1 \}$.
Then, since the derivative indices of $\dB_{p,q}^{n/p-3}(\mathbb{R}^3)$ and $\dB_{p,q}^{n/p-1}(\mathbb{R}^3)$ are negative due to $p>n$, we see that the critical norms of the external force and  linear part of the solution\footnote{The linear solution has the same Fourier support as that for the external force.} are bounded by $2^{-k}\varepsilon\eta$.
Considering the second iteration, obtained by inserting the linear solution into the nonlinear term, a portion of its Fourier support stay at the low frequency region like $\{\xi \in \mathbb{R}^3\ ;\ |\xi| \leq 1\}$, which we call ``high$\times$high$\to$low'' interaction; see Figure \ref{fig:high-high-to-low}.
Due to this nonlinear ``high$\times$high$\to$low'' frequency cascade, the norm of the solution, at time $t = T_k + T_*$, is bounded from below by \( c\eta^2 \), where $T_* \gg 1$ and $0<c \ll 1$ are suitable constant independent of $k$.
Note that this phenomenon essentially comes from the fact that the linear solution on $[T_k,T_k*T_*]$ provides a counter example of \eqref{key-bilin} with $p>n \geq 3$. 
On the other hand, on $t \geq T_k+T_*$, we set $f_{\varepsilon}(t)=0$ for a while 
so that the solution $u_{\varepsilon}$ decays for $t \gg T_k+T_*$ due to the dissipation. 
Therefore, we take $T_{k+1}$ sufficiently large so that $\|u_{\varepsilon}(T_{k+1})\|_{\dB_{p,q}^{n/p-1}} \leq 2^{-(k+1)}\eta^3$;
see Figure \ref{fig:2}.
This argument is based on Proposition \ref{prop:nD} below.
Connecting the solutions on each interval, we obtain a global solution $u_{\varepsilon}$ with 
\begin{align}
    \n{u_{\varepsilon}(T_k)}_{\dB_{p,q}^{n/p-1}}
    \leq 
    2^{-k}\varepsilon\eta,
    \quad
    \n{u_{\varepsilon}(T_k+T_*)}_{\dB_{p,q}^{n/p-1}}
    \geq c\eta^2,
    \quad k=1,2,3,\dots,
\end{align}
which proves the desired result.
Since the $k$-th decreasing time $T_{k+1}-(T_k+T_*)$ is implicit\footnote{This is because we are not able to specify the decay rate of the non-forced Navier--Stokes flow in our setting.}, the decay rate of the external force is unclear. 
Note that our instability essentially relies on the ill-posedness argument via the breakdown of the bilinear estimates.
Thus, if we do not insist on the oscillatory behavior of the solution norm and focus only on showing that the solution does not converge to zero as time tends to infinity, the argument becomes much simpler; see Appendix~\ref{sec:non-osci} for details.
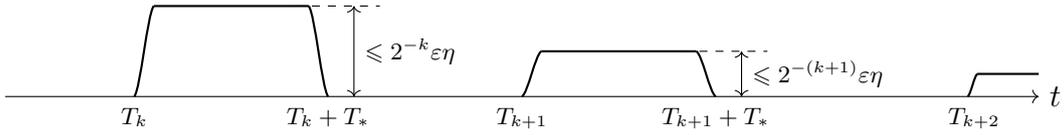
\begin{figure}[h]
    \centering
    
\begin{tikzpicture}[xscale=0.85, yscale=0.6]

\draw[->] (0,0) -- (16,0) node[right] {$t$};

\draw[thick] (2,0) .. controls (2.1,0.2) and (2.2,1.8) .. (2.3,2);
\draw[thick] (2.3,2) -- (4.7,2);
\draw[thick] (4.7,2) .. controls (4.8,1.8) and (4.9,0.2) .. (5,0);

\draw[<->] (5.4,0) -- (5.4,2) node[midway,right] {\scriptsize$ \leq 2^{-k}\varepsilon\eta$};
\draw[dashed] (4.5,2) -- (5.8,2);

\node[below] at (2,0) {\scriptsize$T_k$};
\node[below] at (5,0) {\scriptsize$T_{k}+T_\ast$};

\draw[thick] (8,0) .. controls (8.1,0.1) and (8.2,0.9) .. (8.3,1);
\draw[thick] (8.3,1) -- (10.7,1);
\draw[thick] (10.7,1) .. controls (10.8,0.9) and (10.9,0.1) .. (11,0);

\draw[<->] (11.4,0) -- (11.4,1) node[midway,right] {\scriptsize$ \leq 2^{-(k+1)}\varepsilon\eta$};
\draw[dashed] (10.5,1) -- (11.8,1);

\node[below] at (8,0) {\scriptsize$T_{k+1}$};
\node[below] at (11,0) {\scriptsize$T_{k+1}+T_\ast$};

\node[below] at (15,0) {\scriptsize $T_{k+2}$};
\draw[thick] (14.9,0) .. controls (14.95,0.05) and (15.0,0.45) .. (15.05,0.5);
\draw[thick] (15.05,0.5) -- (16,0.5); 
\end{tikzpicture}

    \caption{The behavior of $\|f_{\varepsilon}(t)\|_{\dB_{p,q}^{n/p-3}}$}
    \label{fig:1}
\end{figure}
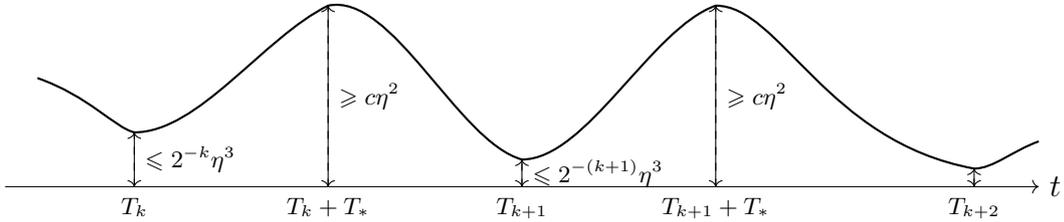
\begin{figure}[h]
    \centering
    
\begin{tikzpicture}[xscale=0.85, yscale=0.6]

\draw[->] (0,0) -- (16,0) node[right] {$t$};

\def\valleyn{1.2}          
\def\valleynext{0.6}       
\def\valleynextnext{0.4}   

\draw[thick]
  (0.5,2.4) 
  .. controls (1.3,2.0) and (1.7,1.3) .. (2,\valleyn)                 
  .. controls (3,\valleyn) and (4,3.3) .. (5,4)                       
  .. controls (6,4.3) and (7,1.0) .. (8,\valleynext)                  
  .. controls (9,\valleynext) and (10,3.5) .. (11,4)                  
  .. controls (12,4) and (12.7,1.0) .. (15,\valleynextnext)           
  .. controls (15.3,0.4) and (15.6,0.8) .. (16,1.0);                  

\draw[dashed] (2,0) -- (2,\valleyn);
\draw[dashed] (5,0) -- (5,4);
\draw[dashed] (8,0) -- (8,\valleynext);
\draw[dashed] (11,0) -- (11,4);
\draw[dashed] (15,0) -- (15,\valleynextnext);

\node[below] at (2,0) {\scriptsize $T_k$};
\node[below] at (5,0) {\scriptsize $T_k+T_\ast$};
\node[below] at (8,0) {\scriptsize $T_{k+1}$};
\node[below] at (11,0) {\scriptsize $T_{k+1}+T_\ast$};
\node[below] at (15,0) {\scriptsize $T_{k+2}$};

\draw[<->] (2,0) -- (2,\valleyn) node[midway,right] {\scriptsize $ \leq 2^{-k}\eta^3$};
\draw[<->] (5,0) -- (5,4) node[midway,right] {\scriptsize $\geq c\eta^2$};
\draw[<->] (8,0) -- (8,\valleynext) node[midway,right] {\scriptsize $\leq 2^{-(k+1)}\eta^3$};
\draw[<->] (11,0) -- (11,4) node[midway,right] {\scriptsize $\geq c\eta^2$};
\draw[<->] (15,0) -- (15,\valleynextnext);

\end{tikzpicture}
    \caption{The behavior of $\|u_{\varepsilon}(t)\|_{\dB_{p,q}^{n/p-1}}$}
    \label{fig:2}
\end{figure}

\begin{figure}[htbp]
  \centering
  \begin{tikzpicture}[scale=0.25, >=latex]
    \draw[semithick, -latex] (-20,0) -- (20.5,0) node[right] {$\xi_1$};
    \draw[semithick, -latex] (0,-5.5) -- (0,5.5) node[above] {$\xi_2,\xi_3$};

    \filldraw[fill=gray!25, draw=black] (-15,0) circle (2); 
    \filldraw[fill=gray!25, draw=black] ( 15,0) circle (2); 
    \filldraw[fill=gray!35, draw=black] (  0,0) circle (2); 

    \draw[thick, -latex]
      (-13,0.0) .. controls (-8,2) and (-5,2) .. (-1.99,0.20);
    \draw[thick, -latex]
      ( 13,0.0) .. controls ( 8,2) and ( 5,2) .. ( 1.99,0.20);
  \end{tikzpicture}
  \caption{The Fourier support of the ``high$\times$high$\to$low'' cascade}
  \label{fig:high-high-to-low}
\end{figure}
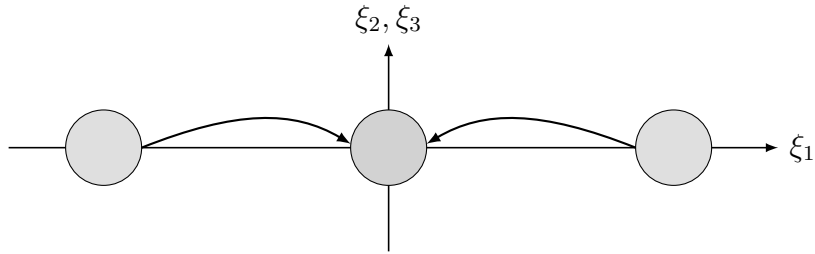

\subsection{Organization of the paper}
This paper is arranged as follows.
In Section \ref{sec:pre}, we recall the definitions and basic facts on the function spaces used in the entire of this manuscript.
In Section \ref{sec:nD}, we state a sharper result of Theorem \ref{thm:rough} in the case of $n \geq 3$ and provide its proof.
In Section \ref{sec:2D}, we focus on the case of $n = 2$ and complete the proof of our result.
Finally, we add a brief proof of Proposition \ref{prop:stab} in Appendix \ref{sec:pf_prop}, and also state an alternative instability result with a simple proof in Appendix \ref{sec:non-osci}.

Throughout this paper, we denote by $C$ and $c$ generic positive constants, which may differ in each line. In particular, $C=C(a_1,\dots,a_n)$ means that $C$ depends only on $a_1,\dots,a_n$. 
For any integrability exponent $1\leq p \leq \infty$, we denote by $p'$ the H\"{o}lder conjugate of $p$. 
For two Banach spaces $X$ and $Y$ with $X \cap Y \neq \varnothing$, 
we write $\| \cdot \|_{X \cap Y} := \| \cdot \|_X + \| \cdot \|_Y$.

\section{Preliminaries}\label{sec:pre}
In this section, we summarize the notations, especially the definitions and basic properties of function spaces, and then prepare some useful lemmas in our calculus.
\subsection{Function spaces}
To state the definition of the Besov space, we first recall the Littlewood--Paley decomposition.
Let $\mathscr{S}(\mathbb{R}^n)$ be the space of Schwartz functions and
$\mathscr{S}'(\mathbb{R}^n)$ be the space of tempered distributions.
Let $\varphi_0 \in \mathscr{S}(\mathbb{R}^n)$ be such that
\begin{align}
    \supp \widehat{\varphi_0} \subset \Mp{ \xi \in \mathbb{R}^n\ ;\ 2^{-1} \leq | \xi | \leq 2 },\quad
    0 \leq \widehat{\varphi_0}(\xi) \leq 1, 
\end{align}
and define the family $\Mp{\varphi_j}_{j\in\mathbb{Z}}\subset \mathscr{S}(\mathbb{R}^n)$ as
\begin{align}
    \sum_{j \in \mathbb{Z}}
    \widehat{\varphi_j}(\xi) = 1 \qquad {\rm for\ all\ }\xi \in  \mathbb{R}^n \setminus \{0\},
\end{align}
where $\varphi_j(x):=2^{nj}\varphi_0(2^jx)$.
Using them, we define the frequency localized operators by
\begin{align}
    \Delta_jf:=\mathscr{F}^{-1}\lp{\widehat{\varphi_j}(\xi)\widehat{f}}, \qquad f \in \mathscr{S}'(\mathbb{R}^n),\ j \in \mathbb{Z}.
\end{align}
For $1 \leq p,q \leq \infty$, and $s \in \mathbb{R}$, the homogeneous Besov space $\dB_{p,q}^s(\mathbb{R}^n)$ is defined as 
\begin{align}
    \dB_{p,q}^s(\mathbb{R}^n)
    :={}&
    \Mp{
    f \in \mathscr{S}'(\mathbb{R}^n)/\mathscr{P}(\mathbb{R}^n)
    \ ; \ 
    \n{f}_{\dB_{p,q}^s(\mathbb{R}^n)}<\infty
    },\\
    \n{f}_{\dB_{p,q}^s(\mathbb{R}^n)}
    :={}&
    \n{
    \Mp{
    2^{sj}
    \n{\Delta_j f}_{L^p(\mathbb{R}^n)}
    }_{j \in \mathbb{Z}}
    }_{\ell^{q}(\mathbb{Z})},
\end{align}
where $\mathscr{P}(\mathbb{R}^n)$ is the space of polynomials.
As basic properties of the Besov space, there hold the following continuous embeddings:
\begin{align}
    &\dB^{s+n/{p_1}}_{p_1,q_1}(\mathbb{R}^n) \hookrightarrow \dB^{s+{n}/{p_2}}_{p_2,q_2}(\mathbb{R}^n),
    \qquad s\in\mathbb{R},\ 1\leq p_1\leq p_2\leq\infty,\ 1\leq q_1\leq q_2\leq \infty
\end{align}
and
\begin{align}
    \dB^{s}_{p,1}(\mathbb{R}^n)
    \hookrightarrow \dot{H}^s_p(\mathbb{R}^n)
    \hookrightarrow \dB^{s}_{p,\infty}(\mathbb{R}^n),
    \qquad s\in\mathbb{R},\ 1\leq p\leq \infty,
\end{align}
where
\begin{align}
    \dot{H}^s_p(\mathbb{R}^n)
    :=&
    \Mp{
    f \in \mathscr{S}'(\mathbb{R}^n)/\mathscr{P}(\mathbb{R}^n)
    \ ; \ 
    \n{f}_{\dot{H}_{p}^s(\mathbb{R}^n)}:=\n{(-\Delta)^{s/2}f}_{L^p(\mathbb{R}^n)}<\infty
    }
\end{align}
denotes the homogeneous Sobolev space, which satisfies $\dot{H}^0_p(\mathbb{R}^n)=L^p(\mathbb{R}^n)$ if $1 \leq p \leq \infty$.
In addition, regarding the boundedness of differential operators and the isomorphism of the Riesz potential, the following holds:
\begin{align}
&\max_{|\alpha|=k}\n{\nabla^{\alpha} f}_{\dB^s_{p,q}(\mathbb{R}^n)}
\sim 
\n{f}_{\dB^{s+k}_{p,q}(\mathbb{R}^n)},
\qquad
s\in\mathbb{R},\ k\in\mathbb{N},\  1\leq p,q\leq\infty,
\\
&\n{(-\Delta)^{\alpha/2}f}_{\dB^s_{p,q}(\mathbb{R}^n)} \sim \n{f}_{\dB^{s+\alpha}_{p,q}(\mathbb{R}^n)},
\qquad
s\in\mathbb{R},\ \alpha\in\mathbb{R},\ 1\leq p,q\leq \infty.
\end{align} 
The above and more properties of the Besov space, we refer to Sawano \cite{Saw-18} for example.

We also frequently use the weak Lebesgue space $L^{p,\infty}(\mathbb{R}^n)$ for $1 <p<\infty$ defined by 
\begin{align}
    &
    L^{p,\infty}(\mathbb{R}^n)
    :=
    \Mp{
    f:\mathbb{R}^n \to \mathbb{R}
    \ ;\ 
    f\text{\ is\ Lebesgue\ measurable\ and\ } 
    \n{f}_{L^{p,\infty}}
    <\infty
    },
    \\
    &
    \n{f}_{L^{p,\infty}}
    :=
    \sup_{\substack{E \in \mathcal{L}(\mathbb{R}^n),\\ |E|<\infty}}
    \frac{1}{|E|^{1-{1}/{p}}}
    \int_E |f(x)| dx,
\end{align}
where $\mathcal{L}(\mathbb{R})$ denotes the set of all Lebesgue measurable subset of $\mathbb{R}^n$, and $|E|$ denotes the $n$-dimensional Lebesgue measure of $E \in \mathcal{L}(\mathbb{R}^n)$.
It is known that the following relationship with the weak  Lebesgue space hold:
\begin{align}
    \dB_{p,\infty}^{n/p-n/q}(\mathbb{R}^n) \hookrightarrow L^{q,\infty}(\mathbb{R}^n) \hookrightarrow \dB^{n/r-n/q}_{r,\infty}(\mathbb{R}^n),
    \qquad 1 \leq p < q < r \leq \infty.
\end{align}
See Cunanan--Okabe--Tsutsui \cite{Cun-Oka-Tsu-22} for the proof.

Moreover, for $1 \leq p,q,r \leq \infty$, $s \in \mathbb{R}$, and a time interval $I \subset \mathbb{R}$, we define 
the Chemin--Lerner space $\widetilde{L^r}(I;\dB_{p,q}^s(\mathbb{R}^n))$ by 
\begin{align}
    \widetilde{L^r}(I;\dB_{p,q}^s(\mathbb{R}^n))
    :={}&
    \Mp{
    F:I \to  \mathscr{S}'(\mathbb{R}^n)/\mathscr{P}(\mathbb{R}^n)
    \ ; \ 
    \n{F}_{\widetilde{L^r}(I;\dB_{p,q}^s(\mathbb{R}^n))}
    <\infty
    },\\
    \n{F}_{\widetilde{L^r}(I;\dB_{p,q}^s(\mathbb{R}^n))}
    :={}&
    \n{
    \Mp{
    2^{sj}
    \n{\Delta_j F}_{L^r(I;L^p(\mathbb{R}^n))}
    }_{j \in \mathbb{Z}}
    }_{\ell^{q}(\mathbb{Z})}.
\end{align}
We also use the following notation
\begin{align}
    \widetilde{C}(I ; \dB_{p,q}^s(\mathbb{R}^n))
    :=
    C(I ; \dB_{p,\sigma}^s(\mathbb{R}^n))
    \cap
    \widetilde{L^{\infty}}(I ; \dB_{p,q}^s(\mathbb{R}^n)).
\end{align}
As for the relationship with the Bochner space $L^r(I;\dB^{s}_{p,q}(\mathbb{R}^n))$,
there holds the following continuous embeddings:
\begin{align}
    &
    \widetilde{L^r}(I; \dB_{p,q}^s(\mathbb{R}^n))
    \hookrightarrow
    {L^r}(I; \dB_{p,q}^s(\mathbb{R}^n)),
    \qquad {\rm if\ } 1\leq q \leq r \leq\infty,\\
    &
    {L^r}(I; \dB_{p,q}^s(\mathbb{R}^n)) 
    \hookrightarrow
    \widetilde{L^r}(I; \dB_{p,q}^s(\mathbb{R}^n)),
    \qquad {\rm if\ } 1\leq r \leq q\leq \infty.
\end{align}
For more properties of the Chimin--Lerner space, we refer to Bahouri--Chemin--Danchin \cite{Bah-Che-Dan-11}.

\subsection{Practical lemmas}

As for the linear term, we use the following maximal regularity estimates:
\begin{lemm}[\cite{Bah-Che-Dan-11}]
\label{lemm:linear}
There exists a positive constant $C$ such that for any $-\infty < t_0<T\leq\infty$, $1\leq p,q \leq \infty$, $1\leq r \leq \rho \leq \infty$, and $s \in \mathbb{R}$, there hold
\begin{align}
    \left\| e^{(t-t_0)\Delta}a\right\|_{\widetilde{L^r}(t_0,T;\dB_{p,q}^{s+{2}/{r}})}
    &\leq C \|a\|_{\dB_{p,q}^s},\\
    \left\| \int_{t_0}^t e^{(t-\tau)\Delta} f(\tau) d\tau \right\|_{\widetilde{L^{\rho}}(t_0,T;\dB_{p,q}^{s+{2}/{\rho}})}
    &\leq C\|f\|_{\widetilde{L^{r}}(t_0,T;\dB_{p,q}^{s-2+{2}/{r}})}
\end{align}
for all $a\in\dB_{p,q}^s(\mathbb{R}^n)$ and $f\in\widetilde{L^{r}}(t_0,T;\dB_{p,q}^{s-2+{2}/{r}}(\mathbb{R}^n))$.
\end{lemm}
We remark here that such estimates are specific to the Chemin--Lerner space, and are hard to obtain in the Bochner space $L^r(0,T;\dB^s_{p,q}(\mathbb{R}^n))$. 

For the estimates of the nonlinear term, we use the Littlewood-Paley decomposition
\begin{align}
    fg = T_fg + R(f,g) + T_gf,
\end{align}
where we define 
\begin{align}
    T_fg := 
    \sum_{k \in \mathbb{Z}} 
    \left( \sum_{\ell \leq k-3} \Delta_{\ell} f \right) 
    \Delta_k g,
    \qquad
    R(f,g) :=
    \sum_{k \in \mathbb{Z}}
    \sum_{|k-\ell|\leq 2}
    \Delta_kf 
    \Delta_{\ell}g.
\end{align}
Actually, we have the following estimates for $T_fg$ and $R(f,g)$ in the Chemin--Lerner space.
\begin{lemm}\label{lemm:nonlin}
Let $n \in \mathbb{N}$ and let $I \subset \mathbb{R}$ be an interval.
Then, the following statements hold:
\begin{enumerate}
    \item 
    Let $1 \leq p,p_1,p_2,q,q_1,q_2,r,r_1,r_2 \leq \infty$ and $s,s_1,s_2 \in \mathbb{R}$ satisfy
    \begin{align}
        \frac{1}{p} = \frac{1}{p_1} + \frac{1}{p_2},\quad
        \frac{1}{q} \leq \frac{1}{q_1} + \frac{1}{q_2},\quad
        \frac{1}{r} = \frac{1}{r_1} + \frac{1}{r_2},\quad
        s=s_1+s_2, \quad 
        s_1 < 0.
    \end{align}
    Then, there exists a positive constant $C=C(q_1,s_1,s_2)$ such that 
    \begin{align}\label{est:T-2}
        \n{T_fg}_{\widetilde{L^r}(I;\dB_{p,q}^s(\mathbb{R}^n))}
        \leq
        C
        \n{f}_{\widetilde{L^{r_1}}(I;\dB_{p_1,q_1}^{s_1}(\mathbb{R}^n))}
        \n{g}_{\widetilde{L^{r_2}}(I;\dB_{p_2,q}^{s_2}(\mathbb{R}^n))}
    \end{align}
    for all 
    $f \in \widetilde{L^{r_1}}(I;\dB_{p_1,q_1}^{s_1}(\mathbb{R}^n))$ 
    and 
    $g \in \widetilde{L^{r_2}}(I;\dB_{p_2,q}^{s_2}(\mathbb{R}^n))$.
    \item 
    Let $1 \leq p,p_1,p_2,q,q_1,q_2,r,r_1,r_2 \leq \infty$ and $s,s_1,s_2 \in \mathbb{R}$ satisfy
    \begin{align}
        \frac{1}{p} = \frac{1}{p_1} + \frac{1}{p_2},\quad
        \frac{1}{q} \leq \frac{1}{q_1} + \frac{1}{q_2},\quad
        \frac{1}{r} = \frac{1}{r_1} + \frac{1}{r_2},\quad 
        s=s_1+s_2>0.
    \end{align}
    Then, there exists a positive constant $C=C(s, s_1,s_2)$ such that 
    \begin{align}\label{est:R-2}
        \n{R(f,g)}_{\widetilde{L^r}(I;\dB_{p,q}^s(\mathbb{R}^n))}
        \leq
        C
        \n{f}_{\widetilde{L^{r_1}}(I;\dB_{p_1,q_1}^{s_1}(\mathbb{R}^n))}
        \n{g}_{\widetilde{L^{r_2}}(I;\dB_{p_2,q_2}^{s_2}(\mathbb{R}^n))}
    \end{align}
    for all 
    $f \in \widetilde{L^{r_1}}(I;\dB_{p_1,q_1}^{s_1}(\mathbb{R}^n))$ 
    and 
    $g \in \widetilde{L^{r_2}}(I;\dB_{p_2,q_2}^{s_2}(\mathbb{R}^n))$.
\end{enumerate}
\end{lemm}
See \cite{Bah-Che-Dan-11}*{Chapter 2} for the method of the proof.

Using Lemmas \ref{lemm:linear} and \ref{lemm:nonlin}, we can show the following key estimates for the nonlinear term.
\begin{lemm}\label{lemm:nonlin-ndim}
Let $n \geq 2$ be an integer.
Let $1 \leq p,p_1,q_1,q_2,q_3, r_1,r_2,r_3 \leq \infty$ satisfy
\begin{gather}
    \frac{1}{r_1}  \leq \frac{1}{r_2} +\frac{1}{r_3} \leq 1,\qquad
    r_2,r_3>2, \qquad 
    \frac{1}{q_1} \leq \frac{1}{q_2} + \frac{1}{q_3},
    \\
    n\sp{\frac{1}{p}-\frac{1}{p_1}}
    < 1 - \frac{2}{r_3},
    \qquad
    \min
    \Mp{n,n\sp{\frac{1}{p}+\frac{1}{p_1}}}-2+\frac{1}{r_2}+\frac{1}{r_3}>0.
    \label{assump:pqr}
\end{gather}
Then, there exists a positive constant $C=C(n,p,p_1,q_1,q_2,q_3,r_1,r_2,r_3)$ such that
\begin{align}
    &\n{
    \int_{t_0}^t
    e^{(t-\tau)\Delta}
    \mathbb{P}\div(u \otimes v)(\tau)d\tau
    }_{\widetilde{L^{r_1}}(I;\dB_{p,q_1}^{n/p-1+2/r_1})}\\
    &\quad\leq
    C
    \n{u}_{\widetilde{L^{r_2}}(I;\dB_{p_1,q_2}^{n/p_1-1+2/r_2})}
    \n{v}_{\widetilde{L^{r_3}}(I;\dB_{p,q_3}^{n/p-1+2/r_3})}
\end{align}
for all 
$I=(t_0,t_1) \subset \mathbb{R}$, 
$u \in {L^{\infty}}(I;\dB_{p_1,q_2}^{n/p_1-1}(\mathbb{R}^n))$, 
and 
$v \in \widetilde{L^r}(I;\dB_{p,q_3}^{n/p-1+2/r}(\mathbb{R}^n))$.
\end{lemm}

\begin{proof}
	Let us first decompose the Duhamel integral as 
	\begin{align}
		&
		\int_{t_0}^t
		e^{(t-\tau)\Delta}
		\mathbb{P}\div(u \otimes v)(\tau)d\tau
		\\
		&\quad
		={}
		\int_{t_0}^t
		e^{(t-\tau)\Delta}
		\mathbb{P}\div \sp{
			T^{\otimes}_{u} v
			+ 
			\widetilde{T}^{\otimes}_{v} u
			+
			R^{\otimes}(u,v)}(\tau)d\tau
		\\
		&\quad 
		=:{}
		D_1[u,v](t) + D_2[u,v](t) + D_3[u,v](t),
	\end{align}
    where we have set 
    \begin{align}
        &
        T_u^{\otimes}v:=
        \sum_{j \in \mathbb{Z}}
        \sum_{k \leq j-3}\Delta_ku \otimes \Delta_jv,
        \quad 
        \widetilde{T}_v^{\otimes}u=\sp{T_v^{\otimes}u}^{\top}
        ,\\
        &
        R^{\otimes}(u,v):=
        \sum_{j \in \mathbb{Z}}
        \sum_{|k - j| \leq 2}
        \Delta_ku \otimes \Delta_jv.
    \end{align}
    In what follows, we make use of Lemmas \ref{lemm:linear} and \ref{lemm:nonlin} to proceed the estimates.
	Let $r$ satisfies $1/r=1/r_2+1/r_3$.
	As for the estimate of $D_1[u,v]$,
	we have  
	\begin{align}
		\n{D_1[u,v]}_{\widetilde{L^{r_1}}(I;\dB_{p,q_1}^{n/p-1+2/r_1})}
		& \leq 
		C
		\n{T^{\otimes}_{u} v}_{\widetilde{L^{r}}(I;\dB_{p,q_1}^{n/p-2+{2}/{r}})}\\
		& \leq 
		C
		\n{u}_{\widetilde{L^{r_2}}(I;\dB_{\infty,q_2}^{-1+2/r_2})}
		\n{v}_{\widetilde{L^{r_3}}(I;\dB_{p,q_3}^{n/p-1+2/r_3})}\\ 
		& \leq 
		C
		\n{u}_{\widetilde{L^{r_2}}(I;\dB_{p_1,q_2}^{n/p_1-1+2/r_2})}
		\n{v}_{\widetilde{L^{r_3}}(I;\dB_{p,q_3}^{n/p-1+2/r_3})}.
	\end{align}
	Next we consider the estimate of $D_2[u,v]$.
	For the case of $p\leq p_1$, 
	setting $p \leq p_2 \leq \infty$ by $1/p_2=1/p-1/p_1$,
	and we see that 
	\begin{align}
		\n{D_2[u,v]}_{\widetilde{L^{r_1}}(I;\dB_{p,q_1}^{n/p-1+2/r_1})}
		& \leq 
		C
		\n{T^{\otimes}_{v} u}_{\widetilde{L^{r}}(I;\dB_{p,q_3}^{n/p-2+{2}/{r}})}\\
		& \leq 
		C
		\n{v}_{\widetilde{L^{r_3}}(I;\dB_{p_2,q_3}^{n/p_2-1+2/r_3})}
		\n{u}_{\widetilde{L^{r_2}}(I;\dB_{p_1,q_2}^{n/p_1-1+2/r_2})}\\ 
		& \leq 
		C
		\n{u}_{\widetilde{L^{r_2}}(I;\dB_{p_1,q_2}^{n/p_1-1+2/r_2})}
		\n{v}_{\widetilde{L^{r_3}}(I;\dB_{p,q_3}^{n/p-1+2/r_3})}.
	\end{align}
	For the  case of $p_1 \leq p$, it holds
	\begin{align}
		\n{D_2[u,v]}_{\widetilde{L^{r_1}}(I;\dB_{p,q_1}^{n/p-1+2/r_1})}
		& \leq 
		C
		\n{T^{\otimes}_{v} u}_{\widetilde{L^{r}}(I;\dB_{p,q_1}^{n/p-2+{2}/{r}})}\\
		& \leq 
		C
		\n{T^{\otimes}_{v} u}_{\widetilde{L^{r}}(I;\dB_{p_1,q_1}^{n/p_1-2+{2}/{r}})}\\
		& \leq 
		C
		\n{v}_{\widetilde{L^{r_3}}(I;\dB_{\infty,q_3}^{-1+2/r_3})}
		\n{u}_{\widetilde{L^{r_2}}(I;\dB_{p_1,q_2}^{n/p_1-1+2/r_2})}\\ 
		& \leq 
		C
		\n{u}_{\widetilde{L^{r_2}}(I;\dB_{p_1,q_2}^{n/p_1-1+2/r_2})}
		\n{v}_{\widetilde{L^{r_3}}(I;\dB_{p,q_3}^{n/p-1+2/r_3})}.
	\end{align}
	Finally we consider the estimate of $D_3[u,v]$.
	In the case of $p_1 \leq p'$, that is, $1/p + 1/p_1 \geq 1$,
	we see that
	\begin{align}
		\n{D_3[u,v]}_{\widetilde{L^{r_1}}(I;\dB_{p,q_1}^{n/p-1+2/r_1})}
		& \leq 
		C
		\n{R^{\otimes}(u,v)}_{\widetilde{L^{r}}(I;\dB_{p,q_1}^{n/p-2+{2}/{r}})}\\
		& \leq 
		C
		\n{R^{\otimes}(u,v)}_{\widetilde{L^{r}}(I;\dB_{1,q}^{n-3+{2}/{r}})}\\
		& \leq 
		C
		\n{u}_{\widetilde{L^{r_2}}(I;\dB_{p',q_2}^{n/p'-1+2/r_2})}
		\n{v}_{\widetilde{L^{r_3}}(I;\dB_{p,q_3}^{n/p-1+2/r_3})}\\ 
		& \leq 
		C
		\n{u}_{\widetilde{L^{r_2}}(I;\dB_{p_1,q_2}^{n/p_1-1+2/r_2})}
		\n{v}_{\widetilde{L^{r_3}}(I;\dB_{p,q_3}^{n/p-1+2/r_3})}.
	\end{align}
	For the case of $p_1 \geq p'$ that is $1/p + 1/p_1 \leq 1$, we define $p_3$ by $1/p_3=1/p + 1/p_1$ and see by $p_3 \leq p$ that 
	\begin{align}
		\n{D_3[u,v]}_{\widetilde{L^{r_1}}(I;\dB_{p,q_1}^{n/p-1+2/r_1})}
        & \leq 
		C
		\n{R^{\otimes}(u,v)}_{\widetilde{L^{r}}(I;\dB_{p,q_1}^{n/p-2+{2}/{r}})}\\
		& \leq 
		C
		\n{R^{\otimes}(u,v)}_{\widetilde{L^{r}}(I;\dB_{p_3,q_1}^{n/{p_3}-2+{2}/{r}})}\\
		& \leq 
		C
		\n{u}_{\widetilde{L^{r_2}}(I;\dB_{p_1,q_2}^{n/p_1-1+2/r_2})}
		\n{v}_{\widetilde{L^{r_3}}(I;\dB_{p,q_3}^{n/p-1+2/r_3})}.  
	\end{align}
	Thus, we complete the proof.
\end{proof}
Finally, we prepare the estimate in the Besov space for a mollified trigonometric function frequently used in the following sections. 
\begin{lemm}\label{lemm:cos}
Let $n\geq 1$, $s\in\mathbb{R}$, $1\leq p\leq\infty$, $R>0$, and $\varphi \in \mathscr{S}(\mathbb{R}^n)$ such that $\supp \widehat{\varphi}\subset \{ \xi \in \mathbb{R}^n\ ;\ |\xi| \leq 1 \}$.
Let us define
\begin{equation*}
W_R(x):= \varphi(x)\cos(R x_1),\qquad x\in\mathbb{R}^n.
\end{equation*}
Then there exist positive constants $r$ and $C=C(p, \varphi)$ such that if $R>r$, it holds
\begin{equation*}
\|W_R\|_{\dot B^{s}_{p,1}(\mathbb{R}^n)} \leq C R^{s}.
\end{equation*}
\end{lemm}

\noindent
\begin{proof}
From a direct calculation, it follows that
\begin{align*}
\widehat{W_R}(\xi) =\frac{1}{2}
\left( \widehat \varphi(\xi-R e_1) + \widehat \varphi (\xi+R e_1) \right),
\end{align*}
which yields
\begin{align*}
\supp \widehat{W_R} \subset \{\xi\in \mathbb{R}^n;\ R-1
\leq |\xi| \leq R+1\}.
\end{align*}
This fact shows that if $R$ is sufficiently large, there exist at most three indices $j\in\mathbb{Z}$ such that $\Delta_j W_R \not\equiv 0$, and such $j$ should satisfy $(R-1)/2\leq 2^j \leq 2(R+1)$. We then obtain
\begin{align*}
\|W_R\|_{\dot B^{s}_{p,1}(\mathbb{R}^n)} 
\leq \sum_{(R-1)/2\leq 2^j \leq 2(R+1)} 2^{js}\|\Delta_j W_R\|_{L^p}
\leq CR^s\|\varphi\|_{L^p} \leq CR^s,
\end{align*}
which completes the proof.
\end{proof}

\section{High-dimensional analysis}
\label{sec:nD}
The aim of this section is to prove Theorem \ref{thm:rough} in the case of $n \geq 3$.
To this end, we focus on the following theorem that claims slightly stronger statement than Theorem \ref{thm:rough}. 
\begin{thm}\label{thm:nD}
    Let $n$, $r$, $\sigma$, and $\rho$ satisfy either \textup{(1)} or \textup{(2)}:
    \begin{enumerate}
        \item [\textup{(1)}] $n \geq 3$,\quad $n<r<2n$,\quad $1 \leq \sigma < \infty$,
        \item [\textup{(2)}] $n \geq 3$,\quad $r=n$,\quad $2<\sigma < \infty$.
    \end{enumerate}
    Then, there exist positive constants $\varepsilon_0=\varepsilon_0(n,r,\sigma)$ and $c_0=c_0(n,r,\sigma)$ such that for any $0<\varepsilon\leq \varepsilon_0$, there exists an external force $f_{\varepsilon} \in \widetilde{C}((0,\infty);\dB_{r,\sigma}^{n/r-3}(\mathbb{R}^n))$ such that 
    \begin{align}\label{f-1}
        \n{f_{\varepsilon}}_{\widetilde{L^{\infty}}(0,\infty;\dB_{r,\sigma}^{n/r-3})}\leq \varepsilon,
        \qquad
        \lim_{t \to \infty}
        \n{f_{\varepsilon}(t)}_{\dB_{r,\sigma}^{n/r-3}}
        =
        0,
    \end{align}
    and
    \eqref{eq:f-NS} with $a=0$ and $f=f_{\varepsilon}$ admits a global solution $u_{\varepsilon} \in \widetilde{C}([0,\infty);\dB_{r,\sigma}^{n/r-1}(\mathbb{R}^n))$ satisfying 
    \begin{align}\label{osc-1}
        \limsup_{t \to \infty}\n{u_{\varepsilon}(t)}_{\dB_{\infty,\infty}^{-1}}
        \geq c_0,
        \qquad
        \liminf_{t \to \infty}
        \n{u_{\varepsilon}(t)}_{\dB_{r,\sigma}^{n/r-1}}
        =0.
    \end{align}
\end{thm}
\begin{rem}\label{rem:nD}
    For any $p$ and $q$ satisfying the assumptions (1) or (2) in Theorem \ref{thm:rough}, we may choose $r=r(n,p) \leq p$ and $\sigma=\sigma(q) \leq q$ enjoying (1) or (2) in Theorem \ref{thm:nD}, respectively.
    Then, by making use of the continuous embedding $\dB_{r,\sigma}^{n/r+s}(\mathbb{R}^n) \hookrightarrow \dB_{p,q}^{n/p+s}(\mathbb{R}^n) \hookrightarrow \dB_{\infty,\infty}^{s}(\mathbb{R}^n)$ with $s=-1,-3$, 
    we immediately obtain Theorem \ref{thm:rough} with $n \geq 3$ from Theorem \ref{thm:nD}.
\end{rem}
\subsection{Key propositions}
In order to prove Theorem \ref{thm:nD}, we begin by preparing some crucial propositions.
We first mention the well-posedness result of \eqref{eq:f-NS} in $L^\infty(0,\infty;L^{n,\infty}(\mathbb{R}^n))$ proved by Yamazaki \cite{Yam-00}, which plays an important role in controlling the $L^{n,\infty}$ norm of our solution.
In order to suit our problem, the following proposition presents a slight modification of the assumption on external forces from the original result \cite{Yam-00}.
\begin{prop}\label{prop:Yamazaki}
    Let $n \geq 3$.
    Then, there exist positive constants $C_*$, $\mu=\mu(n)$, and $\nu=\nu(n)$
    such that
    for if $f \in \widetilde{C}((0,\infty);\dB_{n,2}^{-2}(\mathbb{R}^n))$
    enjoy
    \begin{align}
        \n{f}_{\widetilde {L^{\infty}}(0,\infty;\dB_{n,2}^{-2})}
        \leq
        \mu, 
    \end{align}
    then
    \eqref{eq:f-NS} with $a=0$ possesses a solution $u \in C((0,\infty);L^{n,\infty}(\mathbb{R}^n))$ satisfying 
    \begin{align}
        \sup_{t>0}\n{u(t)}_{L^{n,\infty}}
        \leq 
        C_*
        \n{f}_{\widetilde {L^{\infty}}(0,\infty;\dB_{n,2}^{-2})}.
    \end{align}
    Moreover, if two solutions $u$ and $v$ to \eqref{eq:f-NS} with $u(0)=v(0)=a$ belong to 
    \begin{align}\label{unique}
        \Mp{u \in L^{\infty}(0,\infty;L^{n,\infty}(\mathbb{R}^n))
        \ ; \  
        \n{u}_{L^{\infty}(0,\infty;L^{n,\infty})}
        \leq
        \nu},
    \end{align}
    then it holds $u \equiv v$.
\end{prop}
Proposition \ref{prop:Yamazaki} is immediately obtained from the standard contraction mapping argument via the linear estimate
\begin{align}
    \n{\int_0^t e^{(t-\tau)\Delta}\mathbb{P}f(\tau)d\tau}_{L^{\infty}(0,\infty;L^{n,\infty})}
    \leq{}&
    C
    \n{\int_0^t e^{(t-\tau)\Delta}\mathbb{P}f(\tau)d\tau}_{L^{\infty}(0,\infty;L^{n})}\\
    \leq{}&
    C
    \n{\int_0^t e^{(t-\tau)\Delta}\mathbb{P}f(\tau)d\tau}_{\widetilde{L^{\infty}}(0,\infty;\dB_{n,2}^0)}\\
    \leq{}&
    C
    \n{f}_{\widetilde{L^{\infty}}(0,\infty;\dB_{n,2}^{-2})}
\end{align}
and the nonlinear estimate given by Lemma \ref{lemm:Yamazaki}.
Thus, we omit the proof.

As was mentioned in Section \ref{sec:intro}, we prove our main result by extending the appropriate external force and solution to the next temporal interval inductively. 
To this end, we introduce the following proposition.
\begin{prop}\label{prop:nD}
    Let $n \geq 3$ and $t_0 \in \mathbb{R}$. 
    Let $r$, $\sigma$, and $\rho$ satisfy
    \begin{align}\label{case1}
        n< r < 2n, \qquad 1 \leq \sigma < \infty, \qquad 2 < \rho < \frac{r}{r-n},
    \end{align}
    or
    \begin{align}\label{case2}
        r=n, \qquad 2 < \sigma < \infty, \qquad 2 < \rho < \infty.
    \end{align}
    Then, there exist positive constants $\eta_0=\eta_0(n,r,\sigma,\rho)$, $K_0=K_0(n,r,\sigma,\rho)$, and  $T_*=T_*(n,r,\sigma,\rho)$ such that for any $0< \eta \leq \eta_0$ and $0<\delta \leq \eta^2$, there exists $f_{\delta,\eta} \in \widetilde{C}([t_0,\infty);\dB_{r,\sigma}^{n/r-3}(\mathbb{R}^n))\cap  \widetilde{C}([t_0,\infty);\dB_{n,2}^{-2}(\mathbb{R}^n))$
    satisfying
    $f_{\delta,\eta}(t_0)=0$, $f_{\delta,\eta}(t)=0$ for all $t \geq t_0+T_*$, and 
    \begin{align}
        \n{f_{\delta,\eta}}_{{\widetilde{L^{\infty}}}(t_0,\infty;\dB_{n,2}^{-2})} 
        \leq \eta,
        \qquad 
        \n{f_{\delta,\eta}}_{\widetilde{L^{\infty}}(t_0,\infty;\dot{B}^{n/r-3}_{r,\sigma})}
        \leq \eta\delta,
    \end{align}
    such that
    for any $a \in L^{n,\infty}(\mathbb{R}^n) \cap \dB_{r,\sigma}^{n/r-1}(\mathbb{R}^n)$ with\footnote{The constant $C_*$ is the same one appearing in Proposition \ref{prop:Yamazaki}.}
    \begin{align}
        \div a = 0, 
        \qquad 
        \n{a}_{L^{n,\infty}} \leq C_*\eta, 
        \qquad
        \n{a}_{\dB_{r,\sigma}^{n/r-1}}
        \leq 
        \eta^3,
    \end{align}
    the forced Navier--Stokes system 
    \begin{align}\label{eq:f-NS-t0}
        \begin{cases}
            \partial_t u - \Delta u + \mathbb{P} \div 
            \sp{u \otimes u} = f_{\delta,\eta}, &t > t_0, x \in \mathbb{R}^n, \\
            \div u = 0, & t \geq t_0, x \in \mathbb{R}^n, \\
            u(t_0,x)=a(x), & t=t_0,x\in \mathbb{R}^n
        \end{cases}
    \end{align}
    admits a unique solution  
    \begin{align}
        u_{\delta,\eta}\in {}
        &
        C((t_0,\infty);L^{n,\infty}(\mathbb{R}^n))
        \\
        &
        \cap 
        \widetilde{C}([t_0,\infty);\dB_{r,\sigma}^{n/r-1}(\mathbb{R}^n)) \cap \widetilde{L^{\rho}}(t_0,\infty;\dB_{r,\sigma}^{n/r-1+2/\rho}(\mathbb{R}^n))
    \end{align}
    satisfying
    \begin{align}
        &
        \n{u_{\delta,\eta}}_{L^{\infty}(t_0,\infty;L^{n,\infty})}
        \leq K_0\eta,
        \\
        &
        \n{u_{\delta,\eta}(t_0+T_*)}_{\dB_{\infty,\infty}^{-1}}
        \geq \frac{\eta^2}{K_0},
        \\
        &
        \lim_{t \to \infty}
        \n{u_{\delta,\eta}(t)}_{\dB_{r,\sigma}^{n/r-1}}
        =
        0.
    \end{align}
\end{prop}
To construct the desired solutions, we prepare the following nonlinear estimate.
\begin{lemm}\label{cor:nonlin-ndim}
    For any $n \geq 2$, $n\leq r < 2n$, $2<\rho<r/(r-n)$, and $\rho\leq \rho_1 \leq \infty$, 
    there exists a positive constant $C=C(n,r,\sigma,\rho)$ such that
    \begin{align}
    &\n{
    \int_{t_0}^t
    e^{(t-\tau)\Delta}
    \mathbb{P}\div(u(\tau) \otimes v(\tau))d\tau
    }_{\widetilde{L^{\rho_1}}(I;\dB_{p,q}^{n/r-1+2/\rho_1})}\\
    &\quad\leq
    C
    \n{u}_{L^{\infty}(I;L^{n,\infty})}
    \n{v}_{\widetilde{L^{r}}(I;\dB_{r,\sigma}^{n/r-1+2/\rho})}
    \end{align}
    for all $u \in L^{\infty}(I;L^{n,\infty}(\mathbb{R}^n))$ and $v \in \widetilde{L^{\rho}}(I;\dB_{r,\sigma}^{n/r-1+2/\rho}(\mathbb{R}^n))$.
\end{lemm}

\begin{proof}
    We choose may $r_1$ so that
    \begin{align}
        \max\Mp{0,\frac{1}{r}-\frac{1}{n}\sp{1 - \frac{2}{\rho}},\frac{2}{n}\sp{1 - \frac{1}{\rho}} - \frac{1}{r}
        }
        <
        \frac{1}{r_1}
        < 
        \frac{1}{n}.
    \end{align}
    Then, for $\rho\leq \rho_1 \leq \infty$, we have by Lemma \ref{lemm:nonlin-ndim} that
    \begin{align}
    &\n{
    \int_{t_0}^t
    e^{(t-\tau)\Delta}
    \mathbb{P}\div(u(\tau) \otimes v(\tau))d\tau
    }_{\widetilde{L^{\rho_1}}(I;\dB_{r,q}^{n/r-1+2/\rho_1})}\\
    &\quad\leq
    C
    \n{u}_{L^{\infty}(I;\dB_{r_1,\infty}^{n/r_1-1})}
    \n{v}_{\widetilde{L^{\rho}}(I;\dB_{r,\sigma}^{n/r-1+2/\rho})}\\
    &\quad\leq
    C
    \n{u}_{L^{\infty}(I;L^{n,\infty})}
    \n{v}_{\widetilde{L^{r}}(I;\dB_{r,\sigma}^{n/r-1+2/\rho})}
    \end{align}
    by using
    $L^\infty(I;\dB^{n/r_1-1}_{r_1,\infty}(\mathbb{R}^n))= 
    \widetilde{L^\infty}(I;\dB^{n/r_1-1}_{r_1,\infty}(\mathbb{R}^n))$
    and the embedding $L^{n,\infty}(\mathbb{R}^n) \hookrightarrow \dB_{r_1,\infty}^{n/r_1-1}(\mathbb{R}^n)$.
\end{proof}
In addition, we recall an estimate of the nonlinear term in $L^\infty(I;L^{n,\infty})$.
\begin{lemm}[\cite{Yam-00}]\label{lemm:Yamazaki}
    Let $n\geq 3$.
    Then, there exists a positive constant $C=C(n)$ such that 
    \begin{align}
    \n{
    \int_{t_0}^t
    e^{(t-\tau)\Delta}
    \mathbb{P}\div(u(\tau) \otimes v(\tau))d\tau
    }_{L^{\infty}(I;L^{n,\infty})}
    \leq
    C
    \n{u}_{L^{\infty}(I;L^{n,\infty})}
    \n{v}_{L^{\infty}(I;L^{n,\infty})}
    \end{align}
    for all $I=(t_0,t_1) \subset \mathbb{R}$ and $u,v \in L^{\infty}(I;L^{n,\infty}(\mathbb{R}^n))$.
\end{lemm}
Now, we prove Proposition \ref{prop:nD}.
\begin{proof}[Proof of Proposition \ref{prop:nD}]
We first consider the case \eqref{case1} by separating the argument into five steps.

\noindent
{\it Step 1. The external force.}
Let $0 < \delta \leq \eta \leq 1/100$.
Let $\widetilde{\chi} \in C^{\infty}(\mathbb{R})$ satisfy
\begin{align}
    \widetilde{\chi}(\tau)=
    \begin{cases}
        1 & (\tau \geq h), \\
        0 & (\tau \leq h/2 ),
    \end{cases}
\end{align}
where $0<h\leq 1$ is a pure constant to be determined later,
and put 
\begin{align}
    \chi(t)
    :=
    \widetilde{\chi}\sp{t-t_0}
    \widetilde{\chi}\sp{t_0+T_*-t},
\end{align}
where $T_* \geq 1$ is a constant to be determined later.
Let $\psi\in\mathscr{S}(\mathbb{R}^n)$ be a Schwartz function such that
\begin{align}
    \supp{\widehat\psi} \subset \left\{\xi\in\mathbb{R}^n\ ;\ |\xi|\leq 1\right\},
    \quad
    \widehat{\psi}>0\ {\rm in}\ \left\{\xi\in\mathbb{R}^n\ ;\ |\xi|< 1\right\},
\end{align}
and set
\begin{align}\label{eq:Psi}
    \Psi(x):=\partial_{x_3}\psi(x)e_2-\partial_{x_2}\psi(x)e_3,
\end{align}
where $e_j$ denotes a unit vector along the $x_j$-axis.
Note that $\Psi(x)\phi(x_1)$ is solenoidal for all smooth function $\phi(x_1)$.
Let us define
\begin{align}
    f_{\delta,\eta}(t,x)
    :={}&
    \chi(t)
    g_{\delta,\eta}(x),
\end{align}
where we have set 
\begin{align}
    g_{\delta,\eta}(x):=
    \eta
    \Delta
    \left[
    \Psi(x)
    \cos\sp{\delta^{-\frac{r}{r-n}}x_1}
    \right].
\end{align}
By using Lemma \ref{lemm:cos}, we see that
\begin{align}
   &\n{f_{\delta,\eta}}_{\widetilde{L^\infty}(t_0,\infty;\dB^{-2}_{n,2})}
    \leq
    C\n{g_{\delta,\eta}}_{\dB^{-2}_{n,2}}
    \leq
    C\eta,\\
    &
    \n{f_{\delta,\eta}}_{\widetilde{L^\infty}(t_0,\infty;\dB^{n/r-3}_{r,\sigma})}
    \leq
    C \n{g_{\delta,\eta}}_{\dB^{n/r-3}_{r,\sigma}}
    \leq
    C\eta\delta.
\end{align}

\noindent
{\it Step 2. The first iteration.}
Let us consider the first iteration $u_{\delta,\eta}^{(1)}$ generated by the initial data $a$ and the external force $f_{\delta,\eta}$ defined as 
\begin{align}
    u_{\delta,\eta}^{(1)}(t)
    :={}&
    e^{(t-t_0)\Delta}a
    +
    \int_{t_0}^t
    e^{(t-\tau)\Delta}f_{\delta,\eta}(\tau)d\tau.
\end{align}
Then we see by Lemma \ref{lemm:linear} and the previous step that
\begin{align}
    &
    \n{u_{\delta,\eta}^{(1)}}_{L^{\infty}(t_0,\infty;L^{n,\infty})}
    \leq
    C\left(\n{a}_{L^{n,\infty}}+
    \n{g_{\delta,\eta}}_{\dB^{-2}_{n,2}}\right)
    \leq
    C\eta, \label{eq:u1Ln}\\
    &\n{u_{\delta,\eta}^{(1)}}_{\widetilde{L^{\infty}}(t_0,\infty;\dB_{r,\sigma}^{n/r-1})}
    \leq 
    C\left(\n{a}_{\dB^{n/r-1}_{r,\sigma}}
    +
    \n{g_{\delta,\eta}}_{\dB_{r,\sigma}^{n/r-3}}\right)
    \leq
    C\eta^3. \label{eq:u1Bp}
    \end{align}
Moreover, we have $u^{(1)}_{\delta,\eta}\in \widetilde{L^{\rho}}(t_0,\infty;\dB_{r,\sigma}^{n/r-1+2/\rho}(\mathbb{R}^n))$, which follows from
    \begin{align}
    \n{u_{\delta,\eta}^{(1)}}_{\widetilde{L^{\rho}}(t_0,\infty;\dB_{r,\sigma}^{n/r-1+2/\rho})}
    &
    \leq
    C\left(
    \n{a}_{\dB^{n/r-1}_{r,\sigma}}
    +
    \n{g_{\delta,\eta}}_{\dB_{r,\sigma}^{n/r-1+2/\rho}}
    \n{\chi}_{L^{\rho}(t_0,\infty)}
    \right)
    \\
    &\leq
    C
    \eta
    \delta^{1-\frac{2r}{\rho(r-n)}}
    T_*^{\frac{1}{\rho}}
    <\infty.\label{eq:u1Lr}
\end{align}

\noindent
{\it Step 3. The second iteration.}
Next, we focus on the estimates of the second iteration
\begin{align}
    u_{\delta,\eta}^{(2)}(t,x)
    :={}&
    -
    \int_{t_0}^t
    e^{(t-\tau)\Delta}
    \mathbb{P}\div\sp{u_{\delta,\eta}^{(1)}(\tau) \otimes u_{\delta,\eta}^{(1)}(\tau)}(x)
    d\tau.
\end{align}
By Lemma \ref{lemm:Yamazaki}, we see that
\begin{align}\label{eq:u2Ln}
    \n{u_{\delta,\eta}^{(2)}}_{L^{\infty}(t_0,\infty;L^{n,\infty})}
    \leq
    C\n{u_{\delta,\eta}^{(1)}}_{L^{\infty}(t_0,\infty;L^{n,\infty})}^2
    \leq C\eta^2. 
\end{align}
We decompose $u_{\delta,\eta}^{(2)}(T_*+t_0)$ as
\begin{align}
    \quad &u_{\delta,\eta}^{(2)}(t_0+T_*)\\
    &=
    -
    \int_{t_0}^{t_0+T_*}
    e^{(t_0+T_*-\tau)\Delta}
    \mathbb{P}\div\sp{u_{\delta,\eta}^{(1)}(\tau) \otimes u_{\delta,\eta}^{(1)}(\tau)}
    d\tau\\
    &=-\left(\int_{t_0}^{t_0+h}+\int_{t_0+h}^{t_0+T_*-h}+\int_{t_0+T_*-h}^{t_0+T_*}\right)
    e^{(t_0+T_*-\tau)\Delta}
    \mathbb{P}\div\sp{u_{\delta,\eta}^{(1)}(\tau) \otimes u_{\delta,\eta}^{(1)}(\tau)}
    d\tau\\
    &=: I_1+I_2+I_3.
\end{align}
In what follows, let us estimate $I_1, I_2$, and $I_3$.

We first consider the estimate of $I_2$. 
For $t_0+h\leq t \leq t_0+T_*-h$, $u_{\delta,\eta}^{(1)}$ can be written as
\begin{align}
u_{\delta,\eta}^{(1)}(t)
    &=
    e^{(t-t_0)\Delta}a
    +
    e^{(t-t_0-h)\Delta}\int_{t_0}^{t_0+h}e^{(t_0+h-\tau)\Delta}f_{\delta,\eta}(\tau)d\tau
    +
    \int_{t_0+h}^t
    e^{(t-\tau)\Delta}g_{\delta,\eta}d\tau\\
    &=: u^{(1;1)}_{\delta,\eta}(t)+u_{\delta,\eta}^{(1;2)}(t)
   +\int_{t_0+h}^t
    e^{(t-\tau)\Delta}g_{\delta,\eta}d\tau.
\end{align}
The last term can be calculated as
\begin{align}
    \int_{t_0+h}^t
    e^{(t-\tau)\Delta}g_{\delta,\eta}d\tau
    &=(-\Delta)^{-1}(1-e^{(t-t_0-h)\Delta})g_{\delta,\eta}\\
    &=
   -e^{(t-t_0-h)\Delta}
    (-\Delta)^{-1}g_{\delta,\eta}
    +
    \eta \Psi(x)\cos\sp{\delta^{-\frac{r}{r-n}}x_1}\\
    &=: u_{\delta,\eta}^{(1;3)}(t)
    +u_{\delta,\eta}^{(1;4)}.
\end{align}
Hence $I_2$ can be described as
\begin{align}
    I_2
    &=
    -\int_{t_0+h}^{t_0+T_*-h}
    e^{(t_0+T_*-\tau)\Delta}
    \sum_{1\leq k,\ell\leq 4}
    \mathbb{P}\div\sp{u_{\delta,\eta}^{(1;k)} \otimes u_{\delta,\eta}^{(1;\ell)}}
    d\tau\\
    &=
    -\int_{t_0+h}^{t_0+T_*-h}
    e^{(t_0+T_*-\tau)\Delta}
    \mathbb{P}\div\sp{u_{\delta,\eta}^{(1;4)}\otimes u_{\delta,\eta}^{(1;4)}}
    d\tau\\
    &\qquad
    -\int_{t_0+h}^{t_0+T_*-h}
    e^{(t_0+T_*-\tau)\Delta}
    \sum_{\substack{1\leq k,\ell\leq 4 \\(k,\ell)\neq (4,4)}}
    \mathbb{P}\div\sp{u_{\delta,\eta}^{(1;k)} \otimes u_{\delta,\eta}^{(1;\ell)}}
    d\tau\\
    &=: I_{2;1}+I_{2;2}.
\end{align}
We first focus on $I_{2;1}$.
Since $u^{(1;4)}_{\delta,\eta}$ is stable for $t\in[t_0+h,t_0+T_*-h]$, we have
\begin{align}
    I_{2;1}=-(e^{h\Delta}-e^{(T_*-h)\Delta})(-\Delta)^{-1}\mathbb{P}\div\sp{u^{(1;4)}_{\delta,\eta}\otimes u^{(1;4)}_{\delta,\eta}}.
\end{align}
Moreover, it holds
\begin{align}
    \div\sp{u^{(1;4)}_{\delta,\eta}(x) \otimes u^{(1;4)}_{\delta,\eta}(x)}
    &
    ={}
    \eta^2
    \Phi(x)
    \left(
    1-
    \cos\sp{2\delta^{-\frac{r}{r-n}}x_1}
    \right),
\end{align}
where
\begin{align}\label{eq:Phi}
\Phi=\frac12\left\{\partial_{x_2}(\psi_{x_3})^2-\partial_{x_3}(\psi_{x_2}\psi_{x_3})\right\}e_2
+\frac12\left\{-\partial_{x_2}(\psi_{x_2}\psi_{x_3})+\partial_{x_3}(\psi_{x_2})^2\right\}e_3.
\end{align}
Thus, we see that 
\begin{align}
    I_{2;1}
    ={}&
    -
    \eta^2
    \sp{e^{h\Delta}-e^{(T_*-h)\Delta}}
    (-\Delta)^{-1}
    \mathbb{P}\lp{\Phi
    \left(
    1-
    \cos\sp{2\delta^{-\frac{r}{r-n}}x_1}
    \right)}.
\end{align}
Since $\delta^{-\frac{r}{r-n}}$ is large enough, we have
\begin{align}
    \Delta_j I_{2;1}=-
    \eta^2
    \Delta_j
    \sp{e^{h\Delta}-e^{(T_*-h)\Delta}}
    (-\Delta)^{-1}
    \mathbb{P}\Phi
\end{align}
for every $j\leq 2$. By \cite{Tsu-19}*{Section 3.2}, we see that
\begin{align}
    c_*
    :=
    \frac{1}{5}
    \sup_{j\leq 2}
    2^{-j}
    \n{
    \Delta_j (-\Delta)^{-1}\mathbb{P}\Phi
    }_{L^\infty}
    =
    \frac{1}{5}
    \n{(-\Delta)^{-1}\mathbb{P}\Phi}_{\dB_{\infty,\infty}^{-1}} >0.
\end{align}
In addition, it holds
\begin{align}
    2^{-j}
    \n{
    \Delta_j(1-e^{h\Delta})(-\Delta)^{-1}\mathbb{P}\Phi
    }_{L^\infty}
    &\leq (1-e^{-c2^{2j}h})2^{-j}\n{\Delta_j(-\Delta)^{-1}\mathbb{P}\Phi}_{L^\infty}\\
    &\leq 5c_*(1-e^{-ch})
\end{align}
for every $j\leq 2$, which yields the existence of a pure positive constant $h_1$ such that
\begin{align}
    &\sup_{j\leq2}
    2^{-j}\n{\Delta_j e^{h\Delta}(-\Delta)^{-1}\mathbb{P}\Phi}_{L^\infty}\\
    &\quad \geq
    \sup_{j\leq2}
    2^{-j}\n{\Delta_j (-\Delta)^{-1}\mathbb{P}\Phi}_{L^\infty}
    -
    \sup_{j\leq2}
    2^{-j}\n{\Delta_j (1-e^{h\Delta})(-\Delta)^{-1}\mathbb{P}\Phi}_{L^\infty}\\
    &\quad\geq 4c_*
\end{align}
for any $0<h<h_1$. On the other hand, since $n\geq 3$, we see that
\begin{align}
    \|(-\Delta)^{-1}\mathbb{P}\Phi\|_{\dB^{n/r-1}_{r,\sigma}}
    &\leq C\|\nabla \psi \otimes \nabla \psi\|_{\dB^{n/r-2}_{r,\sigma}}\\
    &\leq C\|\nabla \psi \otimes \nabla \psi\|_{\dB^{n/2-2}_{2,\sigma}}\\
    &\leq C\delta^2\|\nabla\psi\|^2_{\dB^{n/2-1}_{2,\sigma}}.
\end{align}
Therefore, using the density argument, we have
\begin{align}
    &
    \lim_{t \to \infty}
    \n{e^{t\Delta}(-\Delta)^{-1}\mathbb{P}\Phi}_{\dB_{r,\sigma}^{n/r-1}}
    =0,
\end{align}
which yields the existence of $T_*=T_*(n,r,\sigma) > 0$ such that 
\begin{align}
    \n{e^{(T_*-h)\Delta}(-\Delta)^{-1}\mathbb{P}\Phi}_{\dB_{r,\sigma}^{n/r-1}}
    \leq
    \frac{c_*}{2\n{\varphi_0}_{L^{r'}}},
\end{align}
where $\varphi_0 \in \mathscr{S}(\mathbb{R}^n)$ is the frequency cut-off function appearing in the Littlewood--Paley theory. 
Hence, we have
\begin{align}\label{eq:I_21}
    &\sup_{j \leq 2}
    2^{-j}
    \n{\Delta_j I_{2;1}}_{L^{\infty}}
    \\
    &\quad= {}
    \eta^2
    \sup_{j \leq 2}
    2^{-j}
    \n{\Delta_j \sp{e^{h\Delta}-e^{(T_*-h)\Delta}}(-\Delta)^{-1}\mathbb{P}\Phi}_{L^{\infty}}
    \\
    &\quad\geq 
    \eta^2\sup_{j\leq2}
    2^{-j}\n{\Delta_j e^{h\Delta}(-\Delta)^{-1}\mathbb{P}\Phi}_{L^\infty}
    -\eta^2 \n{e^{(T*-h)\Delta}(-\Delta)^{-1}\mathbb{P}\Phi}_{\dB^{-1}_{\infty,\infty}}
    \\
    & \quad \geq
    \eta^2
    \n{e^{h\Delta}(-\Delta)^{-1}\mathbb{P}\Phi}_{\dB_{\infty,\infty}^{-1}}
    -
    \eta^2
    \n{\varphi_0}_{L^{r'}}
    \n{e^{(T_*-h)\Delta}(-\Delta)^{-1}\mathbb{P}\Phi}_{\dB_{r,\sigma}^{n/r-1}}
    \\
    & \quad \geq
    3c_*\eta^2
\end{align}
with $0<h<h_1$.

Next we focus on $I_{2;2}$. 
By using Lemma \ref{cor:nonlin-ndim}, we have the following apriori estimate;
\begin{align}
    &\n{\int_{t_0+h}^{t}e^{(t_0+T_*-\tau)\Delta}\mathbb{P}\div(u(\tau)\otimes v(\tau))d\tau}_{\widetilde{{L^\infty}}(t_0+h,\infty;\dB^{n/r-1}_{r,\sigma})}\\
    &\quad\leq{}
    C
    \n{
    u
    }_{{L^{\infty}}(t_0+h,\infty;L^{n,\infty})}
    \n{
    v
    }_{\widetilde{L^{\rho}}(t_0+h,\infty;\dB_{r,\sigma}^{n/r-1+2/\rho})}
\end{align}
for all $u=u(t,x)$ and $v=v(t,x)$ such that the right-hand side is finite. 
Moreover, we see from \eqref{eq:u1Ln} and Lemma \ref{lemm:linear} that
\begin{align}
   &\n{u^{(1;k)}_{\delta,\eta}}_{L^\infty(t_0+h,\infty;L^{n,\infty})}\leq C\eta,\qquad k=1,2,3,4,\\
   &\n{u^{(1;1)}_{\delta,\eta}}_{\widetilde{L^{\rho}}(t_0+h,\infty;\dB^{n/r-1+2/\rho}_{r,\sigma})}
   \leq C\|a\|_{\dB^{n/r-1}_{r,\sigma}}
   \leq C\eta^3,\\
   &
   \begin{aligned}
          \n{u^{(1;2)}_{\delta,\eta}}_{\widetilde{L^{\rho}}(t_0+h,\infty;\dB^{n/r-1+2/\rho}_{r,\sigma})}
   &\leq 
   C\n{\int_{t_0}^{t_0+h}e^{(t_0+h-\tau)\Delta}f_{\delta,\eta}(\tau)d\tau}_{\dB^{n/r-1}_{r,\sigma}}\\
   &\leq C\|f_{\delta,\eta}\|_{\widetilde{L^\infty}(t_0,\infty;\dB^{n/r-3}_{r,\sigma})}\\
   &\leq C\delta\eta,
   \end{aligned}\\
   &\n{u^{(1;3)}_{\delta,\eta}}_{\widetilde{L^{\rho}}(t_0+h,\infty;\dB^{n/r-1+2/\rho}_{r,\sigma})}
   \leq C\n{g_{\delta,\eta}}_{\dB^{n/r-3}_{r,\sigma}}
   \leq C\delta\eta.
\end{align}
Therefore, we obtain
\begin{align}\label{eq:I_22}
    &\n{I_{2;2}}_{\dB^{n/r-1}_{r,\sigma}}\\
    &\quad
    \leq
    \sum_{\substack{1\leq k,\ell\leq 4 \\(k,\ell)\neq (4,4)}}
    \n{\int_{t_0+h}^{t}
    e^{(t_0+T_*-\tau)\Delta}
    \mathbb{P}\div\sp{u_{\delta,\eta}^{(1;k)} \otimes u_{\delta,\eta}^{(1;\ell)}}
    d\tau}_{\widetilde{{L^\infty}}(t_0+h,\infty;\dB^{n/r-1}_{r,\sigma})} \\
    &\quad \leq
    C(\delta\eta^2+\eta^4)\\
    &\quad \leq C\eta^4.
\end{align}

Next, we focus on the estimate of $I_1$ and $I_3$.
For each integer $j\leq 2$, we see that
\begin{align}
 &2^{-j}\n{\Delta_j I_1}_{L^\infty}\\
 &\quad=
 2^{-j}\n{\Delta_j \int_{t_0}^{t_0+h}e^{(t_0+T_*-\tau)\Delta}\mathbb{P}\div (u_{\delta,\eta}^{(1)}(\tau) \otimes u_{\delta,\eta}^{(1)}(\tau))d\tau}_{L^\infty}\\
 &\quad\leq
 C2^{-j}\int_{t_0}^{t_0+h}e^{-c2^{2j}(t_0+T_*-\tau)}2^j d\tau
 \sup_{t_0\leq\tau\leq t_0+h}\n{\Delta_j (u_{\delta,\eta}^{(1)}(\tau) \otimes u_{\delta,\eta}^{(1)}(\tau))}_{L^\infty}\\
 &\quad\leq
 Ce^{-c2^{2j}T_*}(e^{c2^{2j}h}-1)
 \sup_{t_0\leq\tau\leq t_0+h}2^{-2j}
 \n{\Delta_j (u_{\delta,\eta}^{(1)}(\tau) \otimes u_{\delta,\eta}^{(1)}(\tau))}_{L^\infty}\\
 &\quad\leq
 C(e^{ch}-1)\sup_{t_0\leq\tau\leq t_0+h}
 \n{u_{\delta,\eta}^{(1)}(\tau) \otimes u_{\delta,\eta}^{(1)}(\tau)}_{\dB^{-2}_{\infty,\infty}}\\
 &\quad\leq
 C(e^{ch}-1)
 \sup_{t_0\leq\tau\leq t_0+h}
 \n{u_{\delta,\eta}^{(1)}(\tau) \otimes u_{\delta,\eta}^{(1)}(\tau)}_{L^{n/2,\infty}}\\
 &\quad\leq
 C(e^{ch}-1) \n{u_{\delta,\eta}^{(1)}}_{L^{\infty}(t_0,\infty;L^{n,\infty})}^2\\
 &\quad\leq
 C\eta^2(e^{ch}-1).
\end{align}
We can also obtain a similar estimate for $I_3$.
Hence, there exists a pure positive constant $h_2$ such that
\begin{align} \label{eq:I_13}
    \sup_{j\leq 2}
    2^{-j}
    \sp{
    \n{\Delta_j I_1}_{L^\infty}
    +
    \n{\Delta_j I_3}_{L^\infty}
    }
    \leq c_*\eta^2
\end{align}
for any $0<h<h_2$.

Now we fix $h$ as $0<h<\min \{h_1,h_2\}$. Using the estimates \eqref{eq:I_21}, \eqref{eq:I_22}, and \eqref{eq:I_13}, we finally obtain the lower bound of $u^{(2)}_{\delta,\eta}(t_0+T_*)$ as
\begin{align}\label{eq:u2lower}
    &\n{u^{(2)}_{\delta,\eta}(t_0+T_*)}_{\dB^{-1}_{\infty,\infty}}\\
    &\quad\geq
    \sup_{j \leq 2}
    2^{-j}
    \n{\Delta_j u^{(2)}_{\delta,\eta}(t_0+T_*)}_{L^{\infty}}\\
    &\quad\geq
    \sup_{j \leq 2}
    2^{-j}
    \n{\Delta_j I_{2;1}}_{L^{\infty}}
    -\sup_{j \leq 2}
    2^{-j}
    \sp{
    \n{\Delta_j I_{1}}_{L^{\infty}}
    +
    \n{\Delta_j I_{3}}_{L^{\infty}}
    }
    -\n{I_{2;2}}_{\dB^{n/r-1}_{r,\sigma}}
    \\
    &\quad\geq
    3c_*\eta^2-c_*\eta^2-C\eta^4\\
    &\quad\geq c_*\eta^2.
\end{align}
by taking the sufficiently constant $\eta_0>0$.

\noindent
{\it Step 4. Construction of desired solutions.}
Now, let us construct the desired solution as the form 
$u_{\delta,\eta}=u_{\delta,\eta}^{(1)}+u_{\delta,\eta}^{(2)}+\widetilde{u}_{\delta,\eta}$,
where the remainder term $\widetilde{u}_{\delta,\eta}$ should solve
\begin{align} \label{eq:NS-tilde}
    \begin{cases}
        \begin{aligned}
        \partial_t \widetilde{u} & - \Delta \widetilde{u} 
        + 
        \displaystyle
        \sum_{1 \leq k,\ell \leq 2,\ (k,\ell) \neq (1,1)}
        \mathbb{P}\div
        \left(
        u_{\delta,\eta}^{(k)}
        \otimes
        u_{\delta,\eta}^{(\ell)}
        \right)
        \\
        &
        +
        \sum_{k=1}^2
        \mathbb{P}\div
        \left(
        u_{\delta,\eta}^{(k)}
        \otimes
        \widetilde{u}
        +
        \widetilde{u}
        \otimes
        u_{\delta,\eta}^{(k)}
        \right)
        +
        \mathbb{P}\div \sp{\widetilde{u} \otimes \widetilde{u}} = 0, 
        \end{aligned}
        \\
        \div \widetilde{u} = 0, \\
        \widetilde{u}(t_0,x)=0.
    \end{cases}
\end{align}
To simplify the notations, we set 
    \begin{align}
        X
        &
        :=
        L^{\infty}(t_0,\infty;L^{n,\infty}(\mathbb{R}^n)),
        \\
        Y_{r,\sigma,\rho}
        &
        :=
        \widetilde{C}([t_0,\infty);\dB_{r,\sigma}^{n/r-1}(\mathbb{R}^n)) \cap \widetilde{L^{\rho}}(t_0,\infty;\dB_{r,\sigma}^{n/r-1+2/\rho}(\mathbb{R}^n)).
    \end{align}
    It follows from \eqref{eq:u1Ln}, \eqref{eq:u1Bp}, \eqref{eq:u1Lr}, \eqref{eq:u2Ln}, and Lemmas \ref{lemm:nonlin-ndim}, \ref{cor:nonlin-ndim}, and \ref{lemm:Yamazaki} that
    there exists a positive constant $C_0=C_0(n,p,q,r)$ such that 
    \begin{align}
        &
        \n{u_{\delta,\eta}^{(k)}}_{X}
        \leq C_0\eta^k,
        \\
        &
        \begin{aligned}
        \n{u_{\delta,\eta}^{(k)}}_{Y_{r,\sigma,\rho}}
        &
        \leq
        \sp{C\n{u_{\delta,\eta}^{(1)}}_X}^{k-1}
        \n{u_{\delta,\eta}^{(1)}}_{Y_{r,\sigma,\rho}}
        \\
        &\leq
        C_0
        \eta^{k}
        \delta^{1-\frac{2r}{\rho(r-n)}}
    T_*^{\frac{1}{\rho}}
        \end{aligned}
    \end{align}
    for $k=1,2$
    and
    \begin{align}
        &
        \n{\int_{t_0}^t e^{(t-\tau)\Delta}\mathbb{P}\div(u \otimes v)(\tau)d\tau}_{X}
        \leq{}
        C_0
        \n{u}_X
        \n{v}_X,
        \\
        &
        \n{\int_{t_0}^t e^{(t-\tau)\Delta}\mathbb{P}\div(u \otimes v)(\tau)d\tau}_{Y_{r,\sigma,\rho}}
        \leq{}
        C_0
        \n{u}_{X}
        \n{v}_{Y_{r,\sigma,\rho}}
    \end{align}
    for all $u,v \in X \cap Y_{r,\sigma,\rho}$.
    Let us consider the mapping 
    \begin{align}
        \Psi[v](t)
        :={}&
        -
        \sum_{1 \leq k,\ell \leq 2,\ (k,\ell) \neq (1,1)}
        \int_{t_0}^t
        e^{(t-\tau)\Delta}
        \mathbb{P}\div
        \left(
        u_{\delta,\eta}^{(k)}
        \otimes
        u_{\delta,\eta}^{(\ell)}
        \right)(\tau)
        d\tau 
        \\
        &
        -
        \sum_{k=1}^2
        \int_{t_0}^t
        e^{(t-\tau)\Delta}
        \mathbb{P}\div
        \left(
        u_{\delta,\eta}^{(k)}
        \otimes
        v
        +
        v
        \otimes
        u_{\delta,\eta}^{(k)}
        \right)(\tau)
        d\tau 
        \\
        &
        -
        \int_{t_0}^t
        e^{(t-\tau)\Delta}
        \mathbb{P}
        \div(v \otimes v)(\tau)
        d\tau
    \end{align}
    on a complete metric space 
    \begin{align}
        &
        Z_{r,\sigma,\rho}
        :=
        \Mp{
        v \in X \cap Y_{r,\sigma,\rho}
        \ ;\ 
        \n{v}_{X}
        \leq 
        6C_0^3
        \eta^3,\ 
        \n{v}_{Y_{r,\sigma,\rho}}
        \leq 
        6C_0^3
        \eta^3
        \delta^{2-\frac{4r}{\rho(r-n)}}
        T_*^{\frac{2}{\rho}}
        },
        \\
        &
        d_{Z_{r,\sigma,\rho}}(u,v)
        :=
        \n{u-v}_{X \cap Y_{r,\sigma,\rho}}.
    \end{align}
    In the following, we assume that $\eta$ is so small that 
    \begin{align}
        4C_0^2\eta
        +
        24C_0^4\eta^3
        \leq
        1.
    \end{align}
    We then see by the above estimates that 
    \begin{align}
        \n{\Psi[v]}_X
        \leq {}&
        C_0
        \sum_{1 \leq k,\ell \leq 2,\ (k,\ell) \neq (1,1)}
        \n{u_{\delta,\eta}^{(k)}}_X
        \n{u_{\delta,\eta}^{(\ell)}}_X
        \\
        &
        +
        2
        C_0
        \sum_{k=1}^2
        \n{u_{\delta,\eta}^{(k)}}_X
        \n{v}_X
        +
        C_0\n{v}_X^2
        \\
        \leq{}& 
        3C_0^3\eta^3
        +
        12C_0^5\eta^4
        +
        36C_0^7\eta^6\\
        \leq{}&
        6C_0^3\eta^3,
        \\
        \n{\Psi[v]}_{Y_{r,\sigma,\rho}}
        \leq{}& 
        C_0
        \sum_{1 \leq k,\ell \leq 2,\ (k,\ell) \neq (1,1)}
        \n{u_{\delta,\eta}^{(k)}}_X
        \n{u_{\delta,\eta}^{(\ell)}}_{Y_{r,\sigma,\rho}}
        \\
        &
        +
        2
        C_0
        \sum_{k=1}^2
        \n{u_{\delta,\eta}^{(k)}}_X
        \n{v}_{Y_{r,\sigma,\rho}}
        +
        C_0\n{v}_X\n{v}_{Y_{r,\sigma,\rho}}
        \\
        \leq{}& 
        3C_0^3\eta^3
        \delta^{2-\frac{4r}{\rho(r-n)}}
        T_*^{\frac{2}{\rho}}\\
        &
        +
        6C_0^5\eta^4
        \delta^{2-\frac{4r}{\rho(r-n)}}
        T_*^{\frac{2}{\rho}}
        +
        36C_0^7\eta^6
        \delta^{2-\frac{4r}{\rho(r-n)}}
        T_*^{\frac{2}{\rho}}
        \\
        \leq{}& 
        6C_0^3\eta^3
        \delta^{2-\frac{4r}{\rho(r-n)}}
        T_*^{\frac{2}{\rho}},
    \end{align}
    and 
    \begin{align}
        \n{\Psi[v]-\Psi[w]}_{X\cap Y_{r,\sigma,\rho}}
        &
        \leq
        2
        C_0
        \sum_{k=1}^2
        \n{u_{\delta,\eta}^{(k)}}_X
        \n{v-w}_{X \cap Y_{r,\sigma,\rho}}
        \\
        &\quad
        +
        C_0\sp{\n{v}_X+\n{w}_X}\n{v-w}_{X\cap Y_{r,\sigma,\rho}}
        \\
        &
        \leq
        (2C_0\eta+12C_0^4\eta^3)\n{v-w}_{X\cap Y_{r,\sigma,\rho}}
        \\
        &
        \leq
        \frac{1}{2}\n{v-w}_{X\cap Y_{r,\sigma,\rho}}
    \end{align}
    for all $v,w \in Z_{r,\sigma,\rho}$.
    Hence, the contraction mapping principle yields the unique existence of solution $\widetilde{u}_{\delta,\eta} \in Z_{r,\sigma,\rho}$ to \eqref{eq:NS-tilde}.

    By \eqref{eq:u1Ln}, \eqref{eq:u1Bp}, \eqref{eq:u2Ln}, \eqref{eq:u2lower}, and the definition of $Z_{r,\sigma,\rho}$,
    we can choose a constant $K_0=K_0(n,r,\sigma,\rho)>0$ such that for the solution $u_{\delta,\eta}=u^{(1)}_{\delta,\eta}+u^{(2)}_{\delta,\eta}+\widetilde{u}_{\delta,\eta}$ constructed above,
    \begin{align}
        &\n{u_{\delta,\eta}}_{L^\infty(t_0,\infty;L^{n,\infty})}\\
        &\quad\leq
        \n{u^{(1)}_{\delta,\eta}}_{L^\infty(t_0,\infty;L^{n,\infty})}
        +
        \n{u^{(2)}_{\delta,\eta}}_{L^\infty(t_0,\infty;L^{n,\infty})}
        +
        \n{\widetilde u_{\delta,\eta}}_{L^\infty(t_0,\infty;L^{n,\infty})}\\
        &\quad\leq K_0\eta,
    \end{align}
    while
    \begin{align}
        &\n{u_{\delta,\eta}(t_0+T_*)}_{\dB^{-1}_{\infty,\infty}}\\
        &\quad \geq 
        \n{u^{(2)}_{\delta,\eta}(t_0+T_*)}_{\dB^{-1}_{\infty,\infty}}
        -
        \n{u^{(1)}_{\delta,\eta}}_{\widetilde{L}^\infty(t_0,\infty;\dB^{n/r-1}_{r,\sigma})}
        -
        \n{\widetilde{u}_{\delta,\eta}}_{L^\infty(t_0,\infty;L^{n,\infty})}\\
        &\quad \geq c_*\eta^2-C\eta^3-6C_0^3\eta^3 \\
        &\quad\geq K_0^{-1}\eta^2.
    \end{align}

    \noindent
    {\it Step 5. Temporal decay of the solution.}
    We see for $t>\tilde{T}>T>t_0+T_*$ that\footnote{Note that $\int_T^t e^{(t-\tau)\Delta} \mathbb{P}\div f_{\delta,\eta}(\tau)d\tau=0$ holds due to the support of $f_{\delta,\eta}$ for the time variable.} 
    \begin{align}
        u_{\delta,\eta}(t)
        =
        e^{(t-T)\Delta}u_{\delta,\eta}(T)
        -
        \int_T^t
        e^{(t-\tau)\Delta}
        \mathbb{P}
        \div
        (u_{\delta,\eta} \otimes u_{\delta,\eta})(\tau)
        d\tau,
    \end{align}
    which and Lemmas \ref{lemm:nonlin-ndim}, \ref{cor:nonlin-ndim} imply 
    \begin{align}
        &
        \n{u_{\delta,\eta}}_{\widetilde{L^{\infty}}(\tilde{T},\infty;\dB_{r,\sigma}^{n/r-1})}
        \\
        &\quad
        \leq
        \n{e^{(t-T)\Delta}u_{\delta,\eta}(T)}_{\widetilde{L^{\infty}}(\tilde{T},\infty;\dB_{r,\sigma}^{n/r-1})}
        +
        C
        \n{u_{\delta,\eta}}_{X}
        \n{u_{\delta,\eta}}_{\widetilde{L^{\rho}}(T,\infty;\dB_{r,\sigma}^{n/r-1+2/\rho})}\\
        &\quad 
        \leq
        C
        \Mp{
        \sum_{j \in \mathbb{Z}}
        \sp{
        e^{-c2^{2j}(T-\tilde{T})}
        2^{(n/r-1)j}
        \n{\Delta_j u_{\delta,\eta}(T)}_{L^r}
        }^\sigma
        }^{1/\sigma}
        \\
        &\qquad 
        +
        C
        \n{u_{\delta,\eta}}_{\widetilde{L^{\rho}}(T,\infty;\dB_{r,\sigma}^{n/r-1+2/\rho})}.
    \end{align}
    By the dominated convergence theorem, we have 
    \begin{align}
        \limsup_{\tilde{T} \to \infty}
        \n{u_{\delta,\eta}}_{\widetilde{L^{\infty}}(\tilde{T},\infty;\dB_{r,\sigma}^{n/r-1})}
        \leq 
        C
        \n{u_{\delta,\eta}}_{\widetilde{L^{\rho}}(T,\infty;\dB_{r,\sigma}^{n/r-1+2/\rho})}
        \to 0
    \end{align}
    as $T \to \infty$.
    This completes the proof in the case \eqref{case1}.

    We finally mention the case of \eqref{case2}.
    We have only to change $g_{\delta,\eta}$ as
    \begin{align}
        g_{\delta,\eta}(x)
        :=
        \frac{\eta}{\sqrt{\log N_{\delta}}}
        \Delta
        \sp{\Psi(x)\sum_{k=10}^{N_\delta} \frac{1}{\sqrt{k}}\cos(2^{k^2}x_1)},
    \end{align}
    in the above Step 1, where $N_{\delta}$ is an large integer so that $1/\sqrt{\log N_{\delta}}<\delta$.
    By a similar method of \cite{Tsu-19}, we have $\div g_{\delta,\eta}=0$ and
    \begin{align}
    \|f_{\delta,\eta}\|_{\widetilde{L^\infty}(t_0,\infty;\dB^{-2}_{n,\sigma})}
    \leq 
    \frac{\eta}{\sqrt{\log N_{\delta}}}
    \left(\sum_{k=10}^{N_{\delta}}\frac{1}{k^{\sigma/2}}\right)^{1/\sigma}
    \leq 
    \begin{cases}
    C\eta & (\sigma=2),\\
    C\delta\eta & (2<\sigma<\infty),
    \end{cases}
    \end{align}
    and we can also obtain the desired estimate of $u^{(1)}_{\delta,\eta}$.
    
    For $u^{(2)}_{\delta,\eta}$, we focus on $I_{2,1}$. We see that $u^{(1;4)}_{\delta,\eta}=(-\Delta)^{-1}g_{\delta,\eta}$ and
    \begin{align*}
    \div (u^{(1;4)}_{\delta,\eta}(x)\otimes u^{(1;4)}_{\delta,\eta}(x))
    &= \frac{\eta^2}{\log N_{\delta}} \Phi(x) \sum_{k,\ell=10}^{N_\delta} \frac{1}{\sqrt{k\ell}} \cos(2^{k^2}x_1)\cos(2^{\ell^2}x_1) \\
    &
    =: \frac{\eta^2}{2}\Phi(x)\left(J_1+J_2(x)\right), 
    \end{align*}
    where
    \begin{align*}
    &
    J_1:={}
    \frac{1}{\log N_{\delta}}\sum_{k=10}^{N_{\delta}} \frac{1}{k},\\
    &
    \begin{aligned}
    J_2(x):={}&
    \frac{1}{\log N_{\delta}}
    \sum_{k=10}^{N_{\delta}} \frac{1}{k}\cos(2^{k^2+1}x_1),\\
    &
    +\frac{1}{\log N_{\delta}}
    \sum_{\substack{10\leq k,\ell \leq N_{\delta}\\k\neq \ell}} \frac{1}{\sqrt{k\ell}}  
    \sp{
     \cos((2^{k^2}+2^{\ell^2})x_1)+\cos((2^{k^2}-2^{\ell^2})x_1)
     }.
    \end{aligned}
    \end{align*}
    Since $2^{k^2+1}$, 
    $2^{k^2}+2^{\ell^2}$ 
    and 
    $|2^{k^2}-2^{\ell^2}|$ 
    are large enough, 
    we see for any $j\leq 2$ that
    \begin{align*}
    \Delta_j 
    \left[
    \mathbb{P}\div (u^{(1;4)}_{\delta,\eta}\otimes u^{(1;4)}_{\delta,\eta})\right]
    = 
    \eta^2
    J_1
    \Delta_j
    \mathbb{P}
    \Phi.
    \end{align*}
    By an elementary calculation, we see that $1 \leq J_1 \leq 2$.
    Hence, we obtain the lower bound of $I_{2;1}$ such as
    \begin{align}
        \sup_{j\leq 2}2^{-j}\n{\Delta_j I_{2;1}}_{L^\infty} \geq c\eta^2
    \end{align}
    by using a similar method in Step 3.
    The other calculation is almost same as the case $n<p\leq\infty$; we omit the details and complete the proof.
\end{proof}

\subsection{Proof of Theorem \ref{thm:nD}}
Now, we are in a position to present the proof of Theorem \ref{thm:nD}.
\begin{proof}[Proof of Theorem \ref{thm:nD}]
Let us choose 
\begin{align}
    \eta_*:= \min\Mp{\mu,\frac{\nu}{C_*},\frac{\nu}{K_0},\eta_0}, \quad \delta_k(\varepsilon):=\varepsilon 2^{-k}.
\end{align}
Then, it follows from Proposition \ref{prop:nD} with $t_0=0$ and $a=0$ that there exist an external force $f_{\delta_1(\varepsilon),\eta_*} \in \widetilde{C}((0,\infty);\dB^{-2}_{n,2}(\mathbb{R}^n)) \cap \widetilde{C}([0,\infty);\dB_{r,\sigma}^{n/r-3}(\mathbb{R}^n))$ satisfying $\div f_{\delta_1(\varepsilon),\eta_*}=0$, $f_{\delta_1(\varepsilon),\eta_*}(0)=0$, $f_{\delta_1(\varepsilon),\eta_*}(t)=0$ for all $t \geq T_*$, and 
    \begin{align}
        \n{f_{\delta_1(\varepsilon),\eta_*}}_{\widetilde{L^{\infty}}(0,\infty;\dB^{-2}_{n,2})} 
        \leq \eta_*,
        \qquad 
        \n{f_{\delta_1(\varepsilon),\eta_*}}_{\widetilde{L^{\infty}}(0,\infty;\dot{B}^{n/r-3}_{r,\sigma})}
        \leq \eta_*\delta_1(\varepsilon),
    \end{align}
    such that 
    the forced Navier--Stokes system 
    \eqref{eq:f-NS} with $a=0$
    admits a unique solution  
    \begin{align}
        u_{\delta_1(\varepsilon),\eta_*}\in {}
        &
        C((0,\infty);L^{n,\infty}(\mathbb{R}^n))
        \\
        &
        \cap 
        \widetilde{C}([0,\infty);\dB_{r,\sigma}^{n/r-1}(\mathbb{R}^n)) \cap \widetilde{L^{\rho}}(0,\infty;\dB_{r,\sigma}^{n/r-1+2/\rho}(\mathbb{R}^n))
    \end{align}
    satisfying
    \begin{align}
        &
        \n{u_{\delta_1(\varepsilon),\eta_*}}_{L^{\infty}(0,\infty;L^{n,\infty})}
        \leq C_*\eta_*, \label{wLn:1}
        \\
        &
        \n{u_{\delta_1(\varepsilon),\eta_*}(T_*)}_{\dB_{\infty,\infty}^{-1}}
        \geq \frac{\eta_*^2}{K_0},
        \\
        &
        \lim_{t \to \infty}
        \n{u_{\delta_1(\varepsilon),\eta_*}(t)}_{\dB_{r,\sigma}^{n/r-1}}
        =
        0.\label{lim:1}
    \end{align}
    From \eqref{lim:1}, there exists a time $T_1>T_*$ such that 
    \begin{align}\label{in:T_1-1}
        \n{u_{\delta_1(\varepsilon),\eta_*}(T_1)}_{\dB_{r,\sigma}^{n/r-1}}
        \leq 
        2^{-1}\eta_*^3.
    \end{align}
    Here, we emphasis that the constant in \eqref{wLn:1} is not $K_0$ but $C_*$ even though $K_0$ may be larger than $C_*$, since $u_{\delta_1(\varepsilon),\eta_*}$ coincides the global solution constructed in Proposition \ref{prop:Yamazaki} by the uniqueness. 
    This issue is crucial in our argument as \eqref{wLn:1} implies 
    \begin{align}\label{in:T_1-2}
        \n{u_{\delta_1(\varepsilon),\eta_*}(T_1)}_{L^{n,\infty}} \leq C_*\eta_*
    \end{align}
    which is necessary to apply Proposition \ref{prop:nD} on $[T_1,\infty)$. in the next paragraph.

    Next, while paying attention to \eqref{in:T_1-1} and \eqref{in:T_1-2}, we may apply Proposition \ref{prop:nD} again for \eqref{eq:f-NS-t0} with $\eta=\eta_*$, $\delta=\delta_2$, $t_0=T_1$ and $a=u_{\delta_1(\varepsilon),\eta_*}(T_1)$.
    Then, we may construct a
    an external force $f_{\delta_2(\varepsilon),\eta_*} \in \widetilde{C}([T_1,\infty);\dB^{-2}_{n,2}(\mathbb{R}^n)) \cap \widetilde{C}([T_1,\infty);\dB_{r,\sigma}^{n/r-3}(\mathbb{R}^n))$ satisfying $\div f_{\delta_2(\varepsilon),\eta_*}=0$, $f_{\delta_2(\varepsilon),\eta_*}(T_1)=0$, $f_{\delta_2(\varepsilon),\eta_*}(t)=0$ for all $t \geq T_1+T_*$, and 
    \begin{align}
        \n{f_{\delta_2(\varepsilon),\eta_*}}_{\widetilde{{L^{\infty}}}(T_1,\infty;\dB^{-2}_{n,2})} 
        \leq \eta_*,
        \qquad 
        \n{f_{\delta_2(\varepsilon),\eta_*}}_{\widetilde{L^{\infty}}(T_1,\infty;\dot{B}^{n/r-3}_{r,\sigma})}
        \leq \eta_*\delta_2(\varepsilon),
    \end{align}
    such that
    the forced Navier--Stokes system 
    \eqref{eq:f-NS-t0} 
    admits a unique solution  
    \begin{align}
        u_{\delta_1(\varepsilon),\eta_*}\in {}
        &
        C((T_1,\infty);L^{n,\infty}(\mathbb{R}^n))
        \\
        &
        \cap 
        \widetilde{C}([T_1,\infty);\dB_{r,\sigma}^{n/r-1}(\mathbb{R}^n)) \cap \widetilde{L^{\rho}}(T_1,\infty;\dB_{r,\sigma}^{n/r-1+2/\rho}(\mathbb{R}^n))
    \end{align}
    satisfying
    \begin{align}
        &
        \n{u_{\delta_2(\varepsilon),\eta_*}}_{L^{\infty}(T_1,\infty;L^{n,\infty})}
        \leq K_0\eta_*, \label{wLn:2}
        \\
        &
        \n{u_{\delta_2(\varepsilon),\eta_*}(T_1+T_*)}_{\dB_{\infty,\infty}^{-1}}
        \geq \frac{\eta_*^2}{K_0},
        \\
        &
        \lim_{t \to \infty}
        \n{u_{\delta_2(\varepsilon),\eta_*}(t)}_{\dB_{r,\sigma}^{n/r-1}}
        =
        0.\label{lim:2}
    \end{align}
    From \eqref{lim:2}, there exists a time $T_2>T_1+T_*$ such that 
    \begin{align}\label{in:T_2-1}
        \n{u_{\delta_1(\varepsilon),\eta_*}(T_2)}_{\dB_{r,\sigma}^{n/r-1}}
        \leq 
        2^{-2}
        \eta_*^3.
    \end{align}
    Next, we show that the constant $K_0$ in \eqref{wLn:2} may be replaced by $C_*$.
    Since $u_{\delta_1(\varepsilon),\eta_*}{\bf 1}_{[0,T_1]} + u_{\delta_2(\varepsilon),\eta_*}{\bf 1}_{[T_1,\infty)}$ solves \eqref{eq:f-NS} with $a=0$ and $f=f_{\delta_1(\varepsilon),\eta_*}+ f_{\delta_2(\varepsilon),\eta_*}$ globally in time and belongs to the uniqueness class \eqref{unique}, 
    $u_{\delta_1(\varepsilon),\eta_*}{\bf 1}_{[0,T_1]} + u_{\delta_2(\varepsilon),\eta_*}{\bf 1}_{[T_1,\infty)}$ must coincide the global solution constructed in
    Proposition \ref{prop:Yamazaki}, which implies 
    \begin{align}
        \n{u_{\delta_2(\varepsilon),\eta_*}}_{L^{\infty}(T_1,\infty;L^{n,\infty})}
        &\leq
        \n{u_{\delta_1(\varepsilon),\eta_*}{\bf 1}_{[0,T_1]} + u_{\delta_2(\varepsilon),\eta_*}{\bf 1}_{[T_1,\infty)}}_{L^{\infty}(0,\infty;L^{n,\infty})}
        \\
        &\leq 
        C_*\eta_*. \label{in:T_2-2}
    \end{align}
    From \eqref{in:T_2-1} and \eqref{in:T_2-2}, we see that $a=u_{\delta_2(\varepsilon),\eta_*}(T_2)$ satisfies the assumption of Proposition \ref{prop:nD} on $[T_2,\infty)$.
    Hence, repeating the same procedure inductively, we may construct a triplet of the time sequence, external force, and corresponding solution to \eqref{eq:f-NS}:
    \begin{align}\label{f-u}
        &
        \{T_k\}_{k=0}^{\infty};\ 
        0=T_0<T_*<T_1<\cdots<T_k<T_k+T_*<T_{k+1}<\cdots,
        \\
        &
        f_{\varepsilon}
        :=
        \sum_{k=1}^{\infty}
        f_{\delta_k(\varepsilon),\eta_*}
        \in \widetilde{C}([0,\infty);\dB^{-2}_{n,2}(\mathbb{R}^n)) \cap \widetilde{C}([0,\infty);\dB_{r,\sigma}^{n/r-3}(\mathbb{R}^n)),
        \\
        &
        u_{\varepsilon}
        :=
        \sum_{k=1}^{\infty}
        u_{\delta_k(\varepsilon),\eta_*}
        {\bf 1}_{[T_{k-1},T_k]}
        \in 
        C((0,\infty);L^{n,\infty}(\mathbb{R}^n))
        \cap 
        \widetilde{C}([0,\infty);\dB_{r,\sigma}^{n/r-1}(\mathbb{R}^n)),
    \end{align}
    which satisfy 
    \begin{align}
        \n{f_{\varepsilon}}_{\widetilde{L^{\infty}}(0,\infty;\dot{B}^{n/r-3}_{r,\sigma})}
        \leq \eta_*\varepsilon,
        \qquad
        \n{f_{\varepsilon}}_{\widetilde{L^{\infty}}(T_{k-1},T_{k};\dot{B}^{n/r-3}_{r,\sigma})}
        \leq \eta_*\varepsilon2^{-k},
    \end{align}
    and 
    \begin{align}
        \n{u_{\varepsilon}}_{L^{\infty}(0,\infty;L^{n,\infty})}
        \leq C_*\eta_*,
        \quad
        \n{u_{\varepsilon}(T_k)}_{\dB_{r,\sigma}^{n/r-1}}
        \leq 
        \eta_*^32^{-k},
        \quad
        \n{u_{\varepsilon}(T_k+T_*)}_{\dB_{\infty,\infty}^{-1}}
        \geq
        \frac{\eta_*^2}{K_0}
    \end{align}
    for all $k=1,2,3,\cdots$.
    Hence, we complete the proof.
\end{proof}

\section{Two-dimensional analysis}\label{sec:2D}
In this section, we focus on Theorem \ref{thm:rough} with the case of $n=2$.
To prove this, we show the following improved result. 
\begin{thm}\label{thm:2D}
    Let $n=2$.
    Then, there exist positive constants $\varepsilon_0$ and $c_0$ such that for any $0<\varepsilon\leq \varepsilon_0$, there exists an external force $f_{\varepsilon} \in \widetilde{C}((0,\infty);\dB_{1,1}^{-1}(\mathbb{R}^2))$ such that 
    \begin{align}\label{f-2}
        \n{f_{\varepsilon}}_{\widetilde{L^{\infty}}(0,\infty;\dB_{1,1}^{-2})}\leq \varepsilon,
        \quad
        \lim_{t \to \infty}
        \n{f_{\varepsilon}(t)}_{\dB_{1,1}^{-1}}
        =
        0,
    \end{align}
    and
    \eqref{eq:f-NS} with $a=0$ and $f=f_{\varepsilon}$ admits a global solution $u_{\varepsilon} \in C([0,\infty);\dB_{1,1}^{1}(\mathbb{R}^2))$ satisfying 
    \begin{align}\label{osc-2}
        \limsup_{t \to \infty}\n{u_{\varepsilon}(t)}_{\dB_{\infty,1}^{-1}}
        \geq c_0,
        \quad
        \liminf_{t \to \infty}
        \n{u_{\varepsilon}(t)}_{\dB_{1,1}^{1}}
        =0.
    \end{align}
\end{thm}
\begin{rem}\label{rem:2D}
    Theorem \ref{thm:2D} claims a more improved result than Theorem \ref{thm:rough} with $n=2$ due to the continuous embedding
    $\dB_{1,1}^{2+s}(\mathbb{R}^2) \hookrightarrow \dB_{p,1}^{2/p+s}(\mathbb{R}^2) \hookrightarrow \dB_{\infty,1}^{s}(\mathbb{R}^2)$ with $1 \leq p \leq \infty$ and $s=-1,-3$.
\end{rem}
Before starting the rigorous proof of Theorem \ref{thm:2D},
let us mention the difference between two-dimensional and high-dimensional cases.
In the case of $n=2$, nonlinear estimate in $L^{n,\infty}(I;L^{n,\infty}(\mathbb{R}^n))$ stated in Lemma \ref{lemm:Yamazaki} fails.
Not only does the weak Lebesgue space framework break down, but it is also difficult to find a Banach space $X \subset \mathscr{S}'(\mathbb{R}^2)$ satisfying the nonlinear estimate
\begin{align}
    \n{\int_0^t e^{(t-\tau)\Delta}\mathbb{P}\div (u \otimes v) (\tau)d \tau }_{L^{\infty}(0,\infty;X)}
    \leq
    C
    \n{u}_{L^{\infty}(0,\infty;X)}
    \n{v}_{L^{\infty}(0,\infty;X)}.
\end{align}
Indeed, it seems hopeless to hold the above estimate in the critical Chemin--Lerner spaces $\widetilde{L^{\infty}}(0,\infty;\dB_{p,q}^{2/p-1}(\mathbb{R}^2))$ for all $1 \leq p,q \leq \infty$, whereas it does hold for $1 \leq p < n$ and $1 \leq q \leq \infty$ in the $n$-dimensional case with $n \geq 3$.
In Proposition \ref{prop:2D}, which corresponds to Proposition \ref{prop:nD}, instead of employing weak Lebesgue spaces for the construction of solutions, we use the nonlinear estimate given in the Lemma \ref{lemm:nonlin-2dim} stated below.\footnote{When applying Lemma \ref{lemm:nonlin-2dim} to show Proposition \ref{prop:2D} below, it is crucial that the external force has bounded support with respect to the time variable, with size at most $2^{2N}$.}
\subsection{Key proposition}
Following the same split in the above section, we establish the inductive proposition for the two-dimensional case. 
\begin{prop}\label{prop:2D}
    There exists positive constants $C_*$, $\delta_*$, and $c_*$ such that for any $N \geq 3$ and $t_0 \in \mathbb{R}$,
    there exists $f_N \in \widetilde{C}([t_0,\infty);\dB_{1,1}^{-1}(\mathbb{R}^2))$ satisfying $\div f_N = 0$, $f_N(t_0)=0$, $f_N(t)=0$ for all $t \geq t_0+2^{2N}$, and 
    \begin{align}
        \n{f_N}_{\widetilde{L^{\infty}}(t_0,\infty;\dB_{1,1}^{-1})}
        \leq
        \frac{C}{\sqrt{N}},
    \end{align}
    such that 
    for any $a \in \dB_{1,1}^1(\mathbb{R}^2)$ with $\div a= 0$ and $\n{a}_{\dB_{1,1}^1} \leq \delta_*$,
    the forced Navier--Stokes equations 
    \begin{align}\label{eq:fN-Ns}
        \begin{cases}
            \partial_t u - \Delta u + \mathbb{P} \div (u \otimes u) = f_N, & t>t_0, x \in \mathbb{R}^2,
            \\
            \div u = 0, & t \geq t_0, x \in \mathbb{R}^2, 
            \\
            u(t_0,x)=a(x), & t=t_0, x \in \mathbb{R}^2,
        \end{cases}
    \end{align}
    possesses a global solution 
    $u_N \in \widetilde{C}([t_0,\infty);\dB_{1,1}^1(\mathbb{R}^2))$ satisfying
    \begin{align} \label{eq:uNbound} 
        \lim_{t \to \infty}
        \n{u_N(t)}_{\dB_{1,1}^1}
        =0,
        \qquad
        \n{u_N(t_0+2^{2N})}_{\dB_{\infty,1}^{-1}}
        \geq 
        c_*.
    \end{align}
\end{prop}
To prove the above proposition,
we introduce a new norm
\begin{align}
    \n{f}_{X^N(I)}
    :=
    \n{f}_{\widetilde{L^{\infty}}(t_0,\infty;\dB_{1,1}^1)}
    +
    \sqrt{N}
    \n{f}_{\widetilde{L^N}(I;\dB_{1,2}^{1+2/N})}
\end{align}
for interval $I \subset \mathbb{R}$, $N \geq 1$, $f \in X^N(I):=\widetilde{L^{\infty}}(I;\dB_{1,1}^1(\mathbb{R}^2)) \cap \widetilde{L^N}(I;\dB_{1,2}^{1+2/N}(\mathbb{R}^2))$
and recall the following lemma:
\begin{lemm}[\cite{Fuj-24}*{Lemma 2.5}]\label{lemm:nonlin-2dim}
    There exists a positive constant $C$ such that 
    \begin{align}
        &
        \begin{aligned}
        &
        \n{\int_{t_0}^te^{(t-\tau)\Delta}\mathbb{P}\div (u \otimes v)(\tau)d\tau}_{\widetilde{L^{\infty}}(t_0,\infty;\dB_{1,1}^{1})}
        \\
        &
        \quad
        \leq
        CN
        \n{u}_{\widetilde{L^N}(t_0,\infty;\dB_{1,2}^{1+2/N})}
        \n{v}_{\widetilde{L^N}(t_0,\infty;\dB_{1,2}^{1+2/N})},
        \end{aligned}
        \\
        &
        \begin{aligned}
        &
        \n{\int_{t_0}^te^{(t-\tau)\Delta}\mathbb{P}\div (u \otimes v)(\tau)d\tau}_{\widetilde{L^N}(t_0,\infty;\dB_{1,2}^{1+2/N})}
        \\
        &\quad
        \leq
        C\sqrt{N}
        \n{u}_{\widetilde{L^N}(t_0,\infty;\dB_{1,2}^{1+2/N})}
        \n{v}_{\widetilde{L^N}(t_0,\infty;\dB_{1,2}^{1+2/N})},
        \end{aligned}
    \end{align}
    and in particular
    \begin{align}
        \n{\int_{t_0}^te^{(t-\tau)\Delta}\mathbb{P}\div (u \otimes v)(\tau)d\tau}_{X^N(t_0,\infty)}
        \leq
        C\n{u}_{X^N(t_0,\infty)}\n{v}_{X^N(t_0,\infty)}
    \end{align}
    for all $N \geq 3$, $t_0 \in \mathbb{R}$, and  $u,v \in X^N(t_0,\infty)$.
\end{lemm}

\begin{proof}[Proof of Proposition \ref{prop:2D}]
    Let $N \geq 3$, $0<\eta \leq 1/2$, and let $M \geq 10$ be a constant to be determined later independently of $\eta$ and $N$.
    Let $a \in \dB_{1,1}^1(\mathbb{R}^2)$ with $\div a= 0$ and $\n{a}_{\dB_{1,1}^1} \leq \eta^3$.
    We choose $\psi \in \mathscr{S}(\mathbb{R}^2)$ so that $\widehat{\psi}$ is a radial monotone non-increasing function satisfying ${\bf 1}_{\Mp{|\xi| \leq 1}} \leq \widehat{\psi} \leq {\bf 1}_{\Mp{|\xi| \leq 2}}$.
    We also use a cut-off function $\chi \in C_c^{\infty}([t_0,\infty))$ on the time line satisfying ${\bf 1}_{\Mp{t_0+1,t_0+2^{2N}-1}} \leq \widehat{\psi} \leq {\bf 1}_{\Mp{t_0,t_0+2^{2N}}}$.
    Then, we define
    \begin{align}
        f_N(t,x)
        :=
        \frac{\eta}{\sqrt{N}}
        \chi(t)g(x);
        \qquad
        g(x):=
        \Delta\nabla^{\perp}
        \sp{\psi(x)\cos(Mx_1)}.
    \end{align}
    Note that it holds $\supp \widehat{f_N(t)} \subset \Mp{\xi \in \mathbb{R}^2\ ;\ M-2 \leq |\xi| \leq M+2}$.

    To construct the desired solution $u_N$ to \eqref{eq:fN-Ns}, we decompose it into three parts: $u_N=u_N^{(1)}+u_N^{(2)}+\widetilde{u}_N$, where $u_N^{(1)}$ and $u_N^{(2)}$ are the first and second iterations defined by
    \begin{align}
        u^{(1)}_N(t)
        &:=
        \int_{t_0}^{t} e^{(t-\tau)\Delta}f_N(\tau)d\tau,
        \\
        u^{(2)}_N(t)
        &:=
        -
        \int_{t_0}^{t} e^{(t-\tau)\Delta}\mathbb{P}\div\sp{u^{(1)}_N \otimes u^{(1)}_N}(\tau)d\tau,
    \end{align}
    and $\widetilde{u}_N$ should solve
    \begin{align}\label{eq:NS-tilde-2D}
        \begin{cases}
            \begin{aligned}
            \partial_t \widetilde{u}_N
            - \Delta \widetilde{u}_N
            &
            + \displaystyle\sum_{1\leq k,\ell \leq 2, (k,\ell)\neq (1,1)} \mathbb{P} \div \sp{u_N^{(k)} \otimes u_N^{(\ell)}}\\
            &
            + \displaystyle\sum_{k=1}^2 
            \mathbb{P} \div \sp{u_N^{(k)} \otimes \widetilde{u}_N + \widetilde{u}_N \otimes u_N^{(k)}} 
            + \mathbb{P} \div \sp{\widetilde{u}_N \otimes \widetilde{u}_N}
            = 0,
            \end{aligned}
            \\
            \div \widetilde{u}_N = 0, 
            \\
            \widetilde{u}_N(t_0,x)=a(x).
        \end{cases}
    \end{align}
    For the estimate of the external force, we immediately see by the definition and Lemma \ref{lemm:cos} that 
    \begin{align}
        &
        \n{f_N}_{\widetilde{L^{\infty}}(t_0,\infty;\dB_{1,1}^{-1})}
        \leq
        \frac{\eta}{\sqrt{N}}
        \n{g}_{\dB_{1,1}^{-1}}
        \leq 
        \frac{CM^2\eta}{\sqrt{N}},
        \\
        &
        \n{f_N}_{\widetilde{L^N}(t_0,\infty;\dB_{1,2}^{-1+2/N})}
        \leq
        \frac{\eta}{\sqrt{N}}
        \n{\chi}_{L^N(t_0,\infty)}
        \n{g}_{\dB_{1,2}^{-1+2/N}}
        \leq 
        \frac{CM^3\eta}{\sqrt{N}}.
    \end{align}
    Combining this with Lemma \ref{lemm:linear}, we have the following estimates of the first iteration:
    \begin{align}
        &
        \n{u_N^{(1)}}_{\widetilde{L^{\infty}}(t_0,\infty;\dB_{1,1}^{1})}
        \leq
        C
        \n{f_N}_{\widetilde{L^{\infty}}(t_0,\infty;\dB_{1,1}^{-1})}
        \leq
        \frac{CM^2\eta}{\sqrt{N}},
        \\
        &
        \n{u_N^{(1)}}_{\widetilde{L^{N}}(t_0,\infty;\dB_{1,2}^{1+2/N})}
        \leq
        C
        \n{f_N}_{\widetilde{L^{N}}(t_0,\infty;\dB_{1,2}^{-1+2/N})}
        \leq
        \frac{CM^3\eta}{\sqrt{N}},
    \end{align}
    which implies 
    \begin{align}
        \n{u_N^{(1)}}_{X^N(t_0,\infty)}
        \leq
        CM^3\eta.
    \end{align}
    Then, making use of Lemma \ref{lemm:nonlin-2dim}, we have 
    \begin{align}
        \n{u_N^{(2)}}_{X^{N}(t_0,\infty)}
        \leq
        C
        \n{u_N^{(1)}}_{X^N(t_0,\infty)}^2
        \leq
        CM^6\eta^2.
    \end{align}
    To construct the perturbed solution, we consider a mapping
    \begin{align}
        \Psi[v](t)
        :={}&
        e^{(t-t_0)\Delta}a
        -
        \displaystyle\sum_{1\leq k,\ell \leq 2, (k,\ell)\neq (1,1)} 
        \int_{t_0}^t
        e^{(t-\tau)\Delta}
        \mathbb{P} \div \sp{u_N^{(k)} \otimes u_N^{(\ell)}}(\tau)d\tau\\
        &
        - 
        \displaystyle\sum_{k=1}^2 
        \int_{t_0}^t
        e^{(t-\tau)\Delta}
        \mathbb{P} \div \sp{u_N^{(k)} \otimes v + v\otimes u_N^{(k)} }(\tau)d\tau  
        \\
        &
        -
        \int_{t_0}^t
        e^{(t-\tau)\Delta}
        \mathbb{P} \div \sp{v \otimes v}(\tau)d\tau
    \end{align}
    on a complete metric space 
    \begin{align}
        &
        Y^N
        :=
        \Mp{v \in X^N(t_0,\infty)\ ;\ \n{v}_{X^N(t_0,\infty)}\leq 2C_0M^{12}\eta^3},
        \\
        &
        d_{Y^N}(v,w):=\n{v-w}_{X^N(t_0,\infty)},
    \end{align}
    where the positive constant $C_0$ is chosen so that the following estimates hold:
    \begin{align}
        &
        \n{e^{(t-t_0)\Delta}a}_{X^N(t_0,\infty)}
        +
        \displaystyle\sum_{\substack{1\leq k,\ell \leq 2, \\(k,\ell)\neq (1,1)}} 
        \n{\int_{t_0}^t
        e^{(t-\tau)\Delta}
        \mathbb{P} \div \sp{u_N^{(k)} \otimes u_N^{(\ell)}}(\tau)d\tau}_{X^N(t_0,\infty)}
        \\
        &\quad 
        \leq
        C\n{a}_{\dB_{1,1}^1}
        +
        C
        \displaystyle\sum_{\substack{1\leq k,\ell \leq 2, \\(k,\ell)\neq (1,1)}}
        \n{u_N^{(k)}}_{X^N(t_0,\infty)}
        \n{u_N^{(\ell)}}_{X^N(t_0,\infty)}
        \leq
        C_0M^{12}\eta^3,
        \\
        &
        \displaystyle\sum_{k=1}^2 
        \n{\int_{t_0}^t
        e^{(t-\tau)\Delta}
        \mathbb{P} \div \sp{u_N^{(k)} \otimes v + v\otimes u_N^{(k)} }(\tau)d\tau}_{X^N(t_0,\infty)}
        \\
        &\quad 
        \leq
        C
        \displaystyle\sum_{k=1}^2
        \n{u_N^{(k)}}_{X^N(t_0,\infty)}
        \n{v}_{X^N(t_0,\infty)}
        \leq
        C_0M^6\eta\n{v}_{X^N(t_0,\infty)},
        \\
        &
        \n{\int_{t_0}^t
        e^{(t-\tau)\Delta}
        \mathbb{P} \div \sp{v \otimes v}(\tau)d\tau}_{X^N(t_0,\infty)}
        \leq
        C_0\n{v}_{X^N(t_0,\infty)}^2.
    \end{align}
    Choosing $\eta$ so small that 
    \begin{align}
        2C_0M^{6}\eta
        +
        8C_0^2M^{12}\eta^3
        \leq 1,
    \end{align}
    we then see by the above estimates that
    \begin{align}
        \n{\Psi[v]}_{X^N(t_0,\infty)}
        \leq {}&
        C_0M^{12}\eta^3 
        + 
        C_0M^6\eta\n{v}_{X^N(t_0,\infty)}
        +
        C_0\n{v}_{X^N(t_0,\infty)}^2
        \\
        \leq {}&
        C_0M^{12}\eta^3 
        + 
        2C_0^2M^{18}\eta^4
        +
        4C_0^3M^{24}\eta^6
        \\
        \leq {}&
        2C_0M^{12}\eta^3
    \end{align}
    and
    \begin{align}
        &
        \n{\Psi[v]-\Psi[w]}_{X^N(t_0,\infty)}
        \\
        &\quad 
        \leq {}
        \sp{
        C_0
        M^6\eta
        +
        C_0\n{v}_{X^N(t_0,\infty)}
        +
        C_0\n{w}_{X^N(t_0,\infty)}
        }
        \n{v-w}_{X^N(t_0,\infty)}
        \\
        &
        \quad 
        \leq{}
        \sp{
        C_0
        M^6\eta
        +
        4C_0^2M^{12}\eta^3
        }
        \n{v-w}_{X^N(t_0,\infty)}
        \\
        &
        \quad
        \leq
        \frac{1}{2}
        \n{v-w}_{X^N(t_0,\infty)}
    \end{align}
    for all $v,w \in X^N(t_0,\infty)$.
    Hence, from the contraction mapping principle, we find a unique element $\widetilde{u}_N \in X^N(t_0,\infty)$ solving $\widetilde{u}_N=\Psi[\widetilde{u}_N]$, which means $\widetilde{u}_N$ is a mild solution to \eqref{eq:NS-tilde-2D}.
    Thus, we have a solution $u_N=u^{(1)}_N+u^{(2)}_N+\widetilde{u}_N$ to \eqref{eq:fN-Ns}.

    In what follows, we shall show that the solution $u_N$ constructed above satisfies the desired properties \eqref{eq:uNbound}.
    On the decay of the solution as $t \to \infty$,
    we see for $t>\tilde{T}>T>t_0+2^{2N}$ that\footnote{Note that $\int_T^t e^{(t-\tau)\Delta} \mathbb{P}\div f_N(\tau)d\tau=0$ holds due to the support of $f_{\delta,\eta}$ for the time variable.} 
    \begin{align}
        u_N(t)
        =
        e^{(t-T)\Delta}u_N(T)
        -
        \int_T^t
        e^{(t-\tau)\Delta}
        \mathbb{P}
        \div
        (u_N \otimes u_N)(\tau)
        d\tau,
    \end{align}
    which and Lemma \ref{lemm:nonlin-2dim} imply
    \begin{align}
        &
        \n{u_N}_{\widetilde{L^{\infty}}(\tilde{T},\infty;\dB_{1,1}^{1})}\\
        &\quad\leq
        \n{e^{(t-T)\Delta}u_N(T)}_{\widetilde{L^{\infty}}(\tilde{T},\infty;\dB_{1,1}^{1})}
        +
        CN
        \n{u_N}_{\widetilde{L^N}(T,\infty;\dB_{1,2}^{1+2/N})}^2\\
        &\quad 
        \leq
        C
        \sum_{j \in \mathbb{Z}}
        e^{-c2^{2j}(T-\tilde{T})}
        2^{(1+2/N)j}
        \n{\Delta_j u_N(T)}_{L^1}
        +
        CN
        \n{u_N}_{\widetilde{L^N}(T,\infty;\dB_{1,2}^{1+2/N})}^2.
    \end{align}
    Thus, by the dominated convergence theorem, we have 
    \begin{align}
        \limsup_{\tilde{T} \to \infty}
        \n{u_N}_{\widetilde{L^{\infty}}(\tilde{T},\infty;\dB_{1,1}^{1})}
        \leq 
        CN
        \n{u_N}_{\widetilde{L^N}(T,\infty;\dB_{1,2}^{1+2/N})}^2
        \to 0
    \end{align}
    as $T \to \infty$.
    Finally, we prove the uniform boundedness for $N$ from below.
    We decompose $u_N^{(2)}$ as 
    \begin{align}
        u_N^{(2)}(t)
        ={}&
        -
        \frac{\eta^2}{N}
        \int_{t_0+1}^{t_0+2^{2N}-1}
        e^{(t-\tau)\eta}\mathbb{P}\div\sp{\pnabla\Theta \otimes \pnabla\Theta}d\tau\\
        &
        -
        \int_{[t_0,t]\setminus [t_0+1,t_0+2^{2N}-1]}
        e^{(t-\tau)\eta}
        \mathbb{P}
        \div
        \sp{u_N^{(1)}\otimes u_N^{(1)}}(\tau)
        d\tau\\
        &
        -
        \frac{\eta^2}{N}
        \int_{t_0}^t
        e^{(t-\tau)\eta}\mathbb{P}\div
        \left(
        e^{\tau\Delta}\pnabla\Theta \otimes \pnabla\Theta
        \right.
        \\
        &\qquad \qquad \qquad
        \left.
        +
        \pnabla\Theta \otimes e^{\tau\Delta}\pnabla\Theta
        +
        e^{\tau\Delta}\pnabla\Theta \otimes e^{\tau\Delta}\pnabla \Theta\right)d\tau\\
        =:{}&
        u_N^{(2;1)}(t)+u_N^{(2;2)}(t)+u_N^{(2;3)}(t),
    \end{align}
    where $\Theta(x):=\psi(x)\cos(Mx_1)$.
    For the estimate of $u_N^{(2;1)}(t)$, we see that 
    \begin{align}
        &\mathscr{F}
        \lp{u_N^{(2;1)}(t_0+2^{2N})}(\xi)
        \\
        &\quad 
        ={}
        -
        \frac{\eta^2}{N}
        \int_{t_0+1}^{t_0+2^{2N}-1}
        e^{-(t_0+2^{2N}-\tau)|\xi|^2}
        d\tau
        \mathscr{F}
        \lp{\mathbb{P}
        \div
        \sp{\Theta \otimes \Theta}}(\xi)\label{F-w^(2;1)}\\
        &\quad 
        ={}
        -
        \frac{\eta^2}{N}
        |\xi|^{-2}
        e^{-|\xi|^2}\sp{1-e^{-(2^{2N}-2h)|\xi|^2}}
        \mathscr{F}
        \lp{\mathbb{P}
        \div
        \sp{\Theta \otimes \Theta}}(\xi),
    \end{align}
    Here, since the direct calculation yields
    \begin{align}\label{direct-1}
    &
    \operatorname{div}( \nabla^{\perp}\Theta \otimes \nabla^{\perp}\Theta )
    \\
    &\quad={}
    \frac{M^2}{2}
    \begin{pmatrix}
        0 \\
        \partial_{x_2}\sp{\psi(x)^2}
    \end{pmatrix}
    + \frac{1}{2}\sp{ 1 + \cos(2Mx_1)}\operatorname{div}( \nabla^{\perp}\psi \otimes \nabla^{\perp}\psi )\\
    &\qquad
    - \frac{M}{2}\sin(2Mx_1)
    \begin{pmatrix}
        \sp{\partial_{x_2}\psi(x)}^2 - \psi(x) \partial_{x_2}^2 \psi(x) \\
        \psi(x) \partial_{x_1} \partial_{x_2}\psi(x) - \partial_{x_1}\psi(x) \partial_{x_2}\psi(x)
    \end{pmatrix}
    \end{align}
    and the Fourier transforms of the terms with $\cos(2Mx_1)$ and $\sin(2Mx_1)$ in the above right hand side are supported in $\{\xi \in \mathbb{R}^2 \ ;\ 2(M-2) \leq |\xi| \leq 2(M+2)\}$,
    it holds 
    \begin{align}
        &
        \mathscr{F}
        \lp{\mathbb{P}\operatorname{div}( \nabla^{\perp}\Theta \otimes \nabla^{\perp}\Theta )}(\xi)\\
        &\quad
        ={}
        \frac{M^2}{2}
        \mathscr{F}\lp{
        \mathbb{P}
        \begin{pmatrix}
            0 \\ \partial_{x_2}(\psi^2)
        \end{pmatrix}
        }(\xi)
        +
        \frac{1}{2}
        \mathscr{F}
        \lp{\mathbb{P}\operatorname{div} ( \nabla^{\perp} \psi \otimes \nabla^{\perp} \psi )}(\xi)\\
        &\quad
        ={}
        \frac{M^2}{2}
        \mathscr{F}\lp{
        \begin{pmatrix}
            \partial_{x_1}\partial_{x_2}^2(-\Delta)^{-1}(\psi^2) \\ 
            \left(1+\partial_{x_2}^2(-\Delta)^{-1}\right)\partial_{x_2}(\psi^2)
        \end{pmatrix}
        }(\xi)\\
        &\qquad
        +
        \frac{1}{2}
        \mathscr{F}
        \lp{\mathbb{P}\operatorname{div} ( \nabla^{\perp} \psi \otimes \nabla^{\perp} \psi )}(\xi)
    \end{align}
    for all $\xi \in \mathbb{R}^2$ with $|\xi| \leq 2$.
    Thus, we have 
    \begin{align}\label{est:Theta}
        \begin{split}
        &
        \abso{\mathscr{F}
        \lp{\mathbb{P}\operatorname{div}( \nabla^{\perp}\Theta \otimes \nabla^{\perp}\Theta )}(\xi)}\\
        &\quad
        \geq{}
        \frac{M^2}{2}
        \abso{\mathscr{F}
        \lp{ 
        \left(1+\partial_{x_2}^2(-\Delta)^{-1}\right)\partial_{x_2}(\psi^2)
        }(\xi)}\\
        &\qquad
        -
        \frac{1}{2}
        \abso{\mathscr{F}
        \lp{ \mathbb{P} \operatorname{div} ( \nabla^{\perp} \psi \otimes \nabla^{\perp} \psi )}(\xi)}\\
        &\quad
        \geq {}
        \frac{M^2}{2}
        \sp{1-\frac{\xi_2^2}{|\xi|^2}}
        |\xi_2|
        \sp{\widehat{\psi}*\widehat{\psi}}(\xi)
        -
        C|\xi|
        \abso{\mathscr{F}
        \lp{  \nabla^{\perp} \psi \otimes \nabla^{\perp} \psi  }(\xi)}\\
        &\quad
        \geq {}
        c
        M^2
        2^j
        \sp{\widehat{\psi}*\widehat{\psi}}(\xi)
        -
        C2^j
        \abso{\mathscr{F}
        \lp{  \nabla^{\perp} \psi \otimes \nabla^{\perp} \psi  }(\xi)}
        \end{split}
    \end{align}
    for all $\xi \in A_j:=\Mp{\xi \in \mathbb{R}^2 \ ;\ 2^{j-1} \leq |\xi| \leq 2^{j+1}, \ {|\xi|}/{2} \leq |\xi_2| \leq {|\xi|}/{\sqrt{2}}}$ with $j \leq 0$.
    Hence, it holds by \eqref{F-w^(2;1)} that
    \begin{align}
        &
        \n{u_N^{(2;1)}(t_0+2^{2N})}_{\dB_{\infty,1}^{-1}}
        \geq {}
        \sum_{j \leq 0}2^{-j}\n{\Delta_j u_N^{(2;1)}(t_0+2^{2N})}_{L^{\infty}}\\
        &\quad
        \geq {}
        \sum_{j \leq 0}
        \n{e^{-2^{2j}|x|^2}\Delta_j u_N^{(2;1)}(t_0+2^{2N})}_{L^2}\\
        &\quad
        \geq {}
        c
        \sum_{j \leq 0}
        \n{e^{2^{2j}\Delta_{\xi}}\lp{
        \widehat{\varphi}(2^{-j}\cdot)\mathscr{F}\lp{u_N^{(2;1)}(t_0+2^{2N})}}(\xi)}_{L^2(A_j)}
        \\
        &\quad 
        \geq {}
        \frac{cM^2\eta^2}{N}
        \sum_{-N \leq j \leq 0}
        2^{-j}\sp{ 1 - e^{-c2^{2(j+N)}} }
        \n{e^{2^{2j}\Delta_{\xi}}\lp{\widehat{\varphi}(2^{-j}\cdot)\sp{\widehat{\psi}*\widehat{\psi}}}(\xi)}_{L^2(A_j)}
        \\
        &\qquad
        -
        \frac{C\eta^2}{N}
        \sum_{-N \leq j \leq 0}
        2^{-j}
        \n{\Delta_j\sp{\nabla^{\perp} \psi \otimes \nabla^{\perp} \psi }}_{L^2(\mathbb{R}^2)}
        \\
        &\quad 
        \geq {}
        \frac{cM^2\eta^2}{N}
        \sum_{-N \leq j \leq 0}
        2^{-j}
        \n{e^{2^{2j}\Delta_{\xi}}\lp{\widehat{\varphi}(2^{-j}\cdot)\sp{\widehat{\psi}*\widehat{\psi}}}(\xi)}_{L^2(A_j)}
        \\
        &\qquad
        -
        C\eta^2\n{\nabla^{\perp} \psi \otimes \nabla^{\perp} \psi }_{\dB_{2,\infty}^{-1}}.
    \end{align}
    It follows from
    \begin{align}
        &
        \sp{\widehat{\psi}*\widehat{\psi}}(2^j\xi)
        ={}
        \int_{|\zeta| \leq 2}
        \widehat{\psi}(2^j\xi-\zeta)\widehat{\psi}(\zeta)
        d\zeta
        \geq{}
        \int_{|\zeta| \leq 1/2}
        d\eta
        >0,
        \\
        &
    e^{2^{2j}\Delta_{\xi}}\lp{\widehat{\varphi}(2^{-j}\cdot)\sp{\widehat{\psi}*\widehat{\psi}}}(\xi)
    \\
    &\quad=
    c2^{-2j}
    \int_{A_j}
    e^{-2^{-2j-2}|\xi-\zeta|^2}
    \widehat{\phi_0}(2^{-j}\zeta)(\widehat{\psi}*\widehat{\psi})(\zeta)d\zeta\\
    &\quad=
    c
    \int_{A_0}
    e^{-2^{-2}|2^{-j}\xi-\widetilde{\zeta}|^2}
    \widehat{\phi_0}(\widetilde{\zeta})(\widehat{\psi}*\widehat{\psi})(2^j\widetilde{\zeta})d\widetilde{\zeta} \qquad (\widetilde{\zeta}=2^{-j}\zeta)\\
    &\quad 
    \geqslant
    c
    \int_{A_0}
    e^{-2^{-2}|2^{-j}\xi-\widetilde{\zeta}|^2}
    \widehat{\phi_0}(\widetilde{\zeta})d\widetilde{\zeta}
    \\
    &\quad 
    \geqslant
    c
    \int_{A_0}
    \widehat{\phi_0}(\widetilde{\zeta})d\widetilde{\zeta}
    =
    c,
    \\
        &
        \n{\nabla^{\perp} \psi \otimes \nabla^{\perp} \psi }_{\dB_{2,\infty}^{-1}}
        \leq
        C
        \n{\nabla^{\perp} \psi \otimes \nabla^{\perp} \psi }_{L^1}
        \leq
        C
        \n{\psi}_{\dot{H}^1}^2
    \end{align}
    that 
    \begin{align}\label{low-w^(2;1)}
        \n{u_N^{(2;1)}(t_0+2^{2N})}_{\dB_{1,1}^{1}}
        \geq {}
        c_1M^2\eta^2
        -
        C_1\eta^2\n{\psi}_{\dot{H}^1}^2
    \end{align}
    for some positive constants $c_1$ and $C_1$ independent of $\eta$, $N$, and $M$.
    For the estimate of $u_N^{(2;2)}(t)$, we have 
    \begin{align}
        &\n{\Delta_ju_N^{(2;2)}(t_0+2^{2N})}_{L^1}
        \\
        &\quad\leq{}
        C
        \sp{\int_{t_0}^{t_0+1}+\int_{t_0+2^{2N}-1}^{t_0+2^{2N}}}
        e^{-c(t-\tau)2^{2j}}
        2^j
        d\tau
        \n{\Delta_j\sp{u_N^{(1)}\otimes u_N^{(1)}}}_{L^{\infty}(t_0,t_0+2^{2N};L^1)}
        \\
        &\quad\leq{}
        C
        \sp{ 1 - e^{-c2^{2j}} }
        2^{-j}
        \n{\Delta_j\sp{u_N^{(1)}\otimes u_N^{(1)}}}_{L^{\infty}(t_0,t_0+2^{2N};L^1)}
        \\
        &\quad\leq{}
        C
        2^{j}
        \n{\Delta_j\sp{u_N^{(1)}\otimes u_N^{(1)}}}_{L^{\infty}(t_0,t_0+2^{2N};L^1)},
    \end{align}
    which implies
    \begin{align}\label{up-w^(2;2)}
        \n{u_N^{(2;2)}(t_0+2^{2N})}_{\dB_{1,1}^{1}}
        \leq{}&
        C
        \n{u_N^{(1)} \otimes u_N^{(1)}}_{\widetilde{L^{\infty}}(t_0,t_0+2^{2N};\dB_{1,1}^{2})}\\
        \leq{}&
        C
        \n{u_N^{(1)}}_{\widetilde{L^{\infty}}(t_0,t_0+2^{2N};\dB_{1,1}^{2})}
        \\
        \leq{}&
        C\n{\Theta}_{\dB_{1,1}^{0}}^2
        \leq{}
        C\frac{M^{6}\eta^2}{N}.
    \end{align}
    Using Lemmas \ref{lemm:linear} and \ref{lemm:nonlin}, we have
    \begin{align}\label{up-w^(2;3)}
        \n{u_N^{(2;3)}}_{\widetilde{L^{\infty}}(t_0,t_0+2^{2N};\dB_{1,1}^{1})}
        \leq{}&
        C
        \n{\div\sp{e^{\tau\Delta}\pnabla\Theta \otimes \pnabla\Theta}}_{\widetilde{L^2}(t_0,t_0+2^{2N};\dB_{1,1}^{0})}\\
        &+
        C
        \n{\div\sp{\pnabla\Theta \otimes e^{\tau\Delta}\pnabla\Theta}}_{\widetilde{L^2}(t_0,t_0+2^{2N};\dB_{1,1}^{0})}\\
        &+
        C
        \n{\div\sp{e^{\tau\Delta}\pnabla\Theta \otimes e^{\tau\Delta}\pnabla \Theta}}_{\widetilde{L^2}(t_0,t_0+2^{2N};\dB_{1,1}^{0})}\\
        \leq{}&
        C
        \n{\pnabla\Theta}_{\dB_{1,1}^{1}} 
        \n{e^{\tau\Delta}\pnabla\Theta}_{\widetilde{L^2}(t_0,t_0+2^{2N};\dB_{1,1}^{0})}
        \\
        &
        +
        C\n{e^{\tau\Delta}\pnabla\Theta}_{\widetilde{L^2}(t_0,t_0+2^{2N};\dB_{1,1}^{0})}^2
        \\
        \leq{}&
        C
        \n{\pnabla\Theta}_{\dB_{1,1}^{1}}^2\\
        \leq{}&
        C\frac{M^{4}\eta^2}{N}.
    \end{align}
    Hence, it follows from 
    \begin{align}
        u_N
        ={}&
        u_N^{(1)}
        +
        u_N^{(2)}
        +
        \widetilde{u}_{N}\\
        ={}&
        u_N^{(1)}
        +
        u_N^{(2;1)}
        +
        u_N^{(2;2)}
        +
        u_N^{(2;3)}
        +
        \widetilde{u}_{N},
    \end{align}
    \eqref{low-w^(2;1)}, \eqref{up-w^(2;2)}, and \eqref{up-w^(2;3)} that 
    \begin{align}
        &\n{u_N(t_0+2^{2N})}_{\dB_{1,1}^{1}}
        \\
        &\quad
        \geq{}
        \n{u_N^{(2;1)}(t_0+2^{2N})}_{\dB_{1,1}^{1}}
        -
        \n{u_N^{(2;2)}(t_0+2^{2N})}_{\dB_{1,1}^{1}}
        -
        \n{u_N^{(2;3)}}_{\widetilde{L^{\infty}}(t_0,t_0+2^{2N};\dB_{1,1}^{1})}
        \\
        &\qquad-
        \n{u_N^{(1)}}_{\widetilde{L^{\infty}}(t_0,t_0+2^{2N};\dB_{1,1}^{1})}
        -
        \n{\widetilde{u}_{N}}_{X^N(t_0,\infty)}\\
        &\quad\geq{}
        c_1M^2\eta^2
        -
        C_1\eta^2\n{\psi}_{\dot{H}^1}^2
        -
        C_2\frac{M^{6}\eta^2}{N}
        \\
        &\qquad-
        C_2\frac{M\eta}{\sqrt{N}}
        -
        2C_0M^{12}\eta^3\\
        &\quad
        \geq{}
        c_1M^2\eta^2
        -
        C_1\eta^2\n{\psi}_{\dot{H}^1}^2
        -
        C_2M^{6}\eta^4
        -
        C_2M\eta^3
        -
        2C_0M^{12}\eta^3.
    \end{align}
    for some positive constant $C_2$.
    Now, we choose $M$ and re-choose $\eta$ so  that
    \begin{gather}
        c_1M^2
        -
        C_1\n{\psi}_{\dot{H}^1}^2
        \geq 2, \\
        C_2M^{6}\eta^2
        +
        C_2M\eta
        +
        2C_0M^{12}\eta
        \leq 
        1.
    \end{gather}
    Then, we have 
    \begin{align}
        \n{u_N(t_0+2^{2N})}_{\dB_{1,1}^{1}}
        \geq \eta^2.
    \end{align}
    Hence, we complete the proof.
\end{proof}
\subsection{Proof of Theorem \ref{thm:2D}}
Now, we are in a position to prove Theorem \ref{thm:2D}.
\begin{proof}[Proof of Theorem \ref{thm:2D}]
Let $0 < \varepsilon \leq C_*/\sqrt{3}$ and $N_{k,\varepsilon}=(kC_*)^2/\varepsilon^2$ ($k \in \mathbb{N}$).
Using Proposition \ref{prop:2D} with $t_0=0$ and $N=N_{1,\varepsilon}$, we have the external force $f_{N_{1,\varepsilon}} \in \widetilde{C}([0,\infty);\dB_{1,1}^{-1}(\mathbb{R}^2))$ with $f_{N_{1,\varepsilon}}(0)=0$ and $f_{N_{1,\varepsilon}}(t)=0$ for all $t \geq T_{1,\varepsilon}$ and solution $u_{N_{1,\varepsilon}} \in \widetilde{C}([0,\infty);\dB_{1,1}^1(\mathbb{R}^2))$ satisfying 
\begin{align}
    \n{f_{N_{1,\varepsilon}}}_{\widetilde{L^{\infty}}(0,\infty;\dB_{1,1}^{-1})}
    \leq \varepsilon,
    \qquad
    \lim_{t \to \infty}\n{u_{N_{1,\varepsilon}}(t)}_{\dB_{1,1}^1}=0,
    \qquad
    \n{u_{N_{1,\varepsilon}}(T_{1,\varepsilon})}_{\dB_{\infty,1}^{-1}}
    \geq c_*,
\end{align}
where $T_{1,\varepsilon}:=2^{2N_{1,\varepsilon}}$.
Let us choose $\widetilde{T}_{1,\varepsilon}>T_{1,\varepsilon}$ so small that
\begin{align}
    \n{u_{1,\varepsilon}(\widetilde{T}_{1,\varepsilon})}_{\dB_{1,1}^1} \leq \delta_*.
\end{align}
Next, we apply Proposition \ref{prop:2D} with $t_0=\widetilde{T}_{1,\varepsilon}$, $a=u_{1,\varepsilon}(\widetilde{T}_{1,\varepsilon})$, and $N=N_{2,\varepsilon}$, we have the external force $f_{N_{2,\varepsilon}} \in \widetilde{C}([\widetilde{T}_{1,\varepsilon},\infty);\dB_{1,1}^{-1}(\mathbb{R}^2))$ with $f_{N_{2,\varepsilon}}(\widetilde{T}_{1,\varepsilon})=0$ and $f_{N_{2,\varepsilon}}(t)=0$ for all $t \geq T_{2,\varepsilon}$ and solution $u_{N_{2,\varepsilon}} \in \widetilde{C}([\widetilde{T}_{1,\varepsilon},\infty);\dB_{1,1}^1(\mathbb{R}^2))$ satisfying 
\begin{align}
    \n{f_{N_{2,\varepsilon}}}_{\widetilde{L^{\infty}}(\widetilde{T}_{1,\varepsilon},\infty;\dB_{1,1}^{-1})}
    \leq \frac{\varepsilon}{2},
    \qquad
    \lim_{t \to \infty}\n{u_{N_{2,\varepsilon}}(t)}_{\dB_{1,1}^1}=0,
    \qquad
    \n{u_{N_{2,\varepsilon}}(T_{2,\varepsilon})}_{\dB_{\infty,1}^{-1}}
    \geq c_*,
\end{align}
where $T_{2,\varepsilon}:=\widetilde{T}_{1,\varepsilon}+2^{2N_{2,\varepsilon}}$.
Let us choose $\widetilde{T}_{2,\varepsilon}>T_{2,\varepsilon}$ so small that
\begin{align}
    \n{u_{2,\varepsilon}(\widetilde{T}_{2,\varepsilon})}_{\dB_{1,1}^1} \leq \frac{\delta_*}{2}.
\end{align}
Repeating the same procedure inductively, we have two time sequences $\{T_{k,\varepsilon}\}_{k=1}^{\infty}$ and $\{\widetilde{T}_{k,\varepsilon}\}_{k=0}^{\infty}$, a family of external forces $f_{N_{k,\varepsilon}} \in \widetilde{C}([\widetilde{T}_{k-1,\varepsilon},\infty);\dB_{1,1}^{-1}(\mathbb{R}^2))$, 
and 
corresponding
solutions $u_{N_{k,\varepsilon}} \in \widetilde{C}([\widetilde{T}_{k-1,\varepsilon},\infty);\dB_{1,1}^1(\mathbb{R}^2))$ satisfying
\begin{align}
    &
    \widetilde{T}_{k,\varepsilon}>T_{k,\varepsilon} = \widetilde{T}_{k-1,\varepsilon} + 2^{2N_{k,\varepsilon}},
    \qquad \widetilde{T}_{0,\varepsilon}=0,
    \\
    &
    f_{N_{k,\varepsilon}}(\widetilde{T}_{k-1,\varepsilon})
    =0,
    \quad
    f_{N_{k,\varepsilon}}(t)
    =0\quad {\rm for\ all} \quad t \geq T_{k,\varepsilon},
    \quad 
    \n{f_{N_{k,\varepsilon}}}_{\widetilde{L^{\infty}}(\widetilde{T}_{k-1,\varepsilon},\infty;\dB_{1,1}^{-1})}
    \leq \frac{\varepsilon}{k},
    \\
    &
    u_{N_{k-1,\varepsilon}}(\widetilde{T}_{k-1,\varepsilon})
    =
    u_{N_{k,\varepsilon}}(\widetilde{T}_{k-1,\varepsilon}),
    \quad 
    \n{u_{N_{k,\varepsilon}}(T_{k,\varepsilon})}_{\dB_{\infty,1}^{-1}}
    \geq c_*,
    \quad
    \n{u_{k,\varepsilon}(\widetilde{T}_{k,\varepsilon})}_{\dB_{1,1}^1} \leq \frac{\delta_*}{k}.
\end{align}
Letting
\begin{align}
    &
    f_{\varepsilon}:= \sum_{k=1}^{\infty}f_{N_{k,\varepsilon}}
    \in \widetilde{C}([0,\infty);\dB_{1,1}^{-1}(\mathbb{R}^2)),
    \\
    &
    u_{\varepsilon}:= \sum_{k=1}^{\infty}u_{N_{k,\varepsilon}}{\bf 1}_{[\widetilde{T}_{k-1,\varepsilon},\widetilde{T}_{k,\varepsilon}]}
    \in \widetilde{C}([0,\infty);\dB_{1,1}^{1}(\mathbb{R}^2)),
\end{align}
we see that $(f_{\varepsilon},u_{\varepsilon})$ solves \eqref{eq:f-NS} and satisfies
\begin{align}
    \n{f}_{\widetilde{L^{\infty}}(\widetilde{T}_{k-1,\varepsilon},\widetilde{T}_{k,\varepsilon};\dB_{1,1}^1)}
    \leq\frac{\varepsilon}{k},
    \quad
    \n{u_{\varepsilon}(T_{k,\varepsilon})}_{\dB_{\infty,1}^{-1}}
    \geq c_*,
    \quad 
    \n{u_{\varepsilon}(\widetilde{T}_{k,\varepsilon})}_{\dB_{1,1}^1} \leq \frac{\delta_*}{k}.
\end{align}
Hence, we complete the proof.
\end{proof}

\noindent
{\bf Data availability.} \\
Data sharing not applicable to this article as no dataset was generated or analyzed during the current study.

\noindent
{\bf Conflict of interest.} \\
The authors have declared no conflicts of interest.

\noindent
{\bf Acknowledgments.} \\
The first author was supported by JSPS KAKENHI, Grant Number JP25K17279. 
The second author was supported by JSPS KAKENHI, Grant Number JP24K16946.

\appendix
\def\thesection{\Alph{section}}

\section{Proof of Proposition \ref{prop:stab}}\label{sec:pf_prop}
In this section, we provide the proof of Proposition \ref{prop:stab}.
\begin{proof}[Proof of Proposition \ref{prop:stab}]
Let $n$, $p$, and $q$ be as \eqref{sta:expo}.
We first focus on the solvability of \eqref{eq:f-NS}.
Let us consider the map
\begin{align}
    \mathcal{S}[u](t)
    :=
    e^{t\Delta}a
    +
    \int_0^t e^{(t-\tau)\Delta} \mathbb{P} f(\tau) d\tau
    -
    \int_0^t e^{(t-\tau)\Delta} \mathbb{P} \div\sp{u \otimes u}(\tau) d\tau.
\end{align}
It follows from Lemmas \ref{lemm:linear} and \ref{lemm:nonlin-ndim} that 
\begin{align}
    \n{\mathcal{S}[u]}_{\widetilde{L^{\infty}}(0,\infty;\dB_{p,q}^{{n}/{p}-1})}
    \leq{}&
    C_0\n{a}_{\dB_{p,q}^{{n}/{p}-1}}
    +
    C_0\n{f}_{\widetilde{L^{\infty}}(0,\infty;\dB_{p,q}^{{n}/{p}-3})}\\
    &
    +
    C_0
    \n{u}_{\widetilde{L^{\infty}}(0,\infty;\dB_{p,q}^{{n}/{p}-1})}^2,\\
    \n{\mathcal{S}[u_1]-\mathcal{S}[u_2]}_{\widetilde{L^{\infty}}(0,\infty;\dB_{p,q}^{{n}/{p}-1})}
    \leq{}&
    C_0
    \n{u_1}_{\widetilde{L^{\infty}}(0,\infty;\dB_{p,q}^{{n}/{p}-1})}
    \n{u_1-u_2}_{\widetilde{L^{\infty}}(0,\infty;\dB_{p,q}^{{n}/{p}-1})}
    \\
    &
    +
    C_0
    \n{u_2}_{\widetilde{L^{\infty}}(0,\infty;\dB_{p,q}^{{n}/{p}-1})}
    \n{u_1-u_2}_{\widetilde{L^{\infty}}(0,\infty;\dB_{p,q}^{{n}/{p}-1})}
\end{align}
for all $u,u_1,u_2 \in \widetilde{C}([0,\infty);\dB_{p,q}^{{n}/{p}-1}(\mathbb{R}^n))$ with some positive constant $C_0=C_0(n,p,q)$.
Hence, by the contraction mapping principle, if $a$ and $f$ satisfies 
\begin{align}
    \n{a}_{\dB_{p,q}^{{n}/{p}-1}}
    +
    \n{f}_{\widetilde{L^{\infty}}(0,\infty;\dB_{p,q}^{{n}/{p}-3})}
    \leq
    \frac{1}{8C_0^2},
\end{align}
there exists a solution $u \in \widetilde{C}([0,\infty);\dB_{p,q}^{{n}/{p}-1}(\mathbb{R}^n))$ to \eqref{eq:f-NS} satisfying
\begin{align}
    \n{u}_{\widetilde{L^{\infty}}(0,\infty;\dB_{p,q}^{{n}/{p}-1})}
    \leq
    2C_0
    \n{a}_{\dB_{p,q}^{{n}/{p}-1}}
    +
    2C_0
    \n{f}_{\widetilde{L^{\infty}}(0,\infty;\dB_{p,q}^{{n}/{p}-3})}.
\end{align}

Next, we prove $u(t)$ converges to zero in $\dB_{p,q}^{{n}/{p}-1}(\mathbb{R}^n)$ as $t \to \infty$.
Let $t>T>0$. 
Then, since 
\begin{align}
    u(t)
    = {}&
    e^{(t-T)\Delta}u(T)
    +
    \int_T^t e^{(t-\tau)\Delta} \mathbb{P} f(\tau) d\tau
    -
    \int_T^t e^{(t-\tau)\Delta} \mathbb{P} \div\sp{u \otimes u}(\tau) d\tau,
\end{align}
we have
\begin{align}
    &\n{u}_{\widetilde{L^{\infty}}(t,\infty;\dB_{p,q}^{n/p-1})}\\
    &\quad\leq
    \n{e^{(\tau-T)\Delta}u(T)}_{\widetilde{L^{\infty}_{\tau}}(t,\infty;\dB_{p,q}^{n/p-1})}
    +
    C_0\n{f}_{\widetilde{L^{\infty}}(T,\infty;\dB_{p,q}^{n/p-3})}
    +
    C_0
    \n{u}^2_{\widetilde{L^{\infty}}(0,\infty;\dB_{p,q}^{n/p-1})}\\
    &\quad\leq{}
    C
    \Mp{\sum_{j \in \mathbb{Z}}\sp{e^{-c(t-T)2^{2j}}2^{\sp{n/p-1}j}\n{\Delta_ju(T)}_{L^p}}^q}^{1/q}\\
    &\qquad
    +
    C_0
    \Mp{\sum_{j \in \mathbb{Z}}\sp{2^{\sp{n/p-1}j}\n{\Delta_j f}_{L^{\infty}(T,\infty;L^p)}}^q}^{1/q}
    +
    \frac{1}{2}
    \n{u}_{\widetilde{L^{\infty}}(T,\infty;\dB_{p,q}^{n/p-1})}.
\end{align}
It follows from the dominated convergence theorem that 
\begin{align}
    \n{e^{(\tau-T)\Delta}u(T)}_{\widetilde{L^{\infty}_{\tau}}(t,\infty;\dB_{p,q}^{n/p-1})}
    =
    \Mp{
    \sum_{j \in \mathbb{Z}}
    \sp{
    e^{-c(t-T)2^{2j}}
    2^{\sp{n/p-1}j}
    \n{\Delta_j a}_{L^p}
    }^q
    }^{1/q}
    \to 
    0
\end{align}
as $t \to \infty$.
Thus, we have
\begin{align}
    \limsup_{ t \to \infty } 
    \n{u}_{\widetilde{L^{\infty}}(t,\infty;\dB_{p,q}^{n/p-1})}
    \leq{}
    C_0
    \n{f}_{\widetilde{L^{\infty}}(T,\infty;\dB_{p,q}^{n/p-3})}
    +
    \frac{1}{2}
    \n{u}_{\widetilde{L^{\infty}}(T,\infty;\dB_{p,q}^{n/p-1})}.
\end{align}
Letting $T \to \infty$, we see that
\begin{align}
    \limsup_{ t \to \infty } 
    \n{u}_{\widetilde{L^{\infty}}(t,\infty;\dB_{p,q}^{n/p-1})}
    \leq{}
    2C_0
    \limsup_{T \to \infty}
    \n{f}_{\widetilde{L^{\infty}}(T,\infty;\dB_{p,q}^{n/p-3})}.
\end{align}
Since it holds for every $j \in \mathbb{Z}$ that
\begin{gather}
    2^{(\frac{n}{p}-1)j}
    \n{\Delta_jf(t)}_{L^p}
    \leq 
    \n{f(t)}_{\dB_{p,q}^{n/p-1}}
    \to 0 
    \quad 
    {\rm as }
    \quad 
    t \to \infty,\\
    \lim_{t \to \infty}
    \n{\Delta_jf}_{L^{\infty}(t,\infty;L^p)}
    =
    \lim_{t \to \infty}
    \n{\Delta_jf(t)}_{L^p}
    =0,
\end{gather}
we see by the reverse Fatou lemma that
\begin{align}
    &
    \limsup_{t\to\infty}\n{u(t)}_{\dB_{p,q}^{n/p-1}}
    \leq
    \limsup_{ t \to \infty } 
    \n{u}_{\widetilde{L^{\infty}}(t,\infty;\dB_{p,q}^{n/p-1})}
    \\
    &\quad
    \leq{}
    2C_0
    \limsup_{T \to \infty}
    \n{f}_{\widetilde{L^{\infty}}(T,\infty;\dB_{p,q}^{n/p-3})}\\
    &\quad
    \leq{}
    2C_0
    \Mp{\sum_{j \in \mathbb{Z}}\sp{2^{\sp{n/p-1}j}\limsup_{T \to \infty}\n{\Delta_jf}_{L^{\infty}(T,\infty;L^p)}}^q}^{1/q}\\
    &\quad
    ={}
    2C_0
    \Mp{\sum_{j \in \mathbb{Z}}\sp{2^{(\sp{n/p-1}j}\lim_{t \to \infty}\n{\Delta_jf(t)}_{L^p}}^q}^{1/q}
    =0.
 \end{align}
Thus, we complete the proof.
\end{proof}

\section{Non-oscillating solutions}\label{sec:non-osci}
In this section, we construct an example where the norm of the external force converges to zero, while the norm of the solution, though not necessarily oscillatory as $t \to \infty$, remains uniformly bounded below by a positive constant independent of $t$. This result is weaker than that of Theorem~\ref{thm:nD} in the sense that it still allows for the possibility that the solution converges to some flow as time tends to infinity. Nevertheless, it has the advantages that the decay rate of the external force can be explicitly determined and the proof is more concise.
\begin{thm} \label{prop:nD-non-osci}
    Let $n \geq 3$.
    Let $p$ and $q$ satisfy either \textup{(1)} or \textup{(2)}:
    \begin{enumerate}
        \item [\textup{(1)}]  $n<p\leq \infty$,\quad $1 \leq q < \infty$
        \item [\textup{(2)}]  $p=n$,\quad $2< q < \infty$.
    \end{enumerate}
    Then, there exist positive constant 
    $\eta=\eta(p,q)$
    and 
    $C=C(p,q)$ 
    such that
    for any $0 < \varepsilon \leq \eta$, there exists a solenoidal external force   
    $f_{\varepsilon} \in \widetilde{C}([t_0,\infty);\dB_{p,q}^{n/p-3}(\mathbb{R}^n))$
    such that 
    \begin{align}\label{decay:g}
        \begin{split}
        &
        \n{f_{\varepsilon}}_{\widetilde{L^{\infty}}(t_0,\infty;\dB_{p,q}^{n/p-3})}
        \leq
        C\eta\varepsilon, 
        \\
        &
        \n{f_{\varepsilon}(t)}_{\dB_{p,q}^{n/p-3}}
        \leq 
        \begin{cases}
        C\eta\varepsilon
        (1+\varepsilon^{-1/(1-n/p)}t)^{-(1-n/p)},\quad &n<p\leq \infty,\\
        C\eta\big(\log(e^{1/\varepsilon^2}+t)\big)^{-1/2}, &p=n,
        \end{cases}
        \end{split}
    \end{align}
    the forced Navier--Stokes equations \eqref{eq:f-NS} with $a=0$ and $f=f_{\varepsilon}$
    possesses a global solution 
    $u_{\varepsilon} \in C((0,\infty);L^{n,\infty}(\mathbb{R}^n))$
    satisfying the estimate
    \begin{align}\label{lim:w}
        \liminf_{t \to \infty}
        \n{u_{\varepsilon}(t)}_{\dB_{\infty,\infty}^{-1}}
        \geq c\eta^2,
        \quad
        \n{u_{\varepsilon}}_{L^{\infty}(0,\infty;L^{n,\infty})}
        \leq
        C\eta.
    \end{align}
\end{thm}
To prove this Proposition, we should show the decay property of the Duhamel integral as follows:
\begin{lemm}\label{lemm:converge}
    Let $n \geq 3$, $1\leq p\leq \infty$, and $1\leq q <\infty$.
    Suppose that $f\in \widetilde{C}([0,\infty); \dB^{n/p-3}_{p,q}(\mathbb{R}^n))$ satisfies
    \begin{align}
    \lim_{t\to\infty}\|f(t)\|_{\dB^{n/p-3}_{p,q}}=0.
    \end{align}
    Then it holds
    \begin{align}
    \lim_{t\to\infty}\left\| \int_0^t e^{(t-\tau)\Delta}\mathbb{P}f(\tau)d\tau \right\|_{\dB^{n/p-1}_{p,q}}=0.\\
    \end{align}
\end{lemm}

\begin{proof}[Proof of Lemma \ref{lemm:converge}]
We see that
\begin{align}
&\left\|\Delta_j \int_0^t e^{(t-\tau)\Delta}\mathbb{P}f(\tau)d\tau\right\|_{L^p}\\
&\quad\leq C\int_0^t e^{-c2^{2j}(t-\tau)}\|\Delta_j f(\tau)\|_{L^p}d\tau\\
&\quad\leq C\sup_{\tau>0}\|\Delta_j f(\tau)\|_{L^p}\int_0^{t/2} e^{-c2^{2j}(t-\tau)}d\tau\\
&\qquad +C\sup_{t/2\leq \tau\leq t}\|\Delta_j f(\tau)\|_{L^p}\int_{t/2}^{t} e^{-c2^{2j}(t-\tau)}d\tau\\
&\quad\leq Ce^{-c2^{2j-1}t}2^{-2j}\sup_{\tau>0}\|\Delta_j f(\tau)\|_{L^p}
+C2^{-2j}\sup_{t/2\leq \tau\leq t}\|\Delta_j f(\tau)\|_{L^p}.
\end{align}
Hence we obtain the estimate as
\begin{align}
\left\|\int_0^t e^{(t-\tau)\Delta}\mathbb{P}f(\tau)d\tau\right\|_{\dB^{n/p-1}_{p,q}}
&\leq C\left\{\sum_{j\in\mathbb{Z}}\left(e^{-c2^{2j-1}t}2^{\sp{n/p-3}j}\|\Delta_j f\|_{L^\infty(0,\infty;L^p)}\right)^q\right\}^{1/q} \\
&\quad+ C\left\{\sum_{j\in\mathbb{Z}}\left(2^{\sp{n/p-3}j}\|\Delta_j f\|_{L^\infty(t/2, t;L^p)}\right)^q\right\}^{1/q}.
\end{align}
Hence, the dominated convergence theorem completes the proof.
\end{proof}

\begin{proof}[Proof of Theorem \ref{prop:nD-non-osci}]
The central task is to derive the lower bound estimate in \eqref{lim:w}.
To obtain this, we follow the standard method for the ill-posedness and consider the decomposition of the solution:
\begin{align}
    u_{\varepsilon}
    =
    u_{\varepsilon}^{(1)}
    +
    u_{\varepsilon}^{(2)}
    +
    \widetilde{u}_{\varepsilon},
\end{align}
where $u_{\varepsilon}^{(1)}$ and $u_{\varepsilon}^{(2)}$ are the first and second order iterations respectively defined by 
\begin{align}\label{eq:nD:w}
    \begin{cases}
        \partial_t u_{\varepsilon}^{(1)} - \Delta u_{\varepsilon}^{(1)} 
        =f_{\varepsilon}, \\
        \div u_{\varepsilon}^{(1)}=0,  \\
        u_{\varepsilon}^{(1)}(0,x)=0,
    \end{cases}
    \begin{cases}
        \partial_t u_{\varepsilon}^{(2)} - \Delta u_{\varepsilon}^{(2)} 
        + 
        \mathbb{P}\div \sp{u_{\varepsilon}^{(1)} \otimes u_{\varepsilon}^{(1)}}
        =0, \\
        \div u_{\varepsilon}^{(2)}=0, \\
        u_{\varepsilon}^{(2)}(0,x)=0,
    \end{cases}
\end{align}
and the remainder part $\widetilde{u}_{\varepsilon}$ should solve the following system:
\begin{align}
    \begin{cases}
        \begin{aligned}
        \partial_t \widetilde{u}_{\varepsilon} - \Delta \widetilde{u}_{\varepsilon}
        + 
        &
        \mathbb{P}\div 
        \left(
        u_{\varepsilon}^{(1)} \otimes u_{\varepsilon}^{(2)} 
        + 
        u_{\varepsilon}^{(2)} \otimes u_{\varepsilon}^{(1)}
        + 
        u_{\varepsilon}^{(2)} \otimes u_{\varepsilon}^{(2)}
        \right.\\
        &\qquad\qquad
        \left.
        +
        u_{\varepsilon}^{(1)} \otimes \widetilde{u}_{\varepsilon} 
        + 
        u_{\varepsilon}^{(2)} \otimes \widetilde{u}_{\varepsilon}
        + 
        \right.\\
        &\qquad\qquad
        \left.
        +
        \widetilde{u}_{\varepsilon} \otimes u_{\varepsilon}^{(1)}
        + 
        \widetilde{u}_{\varepsilon} \otimes u_{\varepsilon}^{(2)}
        + 
        \widetilde{u}_{\varepsilon} \otimes \widetilde{u}_{\varepsilon}
        \right)
        =0,
        \end{aligned}
        \\
        \div \widetilde{u}_{\varepsilon}=0, \\
        \widetilde{u}_{\varepsilon}(0,x)=0.
    \end{cases}
\end{align}
Assuming for a while that there exists an external force $f_{\varepsilon} \in  \widetilde{C}([0,\infty);\dB_{p,q}^{n/p-3}(\mathbb{R}^n))$ that satisfies \eqref{decay:g}, 
\begin{align}\label{key:w^1}
    \n{u_{\varepsilon}^{(1)}}_{L^{\infty}(0,\infty;L^{n,\infty})}
    \leq 
    C\eta,\quad
    \lim_{t \to \infty}
    \n{u_{\varepsilon}^{(1)}(t)}_{\dB_{p,q}^{n/p-1}}=0,
\end{align}
and
\begin{align}\label{key:w^2}
    \liminf_{t \to \infty}
    \n{u_{\varepsilon}^{(2)}(t)}_{\dB_{\infty,\infty}^{-1}} \geq c\eta^2.
\end{align}
Then, from Lemma \ref{lemm:Yamazaki} it follows that 
\begin{align}\label{w^2:L^n}
    \n{u_{\varepsilon}^{(2)}}_{L^{\infty}(0,\infty;L^{n,\infty})}
    \leq
    C\n{u_{\varepsilon}^{(1)}}_{L^{\infty}(0,\infty;L^{n,\infty})}^2
    \leq 
    C\eta^2.
\end{align}
Using the contraction mapping principle via Lemma \ref{lemm:Yamazaki}, we see that we may construct the solution $\widetilde{u}_{\varepsilon} \in L^{\infty}(0,\infty;L^{n,\infty}(\mathbb{R}^n))$ satisfying
\begin{align}\label{tw:L^n}
    \n{\widetilde{u}_{\varepsilon}}_{L^{\infty}(0,\infty;L^{n,\infty})}
    \leq
    C\eta^3
\end{align}
provided that $\eta$ is sufficiently small.
We then see by \eqref{key:w^1}, \eqref{key:w^2}, and \eqref{w^2:L^n}, \eqref{tw:L^n}, and the embedding $\dB_{\infty,\infty}^{-1}(\mathbb{R}^n) \hookrightarrow \dB_{p,q}^{n/p-1}(\mathbb{R}^n) \cap L^{n,\infty}(\mathbb{R}^n)$ that the solution $u_{\varepsilon}=u_{\varepsilon}^{(1)}+u_{\varepsilon}^{(2)}
+\widetilde{u}_{\varepsilon}$ satisfies
\begin{align}
    \liminf_{t \to \infty}
    \n{u_{\varepsilon}(t)}_{\dB_{\infty,\infty}^{-1}}
    \geq{}&
    \liminf_{t \to \infty}\n{u_{\varepsilon}^{(2)}(t)}_{\dB_{\infty,\infty}^{-1}}\\
    &
    -
    C
    \lim_{t \to \infty}
    \n{u_{\varepsilon}^{(1)}(t)}_{\dB_{p,q}^{n/p-1}}
    -
    C
    \n{\widetilde{u}_{\varepsilon}}_{L^{\infty}(0,\infty;L^{n,\infty})}\\
    \geq{}&
    c\eta^2
    -
    C
    \eta^3
\end{align}
and 
\begin{align}
    &
    \n{u_{\varepsilon}(t)}_{L^{\infty}(0,\infty;L^{n,\infty})}
    \\
    &
    \quad
    \leq{}
    \n{u_{\varepsilon}^{(1)}}_{L^{\infty}(0,\infty;L^{n,\infty})}
    +
    \n{u_{\varepsilon}^{(2)}}_{L^{\infty}(0,\infty;L^{n,\infty})}
    +
    \n{\widetilde{u}_{\varepsilon}}_{L^{\infty}(0,\infty;L^{n,\infty})}
    \\
    &\quad
    \leq{}
    C\eta
    +
    C\eta^2
    +
    C
    \eta^3,
\end{align}
which completes the proof.

Hence, it suffices to construct an external force that guarantees 
\eqref{key:w^1} and \eqref{key:w^2}. 
To this end, we first construct the first approximation 
$u^{(1)}_{\varepsilon}$, and only afterwards define the external force 
$f_\varepsilon$. 
The reason is that the time variable $t$ should be incorporated into 
the parameter playing a role analogous to $\delta$ used in 
Section~\ref{sec:nD}, so that the nonlinear estimate breaks down as 
$t \to \infty$. 
In this setting, if we attempted to construct the external force 
$f_\varepsilon$ first, then in deriving the first approximation 
$u^{(1)}_\varepsilon$ from $f_\varepsilon$---that is, when taking the 
Duhamel integral of $f_\varepsilon$---the factor responsible for the 
breakdown of the nonlinear estimate (namely $g_{\delta,\eta}$ in 
Section~\ref{sec:nD}) would be significantly distorted, which would 
render the computation much more complicated.

First, we consider the case of $n<p\leq\infty$ and $1\leq q<\infty$.
Let $\Psi(x)$ be as that of \eqref{eq:Psi},
and define
\begin{align}
    &
    \begin{aligned}
    u_{\varepsilon}^{(1)}(t,x)
    :={}&
    \eta
    \Psi(x)
    \cos\sp{\varepsilon^{-1/(1-n/p)}(1+\varepsilon^{-1/(1-n/p)}t)x_1}\\
    &-
    \eta
    e^{t\Delta}\lp{\Psi(x)
    \cos\sp{\varepsilon^{-1/(1-n/p)}x_1}}\\
    =:{}&
    u_{\varepsilon}^{(1;1)}(t,x)+u_{\varepsilon}^{(1;2)}(t,x),
    \end{aligned}
    \\
    &
    \begin{aligned}
    u_{\varepsilon}^{(2)}(t,x)
    :={}&
    -
    \int_{0}^t
    e^{(t-\tau)\Delta}
    \mathbb{P}\div\sp{u_{\varepsilon}^{(1)}(\tau) \otimes u_{\varepsilon}^{(1)}(\tau)}(x)
    d\tau,
    \end{aligned}
\end{align}
and
\begin{align}
    f_{\varepsilon}(t,x)
    :={}&
    \sp{
    \partial_tu_{\varepsilon}^{(1)}
    -
    \Delta u_{\varepsilon}^{(1)}
    }(t,x)\\
    ={}&
    -\eta
    \Psi(x)
    \varepsilon^{-2/(1-n/p)} x_1
    \sin\sp{\varepsilon^{-1/(1-n/p)}(1+\varepsilon^{-1/(1-n/p)}t)x_1}\\
    &
    -\eta\varepsilon
    \lp{
    \Psi(x)
    \cos\sp{\varepsilon^{-1/(1-n/p)}(1+\varepsilon^{-1/(1-n/p)}t)x_1}
    }.
\end{align}
Note that the $u_{\varepsilon}^{(1)}$ and $u_{\varepsilon}^{(2)}$ are the first and second iteration of \eqref{eq:nD:w}, respectively.
For the estimates of these functions, we firstly see that 
\begin{align}\label{w^1-L^n}
    \n{u_{\varepsilon}^{(1)}}_{L^{\infty}(0,\infty;L^{n,\infty})}
    \leq
    \n{u_{\varepsilon}^{(1)}}_{L^{\infty}(0,\infty;L^{n})}
    \leq 
    C\eta.
\end{align}
By Lemma \ref{lemm:cos},
it follows 
\begin{align}
    &
    \n{u_{\varepsilon}^{(1)}}_{\widetilde{L^{\infty}}(0,\infty;\dB_{p,q}^{n/p-1})}
    \leq
    C\eta\varepsilon
\end{align}
and 
\begin{align}
    &
    \n{u_{\varepsilon}^{(1;1)}(t)}_{\dB_{p,q}^{n/p-1}}
    \leq
    C\eta\varepsilon(1+\varepsilon^{-1/(1-n/p)}t)^{-(1-n/p)}
    \to 0, 
    \\
    &
    \n{u_{\varepsilon}^{(1;2)}(t)}_{\dB_{p,q}^{n/p-1}}
    \to 
    0,
    \quad
    \n{u_{\varepsilon}^{(1;2)}(t)}_{L^n}
    \to 
    0
\end{align}
as $t \to \infty$.
Moreover, since $\supp{\widehat{f_{\varepsilon}(t)}} \subset \supp{\widehat{u^{(1;1)}_\varepsilon(t)}}$,
we see that $f_{\varepsilon} \in \widetilde{C}([0,\infty);\dB_{p,q}^s(\mathbb{R}^n))$ for any $s \in \mathbb{R}$ and 
\begin{align}
    \n{f_{\varepsilon}}_{\widetilde{L^{\infty}}(0,\infty;\dB_{p,q}^{n/p-1})}
    \leq
    C\eta\varepsilon
\end{align}
and 
\begin{align}
    \n{f_{\varepsilon}(t)}_{\dB_{p,q}^{n/p-1}}
    \leq
    C\eta\varepsilon\sp{1+\varepsilon^{-1/(1-n/p)}t}^{-(1-n/p)}
    \to 
    0
\end{align}
as $t \to \infty$.
We next focus on the estimates of $u_{\varepsilon}^{(2)}$.
It holds 
\begin{align}
    &\div\sp{u_{\varepsilon}^{(1;1)}(\tau,x) \otimes u_{\varepsilon}^{(1;1)}(\tau,x)}
    \\
    &\quad
    ={}
    \eta^2
    \Phi(x)
    -
    \eta^2
    \Phi(x)
    \cos\sp{2\varepsilon^{-1/(1-n/p)}(1+\varepsilon^{-1/(1-n/p)}t)x_1}\\
    &\quad
    =:{}
    \eta^2\Phi(x)+\eta^2\Theta_{\varepsilon}(t,x),
\end{align}
where $\Phi$ is that of \eqref{eq:Phi}.
Thus, $u_{\varepsilon}^{(2)}$ can be written as
\begin{align}
    u_{\varepsilon}^{(2)}(t)
    ={}&
    -
    \sum_{1 \leq k,\ell \leq 2}
    \int_{0}^t e^{(t-\tau)\Delta}\mathbb{P}\div\sp{u_{\varepsilon}^{(1;k)}(\tau)\otimes u_{\varepsilon}^{(1;\ell)}(\tau)}d\tau\\
    ={}& 
    -
    \eta^2(-\Delta)^{-1}\mathbb{P}\Phi
    +
    \eta^2(-\Delta)^{-1}e^{t\Delta}\mathbb{P}\Phi
    -
    \eta^2
    \int_{0}^t
    e^{(t-\tau)\Delta}
    \mathbb{P}\Theta_{\varepsilon}(\tau)d\tau\\
    &-
    \sum_{\substack{1 \leq k,\ell \leq 2 \\ (k,\ell) \neq (1,1)}}
    \int_{0}^t e^{(t-\tau)\Delta}\mathbb{P}\div\sp{u_{\varepsilon}^{(1;k)}(\tau)\otimes u_{\varepsilon}^{(1;\ell)}(\tau)}d\tau.
\end{align}
We see that
\begin{align}
    &\n{\Theta_{\varepsilon}(t)}_{\dB^{n/p-3}_{p,q}}
    \leq 
     C\eta\varepsilon\sp{1+\varepsilon^{-1/(1-n/p)}t}^{-(3-n/p)}
    \to 
    0
\end{align}
as $t\to\infty$. Moreover, since 
\begin{align}
    \n{u_{\varepsilon}^{(1;2)}(t)}_{L^n}
    \leq C\eta t^{-(n-1)/2}\n{\Psi}_{L^{1}}
    \leq C\eta t^{-(n-1)/2}
\end{align}
by the effect of the heat kernel, it holds for every $k\in \{1,2\}$ and $\max\Mp{2,n/2}\leq r<\infty$ that
\begin{align}
    \n{\div(u_{\varepsilon}^{(1;k)}(t)\otimes u_{\varepsilon}^{(1;2)}(t))}_{\dB^{n/p-3}_{p,r}}
    &\leq C \n{u_{\varepsilon}^{(1;k)}(t)\otimes u_{\varepsilon}^{(1;2)}(t)}_{\dB^{0}_{n/2,r}}\\
    &\leq
    C\n{u_{\varepsilon}^{(1;k)}(t)\otimes u_{\varepsilon}^{(1;2)}(t)}_{L^{n/2}}\\
    &\leq
    C\n{u_{\varepsilon}^{(1;k)}(t)}_{L^n}\n{u_{\varepsilon}^{(1;2)}(t)}_{L^n}\\
    &\to 0
\end{align}
as $t\to\infty$,
where we have used the continuous embedding $L^{n/2}(\mathbb{R}^n) \simeq \dot{F}^0_{n/2,2}(\mathbb{R}^n)
\hookrightarrow\dB^0_{n/2,r}(\mathbb{R}^n)$
($\dot{F}^s_{p,q}(\mathbb{R}^n)$ denotes the Triebel--Lizorkin space; see \cite{Saw-18} for details).
Therefore, by the density argument and Lemma \ref{lemm:converge}, we have
\begin{align}
    &
    \lim_{t \to \infty}
    \n{(-\Delta)^{-1}e^{t\Delta}\mathbb{P}\Phi}_{\dB_{p,q}^{n/p-1}}
    =0,\label{lim:psi-1}\\
    &
    \lim_{t \to \infty}
    \n{\int_{0}^t
    e^{(t-\tau)\Delta}
    \mathbb{P}\Theta_{\varepsilon}(\tau)d\tau}_{\dB_{p,q}^{n/p-1}}
    =0 
    ,\label{lim:psi-2}\\
    &
    \lim_{t \to \infty}
    \sum_{\substack{1 \leq k,\ell \leq 2 \\ (k,\ell) \neq (1,1)}}
    \n{ \int_{0}^t e^{(t-\tau)\Delta}\mathbb{P}\div\sp{u_{\varepsilon}^{(1;k)}(\tau)\otimes u_{\varepsilon}^{(1;\ell)}(\tau)}d\tau}_{\dB_{p,r}^{n/p-1}}
    =0.\label{lim:psi-3}
\end{align}
Hence, we obtain the lower bound as
\begin{align}
    \liminf_{t\to \infty}
    \n{u_{\varepsilon}^{(2)}(t)}_{\dB_{\infty,\infty}^{-1}}
    \geq \eta^2\n{(-\Delta)^{-1}\mathbb{P}\Phi}_{\dB^{-1}_{\infty,\infty}}
    \geq c\eta^2.
\end{align}
This completes the proof
in the case of $n<p\leq\infty$ and $1\leq q<\infty$.

We next focus on the case of $p=n$ and $2<q<\infty$.
Let $\gamma\in BC^\infty(\mathbb{R})$ satisfy
\begin{align*}
{\bf 1}_{\Mp{|t|\geq 2/3}}(t) \leq \gamma(t) \leq {\bf 1}_{\Mp{|t|\geq 1/3}}(t),
\qquad 
\sup_{0\leq t<\infty}|\gamma'(t)|<\infty,\quad
\end{align*}
Then we set 
\begin{align*}
u^{(1)}_{\varepsilon}(t,x)
:=
\frac{\eta}{\sqrt{\log(e^{1/\varepsilon^2}+t)}}
\Psi(x)
\sum_{k=0}^{\lfloor t \rfloor}
\frac{\gamma_k(t)}{\sqrt{k+1}}
\cos(\alpha(k)x_1),
\end{align*}
where 
$\lfloor t \rfloor$ denotes the floor function of $t$, 
$\alpha(k)=2^{(k+10)^2}$, and
$\gamma_k(t):=\gamma(t-k)$.
We see that $\div u^{(1)}_\varepsilon=0$, $u^{(1)}_\varepsilon(0,x)=0$, and
\begin{align}
\n{u^{(1)}_\varepsilon(t)}_{\dB^0_{n,q}}
&
\leq \frac{C\eta}{\sqrt{\log(e^{1/\varepsilon^2}+t)}}
\left(1+\sum_{k=1}^{\lfloor t \rfloor}k^{-q/2}\right)^{1/q}
\\
&\leq 
\begin{cases}
C\eta, &q=2,\\
\dfrac{C\eta}{\sqrt{\log(e^{1/\varepsilon^2}+t)}}, &2<q\leq\infty.
\end{cases}
\end{align}
Since $\dB^0_{n,2}(\mathbb{R}^n)\hookrightarrow
L^n(\mathbb{R}^n)\hookrightarrow
L^{n,\infty}(\mathbb{R}^n)$, there hold  $u^{(1)}_\varepsilon\in C((0,\infty); L^{n,\infty}(\mathbb{R}^n))$ and
\begin{align}
\lim_{t\to\infty}\n{u^{(1)}_\varepsilon(t)}_{\dB^0_{n,q}}=0
\end{align}
for any $2<q<\infty$.
Moreover, $u^{(1)}_\varepsilon$ is smooth enough with respect to $t$ and $x$. 
Indeed, the smoothness with respect to $x$ is clear.
On the other hand, we see that
\begin{align*}
u^{(1)}_\varepsilon(t+h,x)
=
\frac{\eta}{\sqrt{\log(e^{1/\varepsilon^2}+t+h)}}
\Psi(x)
\sum_{k=0}^{\lfloor t \rfloor}
\frac{\gamma_k(t+h)}{\sqrt{k+1}}
\cos(\alpha(k)x_1)
\end{align*}
for $h\in\mathbb{R}$ such that $|h|\ll 1$, which is valid even when $t=\lfloor t \rfloor$ since $\gamma(h)=0$.
Hence $u_1$ is differentiable by $t$ in $[0,\infty)$, and
\begin{align}
\partial_t u^{(1)}_\varepsilon(t,x)
&=
-\frac{\eta}{2(e^{1/\varepsilon^2}+t)\Mp{\log(e^{1/\varepsilon^2}+t)}^{3/2}}
\Psi(x)
\sum_{k=0}^{\lfloor t \rfloor}
\frac{\gamma_k(t)}{\sqrt{k+1}}
\cos(\alpha(k)x_1)\\
&\qquad +
\frac{\eta}{\sqrt{\log(e^{1/\varepsilon^2}+t)}}
\Psi(x)
\frac{(\partial_t\gamma)(t-\lfloor t \rfloor)}{\sqrt{\lfloor t \rfloor+1}}
\cos(\alpha(\lfloor t \rfloor)x_1).
\end{align}
Therefore, we can define $f_\varepsilon$ and $u^{(2)}_\varepsilon$ by \eqref{eq:nD:w}.
Since $\partial_t\gamma\in L^\infty(0,\infty)$, $(-\Delta)^{-1}f_\varepsilon$ has almost same smallness and decaying property as $u^{(1)}_\varepsilon$, that is, $f_\varepsilon\in C([0,\infty); \dB^{-2}_{n,q}(\mathbb{R}^n))$ and
\begin{align}
&\sup_{0\leq t<\infty}\|f_\varepsilon(t)\|_{\dB^{-2}_{n,q}}
\leq \frac{C\eta}{\sqrt{\log e^{1/\varepsilon^2} }}=C\eta\varepsilon,\\
&\|f_\varepsilon(t)\|_{\dB^{-2}_{n,q}}
\leq
\frac{C\eta}{\sqrt{\log (e^{1/\varepsilon^2}+t) }}
\to\infty,\qquad {\rm as}\ t\to\infty
\end{align}
for any $2<q<\infty$. 

Now we focus on $u^{(2)}_\varepsilon$. 
We see that
\begin{align*}
&\div (u^{(1)}_\varepsilon\otimes u^{(1)}_\varepsilon)(t,x)\\
\quad &= \frac{\eta^2}{\log(e^{1/\varepsilon^2}+t)} \Phi(x) \sum_{k,\ell=0}^{\lfloor t \rfloor} \frac{\gamma_k(t)\gamma_\ell(t)}{\sqrt{(k+1)(l+1)}}\cos(\alpha(k)x_1)\cos(\alpha(\ell)x_1) \\
\quad &
=: \eta^2\Phi(x)\left({I_1}(t)+{I_2}(t,x)\right), 
\end{align*}
where
\begin{align*}
&{I_1}(t):=
\frac{1}{\log(e^{1/\varepsilon^2}+t)}\sum_{k=0}^{\lfloor t \rfloor} 
\frac{\gamma_k^2(t)}{k+1}
,\\
&
\begin{aligned}
{I_2}(t,x)
:={}
&
\frac{1}{\log(e^{1/\varepsilon^2}+t)}
\sum_{k=0}^{\lfloor t \rfloor} \frac{\gamma_k^2(t)}{k+1}\cos(2\alpha(k)x_1),\\
&
+\frac{1}{\log(e^{1/\varepsilon^2}+t)}
\left[
\sum_{\substack{0\leq k,\ell \leq \lfloor t \rfloor\\k\neq \ell}}
\frac{\gamma_k(t)\gamma_\ell(t)}{\sqrt{(k+1)(l+1)}}
\left\{
 \cos(A_{k,\ell}x_1)+\cos(B_{k,\ell}x_1)
 \right\}
\right]
\end{aligned}
\end{align*}
with $A_{k,\ell}:=\alpha(k)+\alpha(\ell)$,
$B_{k,\ell}:=\alpha(k)-\alpha(\ell)$.
Let us investigate the lower bound of $u^{(2)}_\varepsilon$ in $\dot B^{-1}_{\infty, \infty}$. 
Since $2\alpha(k)$, 
$A_{k,\ell}$, 
and 
$|B_{k,\ell}|$ 
are large enough, 
we see for any $j\leq 2$ that
\begin{align*}
\Delta_j 
\lp{\mathbb{P}\div (u^{(1)}_\varepsilon\otimes u^{(1)}_\varepsilon)(t)}
= 
\eta^2
{I_1}(t)
\Delta_j
\mathbb{P}
\Phi,
\end{align*}
which means
\begin{align*}
&\|u^{(2)}_\varepsilon(t)\|_{\dB^{-1}_{\infty,\infty}}\\
&\quad\geq
\eta^2
\n{
\int_0^t e^{(t-\tau)\Delta}
{I_1}(\tau)
\mathbb{P}
\Phi d\tau
}_{\dot B^{-1}_{\infty,\infty}}\\
&\quad\geq
\eta^2
\n{
\int_0^t e^{(t-\tau)\Delta}
\mathbb{P}
\Phi d\tau
}_{\dot B^{-1}_{\infty,\infty}}
-
\eta^2
\n{
\int_0^t e^{(t-\tau)\Delta}
(1-{I_1}(\tau))
\mathbb{P}
\Phi d\tau
}_{\dot B^{0}_{n,r}}\\
&\quad:=\eta^2{J_1}(t)-\eta^2{J_2}(t)
\end{align*}
with any $1\leq r\leq \infty$.
By the same discussion as in the case $n<p\leq\infty$, we have the lower estimate
\begin{align*}
\liminf_{t \to \infty} {J_1}(t) \geq c
\end{align*}
with some constant $c>0$.
On the other hand, we see
\begin{align*}
\n{({I_1}(t)-1)\mathbb{P}\Phi}_{\dB^{-2}_{n,r}}
&\leq C|{I_1}(t)-1| \n{\nabla\psi}^2_{\dB^{n/2-1}_{2,1}}
\to 0,\qquad {\rm as}\ t\to\infty.
\end{align*}
Here we have used the inequality for every $t$ such that
\begin{align*}
\frac{1}{\log(e^{1/\varepsilon^2}+t)}\sum_{k=0}^{\lfloor t \rfloor-1} \frac{1}{k+1}
\leq {I_1}(t)
\leq \frac{1}{\log(e^{1/\varepsilon^2}+t)}\sum_{k=0}^{\lfloor t \rfloor} \frac{1}{k+1},
\end{align*}
and the fact that the left and right hand side of this inequality converge to $1$ as $t\to\infty$.
Therefore, by Lemma \ref{lemm:converge}, we have the decay
\begin{align*}
\lim_{t \to \infty} {J_2}(t) =0.
\end{align*}
We then obtain the lower bound of $u^{(2)}_\varepsilon$ as
\begin{align}
\liminf_{t\to\infty}\|u^{(2)}_\varepsilon(t)\|_{\dB^{-1}_{\infty,\infty}}>c\eta^2
\end{align}
with some constant $c>0$.
Thus, we complete the proof.
\end{proof}

\end{document}